\documentclass[12pt]{amsart}
\usepackage{verbatim}
\usepackage{amsmath}
\usepackage{amsthm}
\usepackage{amsfonts}
\usepackage{amssymb}
\usepackage{longtable}
\newcommand{\C}{\mathbb{C}}
\newcommand{\g}{\mathfrak{g}}
\newcommand{\U}{\mathcal{U}}
\newcommand{\Walg}{\mathcal{W}}

\newcommand{\Irr}{\operatorname{Irr}}

\newcommand{\Orb}{\mathbb{O}}

\newcommand{\gr}{\operatorname{gr}}
\newcommand{\Hom}{\operatorname{Hom}}
\newcommand{\h}{\mathfrak{h}}
\newcommand{\Z}{\mathbb{Z}}

\newcommand{\F}{\mathbb{F}}
\newcommand{\B}{\mathcal{B}}
\newcommand{\Br}{\operatorname{Br}}
\newcommand{\Sh}{\mathsf{Sh}}

\newcommand{\Coh}{\operatorname{Coh}}

\newcommand{\Fl}{\mathcal{F}l}

\newcommand{\slf}{\mathfrak{sl}}
\newcommand{\Q}{\mathbb{Q}}

\newcommand{\A}{\mathcal{A}}
\newcommand{\N}{\mathcal{N}}

\newcommand{\St}{\mathsf{St}}
\newcommand{\Perv}{\operatorname{Perv}}
\newcommand{\Nilp}{\mathcal{N}}
\newcommand{\lf}{\mathfrak{l}}
\newcommand{\Fr}{\operatorname{Fr}}
\newcommand{\p}{\mathfrak{p}}

\newcommand{\m}{\mathfrak{m}}

\newcommand{\Hecke}{\mathcal{H}}

\newcommand{\D}{\mathbb{D}}
\newcommand{\n}{\mathfrak{n}}
\newcommand{\bor}{\mathfrak{b}}
\newcommand{\Xfr}{\mathfrak{X}}
\newcommand{\Tilt}{\mathcal{T}}
\newcommand{\Str}{\mathcal{O}}
\newcommand{\End}{\operatorname{End}}

\newcommand{\Pcal}{\mathcal{P}}
\newcommand{\tw}{\mathsf{tw}}
\newcommand{\ug}{\underline{\mathfrak{g}}}
\newcommand{\Fim}{\operatorname{Fim}}
\newcommand{\pr}{\mathsf{pr}}

\newtheorem{Thm}{Theorem}[section]
\newtheorem{Prop}[Thm]{Proposition}
\newtheorem{Cor}[Thm]{Corollary}
\newtheorem{Lem}[Thm]{Lemma}
\theoremstyle{definition}

\newtheorem{defi}[Thm]{Definition}
\newtheorem{Rem}[Thm]{Remark}

\numberwithin{equation}{section}
\title{Dimensions of modular irreducible representations of semisimple Lie algebras}
\author{Roman Bezrukavnikov and Ivan Losev}
\address{R.B: Department of Mathematics, MIT, Cambridge MA 02139 USA }
\email{bezrukav@math.mit.edu}
\address{I.L.: Department
of Mathematics, Yale University, New Haven CT 06511 USA}
\email{ivan.loseu@yale.edu}
\begin{document}
\begin{abstract}
In this paper we classify and give Kazhdan-Lusztig type character formulas for equivariantly irreducible
representations of Lie algebras of reductive algebraic groups
over a field of large positive characteristic. 
The equivariance is with respect to a group whose connected component is a torus.
Character computation is done in two steps. First, we treat the case of distinguished
$p$-characters: those that are not contained in a proper Levi. Here we essentially
show that the category of equivariant modules we consider is a cell quotient of
an affine parabolic category $\mathcal{O}$. For this, we prove an equivalence between two categorifications of a parabolically induced module over the affine Hecke algebra conjectured by the first named author. For the general nilpotent $p$-character,
we get character formulas by explicitly computing the duality operator on a suitable
equivariant K-group.
\end{abstract}
\maketitle
\markright{DIMENSIONS OF MODULAR IRREDUCIBLE REPRESENTATIONS}
\tableofcontents
\section{Introduction}
The goal of this paper is to obtain character formulas for (equivariantly) irreducible representations
of semisimple Lie algebras over  algebraically closed fields of large enough postive
characteristic. Below we write $G$ for a connected reductive algebraic group over $\C$ and $\g$
for its Lie algebra, both are defined over $\Z$. We pick a prime number $p\gg 0$
and choose an algebraically closed field $\F$ of characteristic $p$. We write
$\g_{\F},G_{\F}$ for the  $\F$-forms of $\g,G$.

\subsection{Known results}
Recall that the universal enveloping algebra $\U_{\F}:=U(\g_{\F})$ has big center.
Namely, we have the restricted $p$th power map $x\mapsto x^{[p]}:\g_{\F}^{(1)}\rightarrow \g_{\F}$,
where the superscript ``(1)'' indicates the Frobenius twist so that $(ax)^{[p]}=ax^{[p]}$. Then we have an algebra embedding $S(\g_\F^{(1)})\rightarrow \U_{\F}$ with central image:
on $\g^{(1)}_\F$ it is given by $x\mapsto x^p-x^{[p]}$. The image of $S(\g_\F^{(1)})$
in $\U_\F$ is known as the $p$-center.  We also have the so called Harish-Chandra
center $\U_{\F}^{G_\F}$, as in characteristic $0$ it is identified with $\F[\h^*]^W$,
where $\h$ is a Cartan subalgebra of $\g$, $W$ denotes the Weyl group that acts on
$\h^*$ via the $\rho$-shifted action. By a theorem of Veldkamp, \cite{Veldkamp},
the full center of $\U_{\F}$  coincides with $$S(\g^{(1)}_{\F})\otimes_{S(\g^{(1)}_\F)^{G^{(1)}_\F}}\U_\F^G.$$

For $\chi\in \g_{\F}^{(1)*}, \lambda\in \h^*_\F/(W,\cdot)$ we can define the central
reduction $\U_{\lambda,\F}^\chi$ of $\U_{\F}$ by taking the quotient by the maximal ideals
of $\chi,\lambda$ in the corresponding central subalgebras. The algebra $\U_{\lambda,\F}^\chi$
is  finite dimensional.
The most interesting case is when $\chi$ is nilpotent, the general case is well-known to  reduce
to that one, see \cite{KW}. If $\chi$ is nilpotent and
$\U^\chi_{\lambda,\F}\neq \{0\}$, then $\lambda\in \h^*_{\F_p}/(W,\cdot)$.
Then one can reduce the question about the dimensions of  irreducible modules to the
case when $\lambda$ is regular, see, e.g., \cite[Section 6.1]{BMR}. This is what we are going to
assume from now on. If $\lambda,\lambda'$ are two regular elements of $\h^*_{\F_p}/(W,\cdot)$,
then the algebras $\U^\chi_{\lambda,\F}$ and $\U^\chi_{\lambda',\F}$ are Morita
equivalent.  If $\chi'=g\chi$ for some $g\in G_\F$, then $g$ gives
an isomorphism $\U^\chi_{\lambda,\F}\cong \U^{\chi'}_{\lambda,\F}$. So the algebra
$\U^\chi_{\lambda,\F}$ depends only on the $G_\F$-orbit of $\chi$. The nilpotent
$G_{\F}$-orbits in $\g_{\F}^{(1)*}$ are in a natural bijection with the nilpotent
$G$-orbits in $\g$.  So, up to a Morita equivalence, we have one algebra $\U^\chi_{\lambda,\F}$
per nilpotent orbit in $\g$.

We will write $e$ for the nilpotent element of $\g$ lying in
the orbit corresponding that of $\chi$.

Let us explain known results on the simple $\U^\chi_{\lambda,\F}$-modules.
First of all, the number is known. Indeed, thanks to  \cite[Corollary 5.4.3, Section 7]{BMR}
there is a natural isomorphism
$K_0(\U^\chi_{\lambda,\F}\operatorname{-mod})\xrightarrow{\sim}K_0(\mathcal{B}_e)$,
where $\mathcal{B}_e$ denotes the Springer fiber of $e$ and $K_0(\mathcal{B}_e)$
stands for the Grothendieck group of the coherent sheaves on $\B_e$.

In \cite{BM}, the first named author and Mirkovic described the classes
of simples in $K_0(\mathcal{B}_e)$. Namely, Lusztig,  \cite{Lusztig_K1}, gave a
conjectural definition of  a canonical basis in $K_0^{\C^\times}(\mathcal{B}_e)$ for a suitable contracting $\C^\times$-action on $\mathcal{B}_e$. By \cite[Theorem 5.3.5]{BM},
the specialization of Lusztig's canonical basis to $q=1$ (where $q$ is the equivariant
parameter) coincides with the basis of simple modules in $K_0(\U^\chi_{\lambda,\F}\operatorname{-mod})$.
A problem with this canonical basis is that it is very implicit (with an exception of
the case when $e$ is principal in a Levi subalgebra, see \cite{Lusztig_periodic}
and \cite[Section 10]{Lusztig_K2}).
A part of this problem is that, in general, $K_0(\mathcal{B}_e)$ does not have an easily
understandable standard basis (or a spanning set) which the canonical basis can be compared to.
In particular, the result from \cite{BM} does not allow to get the dimension formulas
for the irreducible representations in the  general case.

Before we proceed to our results on  dimensions and $K_0$-classes of the simple modules,
let us explain what is known about their combinatorial classification.
For this, let us recall that a nilpotent element $e\in \g$ is called {\it distinguished}
if it is not contained in any proper Levi subalgebra. Any nilpotent element $e$
is distinguished in a Levi subalgebra of $\g$.  Namely, consider the  maximal torus $T_0$
of the centralizer $Z_G(e)$. The Levi subalgebra we need is $\underline{\g}:=\g^{T_0}$.
Note that the group $G_\chi$ acts on $\U^\chi_{\lambda,\F}$ by algebra automorphisms. In particular,
a maximal torus $T_{0,\F}\subset G_\chi$ acts. We can consider the category
$\U^\chi_{\lambda,\F}\operatorname{-mod}^{T_0}$ of weakly $T_{0,\F}$-equivariant (a.k.a. graded)
$\U^\chi_{\lambda,\F}$-modules. Every simple $\U_{\lambda,\F}^\chi$-module has an equivariant lift
unique up to a twist with a character of $T_{0,\F}$. So the set $\operatorname{Irr}(\U_{\lambda,\F}^\chi)$
of irreducible $\U_{\lambda,\F}^\chi$-modules  is in bijection with the quotient of
$\operatorname{Irr}(\U^\chi_{\lambda,\F}\operatorname{-mod}^{T_0})$ by the free
action of the character lattice $\mathfrak{X}(T_0)$.
On the other hand, the simples in $\U^\chi_{\lambda,\F}\operatorname{-mod}^{T_0}$ are in bijection with the $T_0$-equivariant simple objects in $\bigoplus_{\underline{\lambda}}\underline{\U}^{\chi}_{\underline{\lambda},\F}\operatorname{-mod}$.
Here $\underline{\U}$ stands for the enveloping algebra for $\underline{\g}$,
the summation is over all $\underline{\lambda}\in \mathfrak{h}^*_{\F_p}/(\underline{W},\cdot)$ that map
to $\lambda$ under the natural projection $\mathfrak{h}^*_{\F_p}/(\underline{W},\cdot)
\rightarrow \mathfrak{h}^*_{\F_p}/(W,\cdot)$. The bijection between the sets of simples
works as follows: one fixes a generic one-parameter subgroup $\nu$ of $T_0$ and then
takes the highest weight space of a simple $T_0$-equivariant
$\U^\chi_{\lambda,\F}$-module to get a simple $T_0$-equivariant
module in $\bigoplus_{\underline{\lambda}}\underline{\U}^{\chi}_{\underline{\lambda},\F}\operatorname{-mod}$.
Note that different choices of $\nu$ lead to different bijections.

In the special case when $e$ is principal in $\underline{\g}$, each  algebra $\underline{\U}^{\chi}_{\underline{\lambda},\F}$
is just $\F$ and so has a unique simple module. Therefore the simples in
$\U^\chi_{\lambda,\F}\operatorname{-mod}$ are in bijection with $W/\underline{W}$.

For the general $\chi$,  the situation is more difficult as there is no explicit labelling set for
the simples in the case of a general distinguished element.

On the other hand, in \cite{W_dim}, the second named author considered a problem that
should be thought as a ``finite'' analog of the problem considered in the present
paper (that is ``affine''). The main result of \cite{W_dim} is the Kazhdan-Lusztig type
formulas for the characters of certain equivariantly simple modules over finite W-algebras.
An approach used in \cite{W_dim} was to relate the category of equivariant modules
over the finite W-algebra to a suitable parabolic category $\mathcal{O}$ over $\g$.
Note that the simple finite dimensional modules over the W-algebra associated to
$e$ with central character $\lambda\in \h^*_{\Z}/(W,\cdot)$
embed into $\operatorname{Irr}(\U_{\lambda,\F}^\chi)$ so that the dimension
multiplies by $p^{\dim Ge/2}$, see, e.g., \cite{BL}.

\subsection{Dimensions of equivariantly irreducible modules: distinguished case}\label{SS_intro_distinguished}
In this paper we give a combinatorial classification and compute dimensions (as well as characters
and, even stronger, $K_0$-classes)
of {\it equivariantly} irreducible $\U^\chi_{\lambda,\F}$-modules. Let us first explain
our setting and our results in the case when $e$ is distinguished. For simplicity,
assume $G$ is semisimple.

As we have mentioned in the previous section, the group $G_\chi$ acts on $\U^\chi_{\lambda,\F}$
by automorphisms. Note that since $e$ is distinguished, the reductive part of $G_\chi$
is finite. Denote this group by $A$. We will consider the category $\U^\chi_{\lambda,\F}\operatorname{-mod}^{A}$
of $A$-equivariant $\U^\chi_{\lambda,\F}$-modules.

It turns out that there is a natural labelling set for the simple objects in
$\U^\chi_{\lambda,\F}\operatorname{-mod}^{A}$. To describe it, let us
recall the parabolic subalgebra attached to $e$. Namely, we include $e$ into
an $\slf_2$-triple $(e,h,f)$. Then we can consider the parabolic subalgebra
$\mathfrak{p}\subset \g$, the sum of all eigenspaces
for $\operatorname{ad}(h)$ with nonnegative eigenvalues. Let $P$ denote the
corresponding parabolic subgroup of $G$ and $L:=Z_G(h)$ be its Levi
subgroup.  We write $W^{a}$
for the extended affine  Weyl group $W\ltimes \mathfrak{X}(T)$, where $\mathfrak{X}(T)$
denote the character lattice of $T$. The parabolic subgroup $P$ defines
a standard parabolic subgroup of $W$ to be denoted by $W_P$. Note that $W_P$
is also a standard parabolic subgroup of $W^a$. We write
$W^{a,P}$ for the set of maximal length representatives of the right cosets for
$W_P$ so that $W^{a,P}\xrightarrow{\sim} W^{a}/W_P$. It is a standard fact
that $W^{a,P}$ contains a left cell ${\mathfrak{c}}_P$ such that $W^{a,P}$ is the union
of left cells that are less than or equal to ${\mathfrak{c}}_P$. Note that since $e$
is distinguished, ${\mathfrak{c}}_P$ is finite, this follows
from \cite{Lusztig_affine}. We will see below that there is a natural
bijection ${\mathfrak{c}}_P\xrightarrow{\sim} \operatorname{Irr}(\U^\chi_{\lambda,\F}\operatorname{-mod}^{A})$.

Let us explain how to compute  dimensions of the simple objects in $\U^\chi_{\lambda,\F}\operatorname{-mod}^{A}$. From an element
in $x\in W^{a,P}$ we can produce a dominant weight $\mu_x$ for $L_{\F}$ that maps
to $\lambda$ under the natural projection $\mathfrak{X}(T)\rightarrow \h^*_{\F_p}/(W,\cdot)$.
Namely, we fix an element $\mu^\circ$ in the p-alcove defined by $\langle\alpha_i^\vee,\bullet\rangle\leqslant -1,
\langle\alpha_0^\vee,\bullet+\rho\rangle\geqslant -p$ (below we will call this p-alcove
{\it anti-dominant}) that maps  to $\lambda$ (here the $\alpha_i^\vee$'s
are the simple coroots and $\alpha_0^\vee$ is the maximal coroot). For example,
for $\lambda=W\cdot 0$ we take $\mu^\circ:=-2\rho$. Consider the $W^{a}$-action on
$\mathfrak{X}(T)$ given by $w.\mu:=w\cdot \mu$ for $w\in W$ and $t_\theta.\mu:=\mu+p\theta$ for $\theta\in \mathfrak{X}(T)$.
Note that for $x\in W^{a,P}$, the element $\mu_x:=x^{-1}\cdot \mu^\circ$ is dominant for
$L$. Let $d_L(\mu_x)$ denote  dimension of the finite dimensional $L$-module
with highest weight $\mu_x$, it is given by the Weyl dimension formula.

The next and final ingredient to state the dimension formula is the parabolic affine Kazhdan-Lusztig
polynomials: to elements $x,y\in W^{a,P}$ we assign the corresponding Kazhdan-Lusztig
polynomial $c^P_{x,y}(v)\in \Z[v^{- 1}]$. Our convention, recalled in more detail
in Section \ref{SS_affine_Hecke_reminder}, is that $c^P_{x,y}(1)$ is the coefficient
of the class of the standard object labelled by $y$ in the parabolic affine category $\mathcal{O}$
in the simple object labelled by $x$.   Note that for any given $x$, only finitely many of the polynomials
$c^P_{x,y}$ are nonzero.

\begin{Thm}\label{Thm:disting_dim}
The dimension of the simple module in $\U^\chi_{\lambda,\F}\operatorname{-mod}^{A}$
labelled by $x\in {\mathfrak{c}}_P$ equals
$$\sum_{y\in W^{a,P}}c^P_{x,y}(1)(p^{\dim Ge/2}d_L(\mu_y)).$$
\end{Thm}

We can upgrade this theorem to computing the $A$-characters of the simple
equivariant $\U^\chi_{\lambda,\F}$-modules. For this we need to replace
$p^{\dim Ge/2}d_L(\mu_y)$ with the $A$-character of $U^0(\mathfrak{m}^-_\F)\otimes V_L(\mu_y)$,
where $\mathfrak{m}^-_\F$ is the maximal nilpotent subalgebra of the opposite
parabolic of $\mathfrak{p}_\F$ and $V_L(\mu_y)$ is the irreducible $L$-module with
highest weight $\mu_y$. In fact, the strongest version of Theorem \ref{Thm:disting_dim}
has to do with $K_0$-classes: we will see that the class of the simple in
$\U^\chi_{\lambda,\F}\operatorname{-mod}^{A}$ labelled by $x$ is
$\sum_{y\in W^{a,P}}c^P_{x,y}(1)[W^\chi_\F(\mu_y)]$, where
$W^\chi_\F(\mu_y)$ is a certain induced module in
$\U^\chi_{\lambda,\F}\operatorname{-mod}^{A}$ to be defined
in Section \ref{SS_chi_Weyl}. The statement about $K_0$-classes
of simples will be established in Section \ref{SS_char_distinguish}.




\subsection{The parabolic affine Hecke category}
The proof of Theorem \ref{Thm:disting_dim} relies upon Theorem \ref{Prop:parabol_equiv_der}
which establishes
a  special case of  \cite[Conjecture 59]{B_Hecke} providing a generalization of the main
result of that paper. Namely, in \cite{B_Hecke} the first named author constructed an
equivalence between
\begin{itemize}
\item  $D^b_{I^\circ}(\Fl)$, the derived category of
Iwahori-monodromic constructible sheaves
on the affine flag variety of the dual group $G^\vee$,
\item
and $D^b(\Coh^G(\St_B))$, the
category of $G$-equivariant coherent sheaves on the Steinberg variety of triples
$\St_B:=\tilde{\g}\times _\g \tilde{\Nilp}$.
\end{itemize}
Here $\tilde{\Nilp}$ denote the Springer resolution, $\tilde{\g}$ denotes the Grothendieck simultaneous resolution and $I^\circ$ is the pro-unipotent radical
of the standard Iwahori subgroup in the loop group of $G^\vee$.
Theorem \ref{Prop:parabol_equiv_der} generalizes the equivalence from
\cite{B_Hecke} to an equivalence
$D^b_{I^\circ}(\Fl_P)\cong D^b(\Coh^G(\St_P))$. Here $P$ is a
parabolic subgroup in $G$. It defines a parabolic subgroup in $G^\vee$ (up to conjugacy) and
$\Fl_P$ is the corresponding parabolic affine flag variety. By
$\St_P$ we denote  a parabolic Steinberg variety, $\tilde{\g}\times_\g \tilde{\Nilp}_P$,
where $\tilde{\Nilp}_P:=T^*(G/P)$.
The proof of Theorem \ref{Prop:parabol_equiv_der} proceeds by identifying the abelian
category $\Perv_{I^\circ}(\Fl_P)$ of $I^\circ$-equivariant perverse sheaves on $\Fl_P$ with a full
subcategory in $D^b_{I^\circ}(\Fl)\cong D^b(\Coh^G(\St_B))$. We then construct an abelian
subcategory in $D^b(\Coh^G(\St_P))$ and show that a pull-push functor
$D^b(\Coh^G(\St_P))\rightarrow D^b(\Coh^G(\St_B))$ identifies this subcategory
with $\Perv_{I^\circ}(\Fl_P)$.
The equivalence of abelian
categories also plays an important role in proving Theorem \ref{Thm:disting_dim}
as it allows to relate a characteristic $0$
counterpart of  $\U^\chi_{\lambda,\F}\operatorname{-mod}^{A}$
to a cell quotient of the category $\Perv_{I^\circ}(\Fl_P)$.

\subsection{Dimensions of equivariantly irreducible modules: general case}
Now let $e$ be arbitrary. As before we fix a maximal torus $T_{0,\F}\subset G_{\F,\chi}$.
We set $\underline{Q}_\F$ to be the centralizer of $T_{0,\F}$ in a maximal reductive subgroup of $G_{\F,\chi}$. The connected component of $\underline{Q}_\F$ is $T_{0,\F}$.
Since $p$ is large enough, $\underline{Q}_\F$ is linearly reductive. We consider
the category $\U^\chi_{\lambda,\F}\operatorname{-mod}^{\underline{Q}}$ of $\underline{Q}_\F$-equivariant
$\U^\chi_{\lambda,\F}$-modules. Inside we consider the Serre subcategory of
all objects $M$ such that the Lie algebra $\underline{\mathfrak{q}}_\F$ of $\underline{Q}_\F$
acts on the graded component $M_\theta$ by $\theta \operatorname{mod} p$
for all $\theta\in \mathfrak{X}(T_{0,\F})$.
Denote this subcategory by
$\U^\chi_{\lambda,\F}\operatorname{-mod}^{\underline{Q},0}$.

Let us write $\underline{G}_{\F}$ for the centralizer of $T_{0,\F}$ in $G_\F$.
Consider the Levi and  parabolic subgroups $L_\F\subset \underline{P}_\F\subset \underline{G}_\F$ constructed from
$e$ as in Section \ref{SS_intro_distinguished}.
 Once we pick a
generic one-parameter subgroup $\nu$ of $T_{0,\F}$, we can identify the set
$\operatorname{Irr}(\U^\chi_{\lambda,\F}\operatorname{-mod}^{\underline{Q},0})$
with the set of pairs $(u,\underline{x})$, where $u\in W$ is shortest in $uW_{\underline{G}}$
and $\underline{x}$ lies in  the left cell
$\mathfrak{c}_{\underline{P}}$. Namely, $\mathcal{L}\in \operatorname{Irr}(\U^\chi_{\lambda,\F}\operatorname{-mod}^{\underline{Q},0})$
gives rise to its $\nu$-highest weight component $\underline{\mathcal{L}}
\in \operatorname{Irr}(\U^\chi_{u^{-1}\cdot\lambda,\F}\operatorname{-mod}^{\underline{Q},0})$,
and, by Section \ref{SS_intro_distinguished}, $\underline{\mathcal{L}}$ gives rise to
$\underline{x}$.
Note that $\underline{G}$ is not semisimple
so $\mathfrak{c}_{\underline{P}}$ is not finite; however, $\mathfrak{c}_{\underline{P}}$ is preserved by
the translations by elements of $\mathfrak{X}(\underline{G})$ and the quotient
by this action is finite.

The choice of a generic one-parameter   subgroup $\nu$ in $T_{0,\F}$ defines a parabolic
subgroup $G_\F^{\geqslant 0}=\underline{G}_\F\ltimes G^{>0}_\F$.
Inside we have the parabolic subgroup $P_\F:=\underline{P}_\F\ltimes G^{>0}_\F$.

The first ingredient that we need to write the dimension formula is
the {\it semiperiodic parabolic} Kazhdan-Lusztig polynomials to be denoted by
$c^{P,\infty}_{x,y}(v)$. Here $x,y\in W^{a,P}$. When $\chi=0$,  we recover periodic affine Kazhdan-Lusztig polynomials,
see \cite{Lusztig_Jantzen}. In the general case, the semiperiodic parabolic
polynomials are constructed as follows.
It turns out that once an element $\theta\in \mathfrak{X}(\underline{G})$ lying
in the dominant Weyl chamber for $G$ is large enough (meaning that its pairings with
the simple coroots not in $\underline{G}$ are large enough; how large depends on $x,y$)
the polynomial $c^{P}_{xt_\theta, yt_\theta}(v)$ depends on $x,y$ and not on $\theta$.
This was proved in Stroppel's master thesis, \cite{Stroppel}, but we give an independent proof.
We denote this stabilized polynomial by  $c^{P,\infty}_{x,y}(v)$.


Furthermore, from $y\in W^{a,P}$ we can produce the dominant weight $\mu_y$ of $L$
as in Section \ref{SS_intro_distinguished}. Let $\hat{d}_L(\mu_y)$ denote the $\underline{Q}$-character of the corresponding
irreducible $L$-module.
Finally, consider the character $\mathsf{ch}_{\mathfrak{m}^-}$ of the action of $\underline{Q}_\F$
on $U^0(\mathfrak{m}^-_\F)$, where $\mathfrak{m}^-_\F$ is the maximal nilpotent
subalgebra of the parabolic opposite to $\mathfrak{p}_\F$.

\begin{Thm}\label{Thm_dim_general}
The $\underline{Q}_\F$-character of the simple module in $\U^\chi_{\lambda,\F}\operatorname{-mod}^{\underline{Q},0}$
labelled by $u\underline{x}$, where $u\in W$ is shortest in $uW_{\underline{G}}$
and $\underline{x}\in \mathfrak{c}_{\underline{P}}$, equals
$$\sum_{y\in W^{a,P}}c^{P,\infty}_{u\underline{x},y}(1)\mathsf{ch}_{\mathfrak{m}^-}\hat{d}_L(\mu_y).$$
\end{Thm}

Note that, unlike in Theorem \ref{Thm:disting_dim}, the sum in the right hand side
is no longer finite. However,  it is easy to see that it converges in
a suitable topology on $K_0(\operatorname{Rep}(\underline{Q}_\F))$. Also Theorem \ref{Thm_dim_general} upgrades to an equality of classes in $K_0(\U^\chi_{\lambda,\F}\operatorname{-mod}^{\underline{Q},0})$,
just like Theorem \ref{Thm:disting_dim}. See Section \ref{SS_final}.

We would like to point out a classical special case of Theorem \ref{Thm_dim_general}:
when $\chi=0$ and hence $\underline{G}=T$. Here we express the classes of simple $(G_1,T)$-modules
via the classes of baby Verma modules with coefficients that are expressed via periodic
affine Kazhdan-Lusztig polynomials of Lusztig. This result was first obtained combining
results of \cite{KL_affine},\cite{KT1},\cite{AJS} (as well
as \cite{Lusztig_monodromic},\cite{KT2} in the non-simply-laced case).

Another notable special case is when $\chi$ is principal in a Levi.
In this case, a conjecture similar in spirit to Theorem \ref{Thm_dim_general}
was stated in \cite[13.17]{Lusztig_periodic}. In fact, that conjecture can be
deduced already from \cite{BM}.  More precisely, the main result of [BM] is a
proof of \cite[Conjecture 5.12]{Lusztig_K2}
and its relation to modular representations outlined in \cite[\S 14]{Lusztig_K1};
this implies \cite[Conjecture 13.17]{Lusztig_periodic}
in view of \cite[Proposition 10.7]{Lusztig_K2}.
Also note that Theorem \ref{Thm_dim_general} and
the conjecture \cite[13.17]{Lusztig_periodic} describe the multiplicities in
a priori different categories. In the former we have modules that are equivariant
with respect to the full center of the Levi, while in the latter the modules are
equivariant with respect to the connected component of that center. One should
be able to deduce  \cite[Conjecture 13.17]{Lusztig_periodic}
from Theorem \ref{Thm_dim_general}, but this is nontrivial.

The proof of Theorem \ref{Thm_dim_general} completed in
Section \ref{SS_final} is quite different from that
of Theorem \ref{Thm:disting_dim}. Namely, we construct a contravariant duality
functor on $\U^\chi_{\lambda,\F}\operatorname{-mod}^{\underline{Q}}$ that fixes all simple
objects in this category. Then we compute a graded lift of this duality functor
in a suitable graded lift of a category closely related to
$\U^\chi_{\lambda,\F}\operatorname{-mod}^{\underline{Q},0}$. In the end of the  paper we will speculate on
a categorical nature of Theorem \ref{Thm_dim_general}.

A natural question is how to compute the $T_0$-characters of irreducible
modules in $\U^\chi_{\lambda,\F}\operatorname{-mod}^{T_0}$. Thanks to
Theorem \ref{Thm_dim_general}, this question reduces to understanding
the decomposition of the $\underline{Q}$-equivariantly irreducible modules
into the usual irreducibles. While we do not have an explicit answer to this
question in the general case, we discuss it in Section \ref{SS_irred}.

\subsection{Applications to characteristic $0$ representation theory}
One can, in principle, use Theorem \ref{Thm_dim_general} to compute
the dimensions of equivariantly irreducible representations
of finite W-algebras. In more detail, to $e\in \g$ one assigns
the finite W-algebra $\Walg$ and to $\lambda\in \h^*/(W,\cdot)$ one assigns
the central reduction $\Walg_\lambda$ of $\Walg$.

In \cite{BL}, we have related the irreducible finite dimensional representations
of $\Walg_\lambda$ to those of $\U^\chi_{\lambda,\F}$. Assume $\lambda$
is rational.  One has  natural bijections between the sets $\operatorname{Irr}(\U^\chi_{\lambda,\F}\operatorname{-mod})$
for different $p$ provided $p$ is sufficiently large and the residue of $p$
modulo the ``denominator'' of $\lambda$
is fixed. Under these bijections, the dimensions of the irreducibles in
$\U^\chi_{\lambda,\F}\operatorname{-mod}$ are polynomials
in $p$. Then $\operatorname{Irr}_{fin}(\Walg_\lambda)$ embeds into
$\operatorname{Irr}(\U^\chi_{\lambda,\F})$ as the subset of all representations
with  degree of the  dimension polynomial equal to $\frac{1}{2}\dim Ge$.
Equivalently, \cite[Theorem 1.1]{BL}, $\operatorname{Irr}_{fin}(\Walg_\lambda)$ consists of
all simples whose $K_0$-class lies in a certain two-sided cell component of
$K_0(\U^\chi_{\lambda,\F}\operatorname{-mod})$ for the integral Weyl group $W_{[\lambda]}$ of $\lambda$.

Thanks to Theorem \ref{Thm_dim_general} we get a formula for the dimensions of equivariantly
irreducible representations of $\Walg_{\lambda}$. However, it would be desirable to
get a formula in terms of finite not affine Kazhdan-Lusztig polynomials, as
that of \cite{W_dim} in the case when $\lambda$ is integral. Knowing the dimensions
of the irreducible $\Walg_\lambda$-modules should lead to a solution of other
problems in the Lie representation theory over $\C$, for example, to formulas
for Goldie ranks of primitive ideals. A precise relation between the Goldie
ranks and the dimensions was conjectured in \cite{LP} (at least, when $\g$
is classical).

\subsection{Content of the paper}
The subsequent sections of the paper can be roughly separated into two groups.

Sections \ref{S_basics}-\ref{S_constr_realiz} are preparatory. While  they
contain some new results, those results are technical ramifications of well-known
ones. In Section \ref{S_basics} we recall some basic facts and constructions
related to the modular representation theory of semisimple Lie algebras. In particular,
there we introduce categories $\U^\chi_{(\lambda),\F}\operatorname{-mod}^{\underline{Q}}$, the main
object of study in this paper as well as ``standard modules'' in these categories
that we call $\chi$-Weyl modules. In Section \ref{S_der_loc} we recall the derived
localization theorem in positive characteristic proved in \cite{BMR}. In Section
\ref{S_tilting} we recall the tilting bundles on the Springer and Grothendieck
simulataneous resolutions constructed in \cite{BM} and their properties.
And then in Section \ref{S_constr_realiz} we recall results of \cite{B_Hecke}
on an equivalence between the coherent and constructible categorifications of
affine Hecke algebras.

Sections \ref{S_parab_equiv}-\ref{S_K_0_classes} are the main part of the paper.
In Section \ref{S_parab_equiv} we generalize the results of \cite{B_Hecke}
to the parabolic setting. This is the main ingredient in the proof of
Theorem \ref{Thm:disting_dim}. In Section \ref{S_duality} we introduce
and study a contravariant duality functor on the category
$\U^\chi_{(\lambda),\F}\operatorname{-mod}^{\underline{Q}}$ that is an important ingredient
in the proof of Theorem \ref{Thm_dim_general}. Then in Section
\ref{S_K_0_classes} we prove Theorems \ref{Thm:disting_dim} and
\ref{Thm_dim_general}.


{\bf Acknowledgements}: The authors thank Dmytro Arinkin, Alexander Braverman, Dennis Gaitsgory,
George Lusztig, Simon Riche, and Geordie Williamson for stimulating discussions.
R.B. has been partially supported by the NSF under
grant DMS-1601953. I.L. has been partially supported by the NSF under grant
DMS-1501558.

\subsection{List of notation}
Here we provide the list of common notation  used in the paper. The notation
is listed alphabetically with Roman letters and then Greek letters.

\setlongtables

\begin{longtable}{p{3.4cm} p{12cm}}
$A$&$:=Z_G(e)/Z_G(e)^\circ$,\\
$\A$&$:=\operatorname{End}_{\tilde{\Nilp}}(\mathcal{T}),$\\
$\A_\h$&:=$\operatorname{End}_{\tilde{\g}}(\mathcal{T}_\h),$\\
$\A_P$&$:=\operatorname{End}_{\tilde{\Nilp}_P}(\mathcal{T}_P),$\\
$B$& a Borel subgroup of $G$,\\
$\B$& $:=G/B$, the flag variety of $G$,\\
$\underline{\B}$& the flag variety of $\underline{G}$,\\
$\B_e$& the Springer fiber in $\B$ of a nilpotent element $e\in \g$,\\
$\Br$& the braid group of $W$,\\
$\Br^a$& the braid group of $W^a$,\\
$C_{x}$& the Kazhdan-Lusztig basis element in $\Hecke^a_G$ labelled
by $x\in W^a$,\\
$\mathfrak{c}_P$& the left cell in $W^a$ containing $w_{0,P}$,\\
$\underline{\C}_X$& the constant sheaf on a topological space $X$ with fiber $\C$,\\
$\Coh_Y(X)$& the category of coherent sheaves on a scheme $X$ that are
supported on $Y$ set-theoretically,\\
$\D$& the contravariant duality functor of $\U^\chi_{(\lambda),\F}\operatorname{-mod}^{\underline{Q}}$
defined by (\ref{eq:D_definition}),\\
$\D_{coh}$& the contravariant duality functor of $D^b(\Coh^{\underline{Q}}(\B_\chi))$
from Section \ref{SS_loc_dual},\\
$D^b_H(X)$& the $H$-equivariant bounded constructible derived category of
a variety (or an ind-variety) $X$,\\
$D^\mu_X$& the sheaf of $\mu$-twisted differential operators on a smooth algebraic variety $X$,\\
$\tilde{D}_\B$& the sheaf $\upsilon_* D_{G/U}$, where $\upsilon:G/U\rightarrow G/B$
is the natural projection,\\
$(e,h,f)$& an $\slf_2$-triple in $\g$,\\
$\operatorname{Fim}(?)$& the full Karoubian subcategory generated by the image of a functor $?$,\\
$\Fl$&$:=G^\vee((t))/I^\vee$,\\
$\Fl_P$&$:=G^\vee((t))/J^\vee$, where $J^\vee$ is the parahoric subgroup of $G^\vee((t))$
corresponding to $P\subset G$,\\
$G$& a connected reductive algebraic group over $\C$,\\
$G^\vee$& the Langlands dual group of $G$,\\
$\g^i$& $:=\{x\in \g| \nu(t)x=t^i x, \forall t\in T_0\}$,\\
$\g^{\geqslant 0}$& $:=\bigoplus_{i\geqslant 0}\g^i$,\\
$\underline{\g}$&$:=\g^0$,\\
$\g(i)$&$:=\ker(\operatorname{ad}h-i)$,\\
$G^{\geqslant 0}$, $\underline{G}$& the connected subgroups of $G$
with Lie algebras $\g^{\geqslant 0},\underline{\g}$,\\
$\g_\h$&$:=\g^*\times_{\h^*/W}\h^*$,\\
$\tilde{\g}$& the Grothendieck resolution of $\g_\h$,\\
$\h$& a Cartan subalgebra of $\g$ contained in $\mathfrak{b}$,\\
$\Hecke_W$& the Hecke algebra of $W$,\\
$\Hecke^a_G$& the affine Hecke algebra of a reductive algebraic group $G$; it is associated
to $W^a$,\\
$H_x$&the standard basis element in $\Hecke^a_G$ labelled by $x\in W^a$,\\
$K_0^H(X)$&$:=K_0(\Coh^H(X))$,\\
$I^\vee$& the Iwahori subgroup of $G^\vee((t))$,\\
$I^\circ$& the kernel of $I^\vee\twoheadrightarrow T^\vee$,\\
$L$& the Levi subgroup of $P$ containing $T$,\\
$M$& the unipotent radical of $P$,\\
$M^-$& the unipotent radical of the parabolic opposite to $P$,\\
$\Nilp$& the nilpotent cone in $\g$,\\
$\tilde{\Nilp}$&$:=T^*\B$,\\
$\tilde{\Nilp}_P$&$:=T^*\mathcal{P}$,\\
$\mathcal{O}(\mu)$& the line bundle on $\B$ (and related varieties) corresponding to
$\mu\in \mathfrak{X}(T)$,\\
$\Orb$&$:=Ge$,\\
$P$& a parabolic subgroup of $G$ containing $B$,\\
$\mathcal{P}$&$:=G/P$,\\
$Q$&$:=Z_G(e,h,f)$,\\
$\underline{Q}$& the centralizer of $T_0$ in $Q$,\\
$S$& the Slodowy slice $e+\mathfrak{z}_\g(f)$,\\
$\St_\h$&$:=\tilde{\g}\times_{\g}\tilde{\g}$,\\
$\St_B$&$:=\tilde{\g}\times_{\g}\tilde{\Nilp}$,\\
$\St_0$&$:=\tilde{\Nilp}\times^L_\g\tilde{\Nilp}$,\\
$\St_P$&$:=\tilde{\g}\times^L_{\g} \tilde{\Nilp}_P$,\\
$T$& the maximal torus of $B$ with Lie algebra $\mathfrak{h}$,\\
$T_0$& a maximal torus in $Q$,\\
$T_w$& the element of $\Br^a$ corresponding to $w\in W$ or the wall-crossing functor for
this element,\\
$\U$&$:=U(\g)$,\\
$\U^\chi_\F$& the $p$-central reduction of $\U_\F$,\\
$\U^\chi_{(\lambda),\F}$& the infinitesimal block in $\U^\chi_\F$ corresponding to a HC character $\lambda$.\\
$\underline{\U}$&$:=U(\underline{\g})$,\\
$\mathcal{V}_\mu^\chi$& the splitting bundle introduced in Section \ref{SS_splitting},\\
$W$& the Weyl group of $G$,\\
$W^a$&$:=W\ltimes \mathfrak{X}(T)$, the extended affine Weyl group,\\
$W_P$&the parabolic subgroup of $W$ corresponding to $P$,\\
$W^{a,P}$& the set of longest coset representatives in $xW_P\subset W^a, x\in W^a$,\\
$W^{P,-}$& the set of shortest coset representatives in $wW_P, w\in W$,\\
$W^\chi_\F(\mu)$& the $\chi$-Weyl module corresponding to a dominant weight
$\mu$ of $L$,\\
$w_0$& the longest element of $W$,\\
$w_{0,P}$& the longest element of $W_P$,\\
$\mathfrak{X}(H)$& the character group of an algebraic group $H$,\\
$Z$&$:=G\times^B \mathfrak{m}$,\\
$\alpha_1,\ldots,\alpha_r$& the simple roots of $\g$,\\
$\alpha_0$& the root of $\g$ such that $\alpha_0^\vee$ is maximal,\\
$\gamma$& the one-parameter subgroup of $G$ corresponding to $h$,\\
$\Delta^P(\mu)$& the parabolic Verma $\U$-module for $P$ with highest weight $\mu$,\\
$\underline{\Delta}^\chi$& the parabolic induction functor $\underline{\U}^\chi_{\F}\operatorname{-mod}^{\underline{Q}}\rightarrow
\U^\chi_{\F}\operatorname{-mod}^{\underline{Q}}$,\\
$\eta$& the projection $\Fl\rightarrow \Fl_P$,\\
$\iota$& the embedding $Z\hookrightarrow \tilde{\Nilp}$,\\
$\mu_x$&$:=x^{-1}\cdot \mu^\circ$ for $\mu^\circ$ in the anti-dominant $p$-alcove,\\
$\nu$& a generic one-parameter subgroup of $T_0$,\\
$\varpi$& the projection $Z\twoheadrightarrow \tilde{\Nilp}_P$,\\
$\rho$& half the sum of positive roots,\\
$\pi$& the Springer resolution morphism $\tilde{\Nilp}\rightarrow \Nilp$
or $\tilde{\g}\rightarrow \g_\h$,\\
$\sigma$& the standard antiinvolution of $\g$ given by $\sigma(e_i)=f_i$,\\
$\varsigma$&$:=\operatorname{Ad}(n)\sigma$, defined in Section \ref{SS_dual_basic},\\
$\tau$& the derived equivalences from Theorem \ref{Thm:derived_equiv},\\
$\tau_P$& the derived equivalence from Theorem \ref{Prop:parabol_equiv_der},\\
$\varphi_1,\varphi_2$& the functors defined by (\ref{eq:coherent_functor}),
(\ref{eq:varphi_2_def}), respectively.\\
$\chi$& the element in $\g_{\F}^{(1)*}$ corresponding to $(e,\cdot)\in \g^*$.\\
$\psi_1,\psi_2$& the left adjoints of $\varphi_1,\varphi_2$.\\
$\Omega_X$& the canonical bundle of a smooth variety $X$.\\
\end{longtable}

\section{Basics on modular representations}\label{S_basics}
\subsection{Notation and content}\label{SS_notation_basic}
Let $G$ be a connected reductive  algebraic
group over $\C$ and $\g$ be its Lie algebra.  We identify $\g$ and $\g^*$ via the Killing form. We fix a
nilpotent orbit $\Orb\subset \g$ and pick an element $e\in \Orb$. We include $e$ into an $\slf_2$-triple
$(e,h,f)$. Let us write $\g(i)$ for $\ker(\operatorname{ad}(h)-i)$. We write  $A$ for $Z_G(e)/Z_G(e)^\circ$. We also write $T_0$ for a maximal torus
in $Q:=Z_G(e,h,f)$. Further, we write $\U$ for $U(\g)$.

Pick a generic one-parameter subgroup $\nu: \mathbb{G}_m\rightarrow T_0$. We set
$\g^i:=\{x\in \g| \nu(t)x=t^i x, \forall t\in \mathbb{G}_m\}$.
We will also write $\ug$ for $\g^0$ and $\underline{\U}$ for $U(\ug)$.
It is known that $e$ is even in $\ug$, i.e., the eigenvalues of $h$ in $\ug$
are even.
Set $\g^{\geqslant 0}:=\bigoplus_{i\geqslant 0}\g^i$, this is a parabolic subalgebra in $\g$
with Levi subalgebra $\ug$.  Let $G^{\geqslant 0},\underline{G}$ denote the corresponding
subgroups of $G$.

We write  $\g(j)$ for $\ker(\operatorname{ad}(h)-j)$. Set $\underline{\p}:=\bigoplus_{j\geqslant 0}(\ug\cap \g(j)), \lf=\underline{\g}\cap \g(0),
\p:=\underline{\p}\oplus \g^{>0}$.
Then $\p$ is a parabolic subalgebra in $\g$ with Levi subalgebra $\lf$.
Let $L\subset \underline{P}\subset P$ denote the corresponding connected subgroups of $G$.
Further, let $M$ denote the unipotent radical of $P$ and
$M^-$ stand for the unipotent radical of the opposite parabolic.

We pick a Borel subgroup $B\subset G$ containing $M$. We also pick  a maximal torus
$T\subset L\cap B$. Let $\alpha_1,\ldots,\alpha_r$ denote the corresponding
simple roots and let $\alpha_0$ be a root such that $\alpha_0^\vee$ is maximal.
Let $\Xfr$ denote the character lattice of $T$ and $W$ denote
the Weyl group. By $W_P$ (or $W_L$) we denote the
parabolic subgroup of $W$ corresponding to $P$.
We consider the (extended) affine Weyl group $W^a:=W\ltimes \Xfr$.
Let $W^{a,P}$ (or $W^{a,L}$) denote the subset of all $x\in W^a$ such that
$x$ is longest in $xW_P$. Finally, let $\Br^a$ denote the
braid group associated to $W^a$.

Now fix a prime number $p\gg 0$ and set $\F:=\overline{\F}_p$. We can assume that $e,h,f$ are defined
over a finite localization of $\Z$ hence all the objects introduced above in this section are defined
over that localization. So they can be base-changed to $\F$, we will indicate this with the subscript
$\F$: $G_\F, \g_\F$, etc. Let $\chi\in \g_\F^{(1)*}$ be the element corresponding to $(e,\cdot)\in \g^*$.
Here, as usual, the superscript (1) denotes the Frobenius twist.

We will need to consider two different $p$-alcoves in $\mathfrak{X}$. By the {\it dominant} $p$-alcove
we mean the locus of $\mu\in \mathfrak{X}$ such that $\langle \mu+\rho, \alpha_i^\vee\rangle
\geqslant 0$ for all $i=1,\ldots,r$, and $\langle \mu+\rho,\alpha_0^\vee\rangle\leqslant p$.
By the {\it antidominant} $p$-alcove we mean the locus of $\mu\in \h^*_{\Z}$ such that $\langle \mu+\rho, \alpha_i^\vee\rangle \leqslant 0$ for all $i=1,\ldots,r$, and $\langle \mu+\rho,\alpha_0^\vee\rangle\geqslant -p$. Note that each of the roots
$\alpha_i, i=0,\ldots,r$, defines a codimension 1 face in the (anti)dominant
$p$-alcove.

We now describe the content of this section. In Section \ref{SS_centr_red} we recall basics on
modular representations of $\g$ and introduce the main category of study in this paper,
$\U^\chi_{(\lambda),\F}\operatorname{-mod}^{\underline{Q}}$. In Section \ref{SS_chi_Weyl} we introduce the
$\chi$-Weyl modules that should be thought as ``standard'' objects in the latter
category. Finally, in Section \ref{SS_transl_braid} we recall translation functors
between the categories  $\U^\chi_{(?),\F}\operatorname{-mod}^{\underline{Q}}$ and also an affine
braid group action on $D^b(\U^\chi_{(\lambda),\F}\operatorname{-mod}^{\underline{Q}})$ (in the case
when $\lambda+\rho$ is regular).

\subsection{Central reductions and equivariant modules}\label{SS_centr_red}
Pick $\lambda\in \h_\F^*/(W,\cdot)$. Let $\m^{HC}_\lambda$ denote the maximal ideal of $\lambda$
in the Harish-Chandra center $\U_\F^{G_\F}\cong \F[\h^*]^{(W,\cdot)}$  and $\m_\chi^{p\operatorname{-cen}}$ be the maximal ideal of $\chi$
in the $p$-center $S(\g_\F^{(1)})$. We can form the central reductions
\begin{equation}\label{eq:centr_red}
\U_{\lambda,\F}:=\U_\F/\U_\F\mathfrak{m}^{HC}_\lambda,\quad
\U^\chi_\F:=\U_\F/\U_\F\mathfrak{m}^{p\operatorname{-cen}}_\chi,\quad
\U^\chi_{\lambda,\F}:=\U_\F/\U_\F(\mathfrak{m}^{HC}_\lambda+\mathfrak{m}^{p\operatorname{-cen}}_\chi).
\end{equation}
We note that $\U^\chi_{\lambda,\F}\neq 0\Rightarrow \lambda\in \h^*_{\F_p}/(W,\cdot)$,
where we write $\h^*_{\F_p}$ for the set of $\F_p$-points of $\h^*_{\F}$. In particular,
the $\F[\h^*]^W$-module $\U^\chi_\F$ is supported on the finite set $\h^*_{\F_p}/(W,\cdot)$.
We write $\U^\chi_{(\lambda),\F}$ for the direct summand of $\U^\chi_\F$ corresponding to
$\lambda$ so that $$\U^\chi_\F=\bigoplus_{\lambda\in \h^*_{\F_p}/(W,\cdot)}\U^\chi_{(\lambda),\F}.$$
The algebra $\U^\chi_{\lambda,\F}$ is the quotient of $\U^\chi_{(\lambda),\F}$ by a nilpotent ideal.

Now let ${\underline{Q}}_\F$ be an algebraic subgroup of $Q_\F$. The group $Q_\F$ acts on $\U^\chi_\F$.
We consider (weakly) ${\underline{Q}}_\F$-equivariant $\U^\chi_\F$-modules, i.e., $\U^\chi_\F$-modules
$V$ equipped with a rational ${\underline{Q}}_\F$-action such that the module map
$\U^\chi_\F\otimes_{\F}V\rightarrow V$ is ${\underline{Q}}_\F$-equivariant.
The  category of these modules is denoted
by $\U^\chi_\F\operatorname{-mod}^{\underline{Q}}$. Similarly, we can consider the categories
$\U^\chi_{\lambda,\F}\operatorname{-mod}^{\underline{Q}}, \U^\chi_{(\lambda),\F}\operatorname{-mod}^{\underline{Q}}$.

The choice of ${\underline{Q}}$ we need is $Z_Q(T_0)$.
In particular, ${\underline{Q}}^\circ=T_0$.
Note that ${\underline{Q}}_\F$ acts on $\Irr(\U^\chi_\F)$ and the action factors
through the component group ${\underline{Q}}_\F/{\underline{Q}}_\F^\circ$ because $\Irr(\U^\chi_\F)$ is a finite set.
When $e$ is distinguished in $\g$, i.e., $T_0=\{1\}$, the group ${\underline{Q}}$ is finite and
coincides with $A$. In general,  the projection ${\underline{Q}}_\F\rightarrow A$ is not surjective.
Note that the order of ${\underline{Q}}_\F/{\underline{Q}}_\F^\circ$ is uniformly bounded with respect to $p$,
in particular, ${\underline{Q}}_\F$ is linearly reductive.

The following lemma
is standard.

\begin{Lem}\label{Lem:irred_reln}
Let $U,V$ be irreducible objects in $\U^\chi_\F\operatorname{-mod}^{\underline{Q}}, \U^\chi_\F\operatorname{-mod}$,
respectively. Then the following hold:
\begin{enumerate}
\item $\Hom_{\g_\F}(V,U)$ is an irreducible projective representation of ${\underline{Q}}_\F$ or zero.
\item  The module $U$ is completely reducible
over $\U^\chi_\F$  and all irreducible $\U^\chi_\F$-modules that appear
in $U$ are in the same ${\underline{Q}}/{\underline{Q}}^\circ$-orbit.
\end{enumerate}
\end{Lem}

\subsection{$\chi$-Weyl modules}\label{SS_chi_Weyl}
Now we are going to produce some examples of modules in $\U^\chi_{\lambda,\F}\operatorname{-mod}^{\underline{Q}}$.
These modules are obtained by induction from a Levi subalgebra.

The induction functor we consider will map
from  $U^0(\lf_{\F})\operatorname{-mod}^{\underline{Q}}$, the category of weakly ${\underline{Q}}_\F$-equivariant
modules over the $p$-central reduction $U^0(\lf_{\F})$, to
$\U^\chi_{\F}\operatorname{-mod}^{{\underline{Q}}}$. Note that the $p$-central reduction $U^0(\mathfrak{p}_{\F})$ has a natural $P_{\F}$-equivariant epimorphism
onto $U^0(\lf_{\F})$. In particular, for an object of $U^0(\lf_{\F})\operatorname{-mod}^{{\underline{Q}}}$
we can consider its inflation to a $U^0(\mathfrak{p}_{\F})$-module, this inflation is ${\underline{Q}}_{\F}$-equivariant.

The embedding $U(\mathfrak{p}_\F)\hookrightarrow U(\mathfrak{g}_\F)$ gives
rise to an embedding $U^0(\mathfrak{p}_\F)\hookrightarrow \U^\chi_{\F}$, which is ${\underline{Q}}_{\F}$-equivariant.
Note that $\mathfrak{m}^-_{\F}$ is ${\underline{Q}}_\F$-stable and
we have a ${\underline{Q}}_{\F}$-equivariant linear isomorphism $\U^\chi_{\F}\xrightarrow{\sim}
U^\chi(\mathfrak{m}^-_{\F})\otimes U^0(\mathfrak{p}_\F)$. The induction functor we need
is
$$\underline{\Delta}^\chi:=\U^\chi_{\F}\otimes_{U^0(\mathfrak{p}_{\F})}\bullet: U^0(\lf_{\F})\operatorname{-mod}^{{\underline{Q}}}
\rightarrow \U^\chi_{\F}\operatorname{-mod}^{{\underline{Q}}}.$$

We will be interested in certain induced modules that we call $\chi$-Weyl modules. Namely, pick
a dominant weight $\mu$ for $L$. Then we can consider the  Weyl module $W_{L,\F}(\mu)$  over
$L_{\F}$, it can be defined as  $\Gamma(\mathcal{O}_{P_\F/B_\F}(\mu^*))^*$,
where $\mu^*$ is the dual highest weight.
Its character is given by the Weyl character
formula. Clearly, we can view $W_{L,\F}(\mu)$ as a ${\underline{Q}}_\F$-equivariant module
over $U^0(\lf_{\F})$.

\begin{defi}\label{defi_chi_Weyl}
The {\it $\chi$-Weyl module} labelled by $\mu$ is, by definition,  $W^\chi_{\F}(\mu):=\underline{\Delta}^\chi(W_{L,\F}(\mu))$. This is
an object in $\U^\chi_\F\operatorname{-mod}^{\underline{Q}}$.
\end{defi}

Let us establish some  properties of the modules $W^\chi_{\F}(\mu)$. The following lemma is
straightforward.

\begin{Lem}\label{Lem:chi_Weyl_basic}
In the notation above, we have the following:
\begin{enumerate}
\item The HC central character of $W^\chi_{\F}(\mu)$ is the image of $W\cdot \mu$
modulo $p$.
\item Let us write $\hat{d}(\mu)$ for the ${\underline{Q}}$-character of  $W_{L,\F}(\mu)$
and $\mathsf{ch}_{\mathfrak{m}^-}$ for the ${\underline{Q}}$-character of $U^\chi(\mathfrak{m}^-)$.
Then the ${\underline{Q}}$-character of $W^\chi_{\F}(\mu)$ equals $\mathsf{ch}_{\mathfrak{m}^-}\hat{d}(\mu)$.
\end{enumerate}
\end{Lem}

Other properties will be proved as they are needed.

\subsection{Translation functors and braid group action}\label{SS_transl_braid}
First, we  discuss translation equivalences between categories $\U^\chi_{(\lambda),\F}
\operatorname{-mod}^{\underline{Q}}, \U^\chi_{(\lambda'),\F}\operatorname{-mod}^{\underline{Q}}$, where $\lambda,\lambda'$
are regular in $\h^*_{\F_p}/(W,\cdot)$.

Fix an element $\mu^\circ$ lying in the antidominant $p$-alcove representing $\lambda$.
For example, if $\lambda=W\cdot 0$, then we pick $\mu^\circ=-2\rho$.
Note that, for $w\in W^{a,P}$, the element $\mu_w:=w^{-1}\cdot \mu^\circ$ is dominant for $L$.

\begin{Lem}\label{Lem:trans_equiv}
There is an equivalence $\U^\chi_{(\lambda),\F}
\operatorname{-mod}^{\underline{Q}}\xrightarrow{\sim} \U^\chi_{(\lambda'),\F}
\operatorname{-mod}^{\underline{Q}}$ that sends $W^\chi_\F(\mu_w)$ to $W_\F^\chi(\mu'_w)$.
\end{Lem}
\begin{proof}
The construction of an equivalence is standard -- via translation functors.
Namely, let $\mu'^\circ-\mu^\circ=w\kappa$, where $\kappa$ is dominant and $w\in W$.
Let $V$ denote the irreducible $G_\F$-module with highest weight $\kappa$.
The equivalence $\mathsf{T}_{\lambda'\leftarrow \lambda}:
\U^\chi_{(\lambda),\F}\operatorname{-mod}^{\underline{Q}}\xrightarrow{\sim} \U^\chi_{(\lambda'),\F}
\operatorname{-mod}^Z$ is given by $\operatorname{pr}_{\lambda'}(V\otimes\bullet)$,
where we write $\operatorname{pr}_{\lambda'}$ for the projection to
$\U^{(\lambda')}_{\chi,\F}\operatorname{-mod}^{\underline{Q}}$. The functor
$\mathsf{T}_{\lambda'\leftarrow \lambda}$ sends the parabolic Verma module
$\Delta^P_\F(\mu_w)$ to the parabolic Verma module $\Delta^P_{\F}(\mu'_w)$.
It is $S(\g_{\F}^{(1)})$-linear and hence $S(\mathfrak{m}_{\F}^{-(1)})$-linear. Since
$W^\chi_\F(?)$ is the fiber of the $S(\mathfrak{m}_{\F}^{-(1)})$-flat
module $\Delta^P_{\F}(?)$ at $\chi|_{\mathfrak{m}^-}$,  it follows that $\mathsf{T}_{\lambda'\leftarrow \lambda}
W^\chi_\F(\mu_w)=W^\chi_\F(\mu'_w)$.
\end{proof}

Now let us proceed to the braid group action.

Suppose that $\lambda+\rho$ is regular.
The affine braid group $\Br^{a}$ acts on $D^b(\U^\chi_{(\lambda),\F}\operatorname{-mod}^{{\underline{Q}}})$
 by the so called wall-crossing functors, see \cite[Section 2.1]{BMR_sing}. Namely, for $i=0,\ldots,r$, let $\mu_i\in \mathfrak{X}$ be such that $\mu_i$ lies on exactly one wall of the anti-dominant $p$-alcove, and this wall corresponds to the simple affine root $\alpha_i$. Then $T_i$
is given by the complex $\operatorname{id}\rightarrow
\mathsf{T}_{\mu^\circ\leftarrow \mu_i}\mathsf{T}_{\mu_i\leftarrow \mu^\circ}$,
where the target functor is in homological degree $0$.

We denote the wall-crossing functor corresponding
to $T_x\in \Br^a$ for $x\in W^{a}$ again by $T_x$. This functor is right t-exact.
The translation equivalences $\U^\chi_{(\lambda),\F}\operatorname{-mod}^{\underline{Q}}
\xrightarrow{\sim} \U^\chi_{(\lambda'),\F}\operatorname{-mod}^{\underline{Q}}$ intertwine the actions
of $\Br^{a}$. The $\Br^a$-action on the category induces an action of $W^{a}$ on $K_0(\U_{(\lambda),\F}^{\chi}\operatorname{-mod}^{{\underline{Q}}})$.

\begin{Lem}\label{Lem:W_aff_action}
For $x\in W^{a}$, let $x_-$ denote
the shortest element in $xW_P$ and $x_+$ be the longest
element in $xW_P$. Then
\begin{equation}\label{eq:Weyl_Weyl_action}
x[W^\chi_\F(\mu^\circ)]=(-1)^{\ell(x)-\ell(x_-)}[W^\chi_\F(\mu_{x_+})].
\end{equation}
\end{Lem}
\begin{proof}
Note that for $s\in W_P$, we have $T_s\Delta^P_{\F}(w_{0,P}\mu^\circ)=\Delta^P_\F(w_{0,P}\mu^\circ)[1]$.
Therefore in the proof it is enough to assume that $x$ is shortest in $xW_P$. The proof is
by induction on $\ell(x)$. If $x\neq 1$, then there is a simple affine reflection $s$
such that $sxw_{0,P}\in W^{a,P}$ and $\ell(sx)=\ell(x)-1$. In this case, $$T_s \Delta^P_{\F}(\mu_{xw_{0,P}})\cong
\Delta^P_\F(\mu_{sxw_{0,P}})\Rightarrow T_s W^\chi_{\F}(\mu_{xw_{0,P}})\cong W^\chi_\F(\mu_{sxw_{0,P}}).$$
So if (\ref{eq:Weyl_Weyl_action}) holds for $sx$, then it holds for $x$. This gives the induction step
and finishes the proof.
\end{proof}

Additional properties of the braid group action will be established or recalled
as needed.

\section{Derived localization in positive characteristic}\label{S_der_loc}
\subsection{Notation and content}\label{SS_loc_notation}
The notation $G, (e,h,f), \nu, \chi$ has the same meaning as
in Section \ref{SS_notation_basic}.
Let ${\underline{Q}}$ stand for the centralizer of $(e,h,f)$ in $\underline{G}$. Also recall
the one-parameter subgroup $\gamma:\mathbb{G}_m\rightarrow G$
associated to $h$.

Let $\B$ denote the flag variety for $\g$, i.e., $\B:=G/B$,
where $B$ is a Borel subgroup.
We write $U$ for the unipotent radical of $B$. We write $\upsilon$ for the projection
$G/U\rightarrow G/B$.

Let $\bor$ be the Lie algebra of $B$ and $\n$ its nilpotent radical,
the Lie algebra of $U$. We also pick a Cartan subalgebra $\h\subset \bor$. Let $T$ be the maximal
torus in $B$ corresponding to $\h$ and $W$ be the Weyl group.

We consider the cotangent bundle $\tilde{\Nilp}:=T^*\B(=G\times^B \n)$. It is a resolution
of singularities for the nilpotent cone $\Nilp\subset \g$. We can also consider the Grothendieck
simultaneous resolution $\tilde{\g}:=G\times^B \bor$ of $\g_\h:=\g^*\times_{\h^*/W}\h^*$. The scheme
$\tilde{\g}$ is  smooth  over $\h$ and its fiber over $0$ is $\tilde{\Nilp}$. We write
$\pi$ for the Springer morphism $\tilde{\Nilp}\rightarrow \Nilp$
and also for $\tilde{\g}\rightarrow \g_\h$.

As before, we can reduce all objects introduced above mod $p\gg 0$.
We will write $\B_\chi$ for the corresponding Springer fiber, it consists
of all Borel subalgebras $\mathfrak{b}_\F^{(1)}\subset \g_\F^{(1)}$
such that $\chi$ vanishes on $[\mathfrak{b}_\F^{(1)},
\mathfrak{b}_\F^{(1)}]$. We view $\B_\chi$ as a subvariety of $\tilde{\Nilp}^{(1)}_\F$.

We set $(\tilde{\g}_{\F}^{(1)})_\h:=\tilde{\g}_\F^{(1)}\times_{\h_\F^{(1)*}}\h_\F^{*}$.
Define $(\g_{\F}^{(1)})_\h$ similarly.

We now describe the content of this section. In Section \ref{SS_derived_loc} we
recall results about derived localization in positive
characteristic from \cite{BMR}. Then we describe splitting bundles for
Azumaya algebras that arise in the derived localization theorem in
Section \ref{SS_splitting}, also following \cite{BMR}. Finally, in
Section \ref{SS_splitting_equivar} we discuss equivariant structures
on the splitting bundles. The latter has not appeared in the literature but
is pretty standard.

\subsection{Derived localization equivalence}\label{SS_derived_loc}
Let $\lambda\in \h^*_{\F_p}/(W,\cdot)$ be such that $\lambda+\rho$ is regular
(recall that we consider the $\rho$-shifted action of $W$).
Pick $\mu\in \h^*_{\Z}$ such that $W\cdot \mu$ mod $p$ coincides with $\lambda$.
Let $\mathsf{t}_\mu$ denote the translation by $\mu$ in $\h^*_\F$. Note that it
intertwines the Artin-Schreier map $\h^*_\F\rightarrow \h_\F^{*(1)}$.


Then we can consider the sheaf $\tilde{D}_{\B_\F}:=\upsilon_* (D_{G_\F/U_\F})^{T_\F}$.
We have $\Gamma(\tilde{D}_{\mathcal{B}_\F})=\U_{\F,\h}:=\U_{\F}\otimes_{\F[\h^*]^W}\F[\h^*]$, and $R^i\Gamma(\tilde{D}_{\B_\F})=0$ for $i>0$, see \cite[Proposition 3.4.1]{BMR}. We can view $\tilde{D}_{\mathcal{B}_\F}$ as
an Azumaya algebra on $(\tilde{\g}_{\F}^{(1)})_\h$,
\cite[Section 3.1.3]{BMR}.
We write $\F[\h^*]^{\wedge_\mu}$ for the completion of
$\F[\h^*]$ at $\mu$. We set
$$(\tilde{\g}_{\F}^{(1)})_\h^{\wedge_\mu}:=(\tilde{\g}_{\F}^{(1)})_\h\times_{\h^*_\F}
\operatorname{Spec}(\F[\h^*]^{\wedge_\mu}),
\tilde{D}_{\mathcal{B}_\F}^{\wedge_\mu}:=\mathsf{t}_\mu^*
\left(\tilde{D}_{\mathcal{B}_\F}
|_{(\tilde{\g}_{\F}^{(1)})_\h^{\wedge_\mu}}\right).$$ Then
$\tilde{D}_{\mathcal{B}_\F}^{\wedge_\mu}$ is an Azumaya algebra
on $(\tilde{\g}_{\F}^{(1)})_\h^{\wedge_0}$. We have
$R\Gamma(\tilde{D}_{\mathcal{B}_\F}^{\wedge_\mu})=\U_\F^{\wedge_\lambda}$, where we write $\U_\F^{\wedge_\lambda}$ for $\U_{\F}\otimes_{\F[\h^*]^W}(\F[\h^*]^W)^{\wedge_\lambda}$
with the second factor being the completion at $\lambda$.

So it makes
sense to consider the derived global section functor
$$R\Gamma^\mu: D^b(\Coh(\tilde{D}_{\mathcal{B}_\F}^{\wedge_\mu}))
\rightarrow D^b(\U^{\wedge_\lambda}_{\F}\operatorname{-mod}).$$

The following is
\cite[Theorem 3.2]{BMR}.

\begin{Thm}\label{Thm:derived_localization}
The functor $R\Gamma^\mu$ is an equivalence.
\end{Thm}

\subsection{Splitting bundle}\label{SS_splitting}
It turns out that $\U^{\wedge_{-\rho}}_{\h,\F}$ is an
Azumaya algebra on $(\g^{(1)}_{\F})^{\wedge_0}_\h$. Moreover, $\tilde{D}^{\wedge_{-\rho}}_{\B_\F}=\pi^* \U^{\wedge_{-\rho}}_{\h,\F}$, see \cite[Proposition 5.2.1]{BMR}. The Azumaya algebras $\tilde{D}^{\wedge_\mu}_{\B_\F}$ and $\tilde{D}^{\wedge_{-\rho}}_{\B_\F}$ are Morita equivalent via the $\tilde{D}^{\wedge_\mu}_{\B_\F}$-$\tilde{D}^{\wedge_{-\rho}}_{\B_\F}$-bimodule
$\operatorname{Fr}_{\h,*}\left(\mathcal{O}(\mu+\rho)\otimes
\tilde{D}^{\wedge_{-\rho}}_{\B_\F}\right)$. Here we write $\operatorname{Fr}_\h$ for the morphism
$\tilde{\g}_\F\rightarrow (\tilde{\g}_{\F}^{(1)})_\h$ given by $(\Fr,\operatorname{id})$. Note that $\mathcal{O}(\mu+\rho)$ is a $G$-equivariant line bundle when $\mu+\rho\in \mathfrak{X}(T)$.

From here we  deduce that the restriction $\tilde{D}^{\wedge_{\mu,\chi}}_{\B_\F}$ of  $\tilde{D}^{\wedge_\mu}_{\B_\F}$ to $$(\tilde{\g}_{\F}^{(1)})_\h^{\wedge_\chi}:=(\tilde{\g}_{\F}^{(1)})_\h\times_{\g^{(1)*}_\F}
\g^{(1)*\wedge_\chi}_{\F}$$ splits. Here and below we write
$\g^{(1)*\wedge_\chi}_{\F}$ for $\operatorname{Spec}(\F[\g^{*(1)}]^{\wedge_\chi})$.
Note that since the Artin-Schreier map $\h_\F^*\rightarrow \h_\F^{*(1)}$ is unramified, we have a natural identification of
$(\tilde{\g}_{\F}^{(1)})_\h^{\wedge_\chi}$ with $\tilde{\g}_{\F}^{(1)\wedge_\chi}$.

Recall that a splitting bundle of an Azumaya algebra is unique up to a twist with a line bundle.
A choice of a splitting bundle $\mathcal{V}$ for $\tilde{D}^{\wedge_\mu}_{\B_\F}$
then gives rise to an abelian equivalence
\begin{equation}\label{eq:abel_p_1} \Coh_{\B_\chi}\left(\tilde{\g}_{\F}^{(1)}\right)\xrightarrow{\sim} \Coh_{\mathcal{B}_\chi}\left(\tilde{D}^{\wedge_\mu}_{\B_\F}\right), \quad \mathcal{F}\mapsto \mathcal{V}\otimes\mathcal{F}.
\end{equation}
Here we write $\Coh_{\B_\chi}$ for the category of all coherent sheaves that are set theoretically
supported on $\B_\chi$.

We have full embeddings
\begin{equation}\label{eq:full_embeddings_localization}
D^b\left(\Coh_{\B_\chi}(\tilde{\g}_{\F}^{(1)})\right)\hookrightarrow
D^b\left(\operatorname{Coh}\left((\tilde{\g}_{\F}^{(1)})_\h^{\wedge_0}\right)\right),
D^b(\U_{\F}\operatorname{-mod}_\lambda^\chi)
\hookrightarrow D^b(\U^{\wedge_\lambda}_{\F}\operatorname{-mod}),
\end{equation}
where we write $\U_{\F}\operatorname{-mod}_\lambda^\chi$ for the category
of $\U_\F$-modules with generalized HC character $\lambda$ and generalized $p$-character
$\chi$.  The claim that the functors are indeed full embeddings is pretty standard:
for example, $D^b\left(\Coh_{\B_\chi}(\tilde{\g}_{\F}^{(1)})\right)$ is identified
with the full subcategory in  $D^b\left(\Coh(\tilde{\g}_{\F}^{(1)})\right)$
of all objects with homology set theoretically supported on $\B_\chi$.
The functor in (\ref{eq:full_embeddings_localization}) is the identification of this full
subcategory with the similarly defined full subcategory in
$D^b\left(\operatorname{Coh}\left((\tilde{\g}_{\F}^{(1)})_\h^{\wedge_0}\right)\right)$.

Thanks to (\ref{eq:full_embeddings_localization}) we also have the derived equivalence
\begin{equation}\label{eq:der_p_1}
D^b\left(\Coh_{\B_\chi}(\tilde{\g}_{\F}^{(1)})\right)\xrightarrow{\sim} D^b(\U_{\F}\operatorname{-mod}_\lambda^\chi),
\quad \mathcal{F}\mapsto R\Gamma(\mathcal{V}\otimes \mathcal{F}).
\end{equation}

Let us now explain our choice of a splitting bundle.
A splitting bundle $\overline{\mathcal{V}}^\chi_{-\rho}$ for $\U^{\wedge_{\chi,-\rho}}_{\h,\F}$
on $(\g_{\F}^{(1)})_\h^{\wedge_\chi}$ is unique up to an isomorphism because
$(\g_{\F}^{(1)})_\h^{\wedge_\chi}$  is the spectrum of a  local ring.
We set
\begin{equation}\label{eq:splitting_Springer}
\mathcal{V}_\mu^\chi:=\left(\operatorname{Fr}_{\h*}(\mathcal{O}(\mu-(p-1)\rho)\otimes
\tilde{D}^{\wedge_{-\rho}}_{\B_\F})\right)^{\wedge_\chi}\otimes_{\pi^* (\U^{\wedge_{-\rho}}_{\h,\F})^{\wedge_\chi}}\pi^*\overline{\mathcal{V}}^\chi_{-\rho}.
\end{equation}
This is a splitting bundle for $\tilde{D}^{\wedge_\mu}_{\B_\F} $.

Finally, let us discuss a compatibility with braid group actions. The braid group $\Br^a$
acts on both $D^b(\Coh_{\B_\chi}(\tilde{\g}_\F^{(1)}))$ (see
\cite[Section 1.3]{BR}) and $D^b(\U_{\F}\operatorname{-mod}_\lambda^\chi)$ as was recalled in
Section \ref{SS_transl_braid}. The following proposition was proved
in \cite[Section 5.4]{Riche}.

\begin{Prop}\label{Prop:braid_equivariance}
Let $\mu+\rho$ lie inside the dominant alcove and $\rho\in \mathfrak{X}(T)$.
The functor $R\Gamma(\mathcal{V}_\mu^\chi(\rho)\otimes\bullet)$ is $\Br^a$-equivariant.
\end{Prop}

\subsection{Equivariance  of the splitting bundle}\label{SS_splitting_equivar}
Here we introduce a ${\underline{Q}}_\F$-equivariant structure on $\mathcal{V}^{\chi}_\mu$
and discuss its properties.

We start by treating the case of $\chi=0$.
Note that the fiber $\left(\overline{\mathcal{V}}^0_{-\rho}\right)_0$ of $\overline{\mathcal{V}}^0_{-\rho}$ at $0$
carries a natural $G_\F$-action, in fact, it is the Steinberg $G_\F$-module.

\begin{Lem}\label{Lem:G_equiv_bundle}
There is an extension of the $G_\F$-equivariant structure
from $\left(\overline{\mathcal{V}}^0_{-\rho}\right)_0$  to  $\overline{\mathcal{V}}^0_{-\rho}$.
\end{Lem}
\begin{proof}
The obstruction for the existence of a $G_\F$-equivariant structure
lies in the cohomology group $H^1_{G_\F}(1+\mathfrak{m})$, where $\mathfrak{m}$
is the maximal ideal in  $\F[(\g_{\F}^{(1)})_\h]^{\wedge_0}$
and we view $1+\mathfrak{m}$ as a group with respect to multiplication.
Indeed, this obstruction is an obstruction to lifting the homomorphism
$G_\F\rightarrow \operatorname{PGL}(\overline{\mathcal{V}}^0_{-\rho})$
to  $G_\F\rightarrow \operatorname{GL}(\overline{\mathcal{V}}^0_{-\rho})$
and such a lift is already fixed at the closed point.

Note that $1+\mathfrak{m}$ is a filtered group with associated graded
being the augmentation ideal in the positively graded algebra $\F[(\g_{\F}^{(1)})_\h]$.
The filtration is complete and separated so the equality $H^1_{G_\F}(1+\mathfrak{m})=0$ will follow once
we know $H^1_{G_\F}(\F[(\g_{\F}^{(1)})_\h])=0$.

First of all, we claim that
\begin{equation}\label{eq:cohom_G_vanish}
H^i_{G^{(1)}_\F}(\F[(\g_{\F}^{(1)})_\h])=0, \forall i>0.
\end{equation}
Note that $\F[(\g_{\F}^{(1)})_\h]=\F[\h^*]\otimes_{\F[\h^{*(1)}]}\F[\tilde{\g}_\F^{(1)}]$.
Since $\F[\h^*]$ is free over $\F[\h^{*(1)}]$, (\ref{eq:cohom_G_vanish}) will follow if we check that
$H^i_{G^{(1)}_\F}(\F[\tilde{\g}_\F^{(1)}])=0$ for $i>0$.
Note that
$$\F[\tilde{\g}_\F^{(1)}]\cong R\operatorname{Ind}_{B_\F^{(1)}}^{G_\F^{(1)}}\F[\mathfrak{b}_\F^{(1)}].$$
So we reduce to checking that
$H^i_{B^{(1)}_\F}(\F[\mathfrak{b}_\F^{(1)}])=0$ for all $i>0$.
This follows from the weight considerations for the $T^{(1)}_\F$-action. Indeed,
the $T_\F^{(1)}$-weights in both $\F[\mathfrak{b}_\F^{(1)}]$ and the
injective $B_\F^{(1)}$-module $\F[B_\F^{(1)}]$ are in the negative of
the root lattice. So $H^i_{B^{(1)}_\F}(\F[\mathfrak{b}_\F^{(1)}])=
H^i_{T^{(1)}_\F}(\F[\mathfrak{h}_\F^{(1)}])=0$. This proves
(\ref{eq:cohom_G_vanish}).

Using (\ref{eq:cohom_G_vanish}) we deduce that
$H^1_{G_\F}(\F[\g_\F^{(1)*}])=H^1_{G_1}(\F[\g_\F^{(1)*}])$,
where we write $G_1$ for the Frobenius kernel.
But in the cohomology group we have a trivial $G_1$-representation.
By the weight reasons, the 1st self-extensions of the trivial
one-dimensonal $G_1$-module vanish. Hence $H^1_{G_\F}(\F[\g_\F^{(1)*}])=\{0\}$,
which finishes the proof.
\end{proof}


Now we proceed to studying an equivariant structure in the case of general $\chi$. For technical reasons we need to work with a one-parameter
version of $\mathcal{V}^{\chi}_\mu$ that we going to introduce now.
We note that  $\U_{-\rho,\F}$ splits on $\F\chi$. The splitting bundle
is $\underline{\Delta}^{\F\chi}\left(W_{L,\F}((p-1)\rho)\otimes \F[z]\right)$, where we write $z$
for a coordinate on $\F\chi$ and $\underline{\Delta}^{\F\chi}$ for the induction
functor whose fiber at  $a\in \F$ is $\underline{\Delta}^{a\chi}$. We note that
$\underline{\Delta}^{\F\chi}\left(W_{L,\F}((p-1)\rho)\otimes \F[z]\right)$ carries a
natural ${\underline{Q}}_\F$-action. Besides, it also has an action of
$\F^\times$ via $\gamma$, note that this action rescales $z$.

The following lemma shows that we can extend
$\underline{\Delta}^{\F\chi}\left(W_{L,\F}((p-1)\rho)\otimes \F[z]\right)$ to a ${\underline{Q}}_\F\times \F^\times$-equivariant splitting bundle for the restriction of $\U^{\wedge_{-\rho}}_{\h,\F}$ to
$$(\g_{\F}^{*(1)})_\h^{\wedge_{\F\chi}}:=(\g_{\F}^{*(1)})_\h\times_{\g_{\F}^{(1)*}}
\operatorname{Spec}(\F[\g^{*(1)}]^{\wedge_{\F\chi}}).$$
Here $\F[\g^{*(1)}]^{\wedge_{\F\chi}}$ is the completion of $\F[\g^{(1)*}]$ with respect
to the ideal of the line $\F\chi$.

\begin{Lem}\label{Lem:equiv_splitting}
Let $R$ be an $\F[z]$-algebra,
$\mathfrak{m}\subset R$ be an ideal such that $R/\mathfrak{m}=\F[z]$
and $R$ is complete in the $\mathfrak{m}$-adic topology. Let $\mathsf{B}$  be an Azumaya $R$-algebra.
Let $\tilde{Q}_\F$ be an algebraic group with the following properties:
\begin{itemize}
\item $\tilde{Q}_\F^\circ$ is a torus,
\item $\tilde{Q}_\F/\tilde{Q}_\F^\circ$ is of order coprime to $p$.
\end{itemize}
Suppose that $\tilde{Q}_\F$ acts on $R$ pro-rationally by $\F$-linear automorphisms and
the algebra $\mathsf{B}$ is $\tilde{Q}_\F$-equivariant. We also suppose that $z$ gets rescaled
with a nontrivial character. Further, suppose that we have a $\tilde{Q}_\F$-equivariant
isomorphism $\mathsf{B}/\mathsf{B}\mathfrak{m}\cong \operatorname{End}_{\F[z]}(V_z)$ for some
finite rank free
$\F[z]$-module $V_z$ with a rational $\tilde{Q}_\F$-action. Then the following claims
hold:
\begin{enumerate}
\item There is a finite rank free $R$-module $\tilde{V}_z$ with a pro-rational $\tilde{Q}_\F$-action
such that $\tilde{V}_z/\mathfrak{m}\tilde{V}_z\cong V_z$ and $\operatorname{End}_R(\tilde{V}_z)\cong \mathsf{B}$, $\tilde{Q}_\F$-equivariant isomorphisms.
\item Suppose we are given
a free $R/(z)$-module $\tilde{V}$ with a rational $\tilde{Q}_\F$-action and $\tilde{Q}_\F$-equivariant isomorphisms
$\mathsf{B}/(z)\xrightarrow{\sim} \operatorname{End}_{R/(z)}(\tilde{V})$ and
$\tilde{V}/\mathfrak{m}\tilde{V}\xrightarrow{\sim} V_z/z V_z$ that are compatible
in the sense that the induced isomorphism
$\mathsf{B}/(z,\mathfrak{m})\xrightarrow{\sim} \operatorname{End}_{\F}(\tilde{V}/\mathfrak{m}V)$
comes from $\tilde{V}/\mathfrak{m}\tilde{V}\xrightarrow{\sim} V_z/z V_z$. Then we can find
$\tilde{V}_z$ as in (1) that comes with a $\tilde{Q}_\F$-equivariant isomorphism
$\tilde{V}_z/z\tilde{V}_z\xrightarrow{\sim} \tilde{V}$ compatible with the isomorphism
$\mathsf{B}/(z)\xrightarrow{\sim}\operatorname{End}_{R/(z)}(\tilde{V})$.
\end{enumerate}
\end{Lem}
\begin{proof}
Let us prove (1).
Suppose that we have constructed a lift $\tilde{V}_{k,z}$ of $V_z$ to $R/\mathfrak{m}^k$
with a rational $\tilde{Q}_\F$-action and a $\tilde{Q}_\F$-equivariant
isomorphism $\End_{R/\mathfrak{m}^k}(\tilde{V}_{k,z})
\xrightarrow{\sim} \mathsf{B}/\mathsf{B}\mathfrak{m}^k$.
Note that the set of lifts $(\tilde{V}_{k+1,z},\iota_{k+1})$ of any given isomorphism
$\iota_k:\operatorname{End}_{R/\mathfrak{m}^k}(\tilde{V}_{k,z})\xrightarrow{\sim}
\mathsf{B}/\mathsf{B}\mathsf{m}^k$ to
$\operatorname{End}_{R/\mathfrak{m}^{k+1}}(\tilde{V}_{k+1,z})\xrightarrow{\sim}
\mathsf{B}/\mathsf{B}\mathfrak{m}^{k+1}$ is an affine bundle over
$\mathbb{A}^1$. If $\iota_k$ is $\tilde{Q}_\F$-equivariant,
then $\tilde{Q}_\F$ acts on the affine bundle of lifts by affine transformations.
The fixed point locus is also an affine bundle.
But any affine bundle over $\mathbb{A}^1$ can be  trivialized,
which implies the existence of a $\tilde{Q}_\F$-equivariant lift $\iota_{k+1}$.

The argument above also proves (2).
\end{proof}

Let $\overline{\mathcal{V}}_{-\rho}^{\F\chi}$ denote the resulting
$\tilde{Q}_\F\times \F^\times$-equivariant splitting bundle for the restriction of
$\U^{\wedge_{-\rho}}_\F$ to $(\g_{\F}^{(1)})_\h^{\wedge_{\F\chi}}$.
This gives rise to a splitting bundle for the restriction of $\tilde{D}^{\wedge_\mu}_{\B_\F}$ to $$(\tilde{\g}_{\F}^{(1)})_\h^{\wedge_{\F\chi}}:=(\tilde{\g}_{\F}^{(1)})_\h
\times_{(\g_{\F}^{*(1)})_\h}(\g_{\F}^{*(1)})_\h^{\wedge_{\F\chi}}$$
 given by the formula analogous to (\ref{eq:splitting_Springer}).
This splitting bundle will be denoted by $\mathcal{V}^{\F\chi}_\mu$.
Its restriction to $(\tilde{\g}_{\F}^{(1)})_\h^{\wedge_\chi}$
is $\mathcal{V}^\chi_\mu$. From (2) of Lemma \ref{Lem:equiv_splitting}, it follows that the ${\underline{Q}}_\F$-equivariant structure on $\mathcal{V}^0_\mu$ coincides with the restriction of the $G_\F$-equivariant structure.


\begin{Lem}\label{Lem:split_properties}
We have the following:
\begin{enumerate}
\item The restriction of $\mathcal{V}_\mu^0$ to $\B_\F^{(1)}$ is $G_\F$-equivariantly isomorphic  to  $\Fr_{\B,*}\mathcal{O}(\mu)$.
\item The class of $\mathcal{V}_\mu^\chi$ in $K_0^{{\underline{Q}}}(\B_{\chi})$
coincides with the pull-back of $[\Fr_{\B,*}\mathcal{O}(\mu)]$ under the inclusion
$\B_{\chi}\hookrightarrow \B_\F^{(1)}$.
\end{enumerate}
\end{Lem}
Note that this lemma is an equivariant version of \cite[Lemma 6.5.2]{BMR}.
\begin{proof}
Let us prove (1). Note that we can replace $G_\F$ with a cover. So we can assume that
$G_\F$ is the product of a torus and a simply connected semisimple group.
It is sufficient to consider the case of a torus and the case of a simply connected semisimple
group separately. The torus case is trivial.

Let $G_\F$ be semisimple and simply connected.
Both $\mathcal{V}_\mu^0|_{\B_\F^{(1)}}$ and $\Fr_{\B,*}\mathcal{O}(\mu)$
are $G_\F$-equivariant splitting bundles for $D^\mu_{\B_\F}|_{\B_\F^{(1)}}$.
So they differ by a twist with a $G_\F$-equivariant line bundle $\mathcal{L}$ on $\B_\F^{(1)}$.
Note that every line bundle on $\B_\F^{(1)}$ has a unique $G_\F$-equivariant structure.

The abelian group $K_0(\B_\F^{(1)})$ is torsion free.
Therefore a class of a vector bundle is not a zero divisor in the ring $K_0(\B_\F^{(1)})$.
Moreover,  it is a standard fact that $\operatorname{Pic}(\B_\F^{(1)})$ embeds into $K_0(\B_\F^{(1)})$.
So in order to check that $\mathcal{L}$ is equivariantly  trivial it suffices to show that  the classes of
$\mathcal{V}_\mu^0|_{\B_\F^{(1)}}$ and $\Fr_{\B,*}\mathcal{O}(\mu)$
in the usual (i.e., non-equivariant) $K_0$-group are the same.
This is true for $\mu=-\rho$: it is easy to see that
in that case  both  bundles are $\mathcal{O}_{\B_\F^{(1)}}(-\rho)^{\oplus p^{\dim \B}}$.
And for an arbitrary element $\mu$, we have
\begin{equation}\label{eq:K0_class_equality}
[\mathcal{V}_\mu^0|_{\B_\F^{(1)}}]=[\mathcal{O}_{\B_\F^{(1)}}(-\rho)^{\oplus p^{\dim G/B}}][\mathcal{O}(\frac{\mu+\rho}{p})]\end{equation}
in $K_0(\B_\F^{(1)})$. By $[\mathcal{O}(\frac{\mu+\rho}{p})]$ we mean the $p$th
root of $[\mathcal{O}(\mu+\rho)]$, the latter class is unipotent so its $p$th root
makes sense in $K_0(\B_\F^{(1)})\otimes_{\Z}\mathbb{Q}$. On the other hand,
we have $[\Fr_{\B,*}\mathcal{O}(p\mu'-\rho)]=\mathcal{O}(\mu'-\rho)^{\oplus
\dim \mathcal{B}}$ for any $\mu'\in \mathfrak{X}(T)$. It follows that
$[\Fr_{\B,*}\mathcal{O}(\mu)]$ equals to the right hand side of (\ref{eq:K0_class_equality}).
This finally implies a $G_\F$-equivariant isomorphism $\mathcal{V}_\mu^0|_{\B_\F^{(1)}}\cong\Fr_{\B,*}\mathcal{O}(\mu)$.

Let us prove (2). Thanks to the existence of $\mathcal{V}^{\F\chi}_\mu$, in
the proof we can replace $\mathcal{V}_\mu^\chi$ with $\mathcal{V}_\mu^0|_{\mathcal{B}_\chi}$.
Now the claim follows from (1).
\end{proof}

\section{Tilting bundle on $\tilde{\g}$}\label{S_tilting}
\subsection{Notation and content}\label{SS_tilting_notation}
In this section, our base field is $\C$.
The notation $G,(e,h,f),\rho,$ $\chi,\gamma, \g,\B, \tilde{\Nilp}, \tilde{\g}$
has the same meaning as in Section \ref{SS_loc_notation}.
We assume $G$ is semisimple and simply connected. Recall that we identify $\g$
with $\g^*$ using the Killing form.

Let $P$ denote  a parabolic subgroup of $G$ containing $B$. Set  $\Pcal:=G/P, \tilde{\Nilp}_P:=T^*\Pcal$.
We will write $Z$ for $G\times^B \mathfrak{m}$, $\iota$
for the natural inclusion $Z\hookrightarrow \tilde{\Nilp}$
and $\varpi$ for the natural projection $Z\twoheadrightarrow \tilde{\Nilp}_P$.

We consider the following versions of the Steinberg varieties:
$\St_{\h}:=\tilde{\g}\times_{\g}\tilde{\g}$,
$\St_{B}:=\tilde{\g}\times_{\g}\tilde{\N}$, $\St_{0}:=\tilde{\N}\times^L_{\g}\tilde{\N}$.
The first two are genuine varieties, while the last one is
a derived scheme. This makes sense because the codimensions of $\St_{\h}$ in $\tilde{\g}\times\tilde{\g}$
and of $\St_{B}$ in $\tilde{\g}\times \tilde{\N}$ are equal to $\dim \g$ -- so the derived schemes
are the same as usual schemes. Moreover, these schemes are generically reduced
hence reduced.  On the other hand,
the codimension of $\St_{0}$ in $\tilde{\N}\times \tilde{\N}$ is less than
$\dim \g$ so we need to consider the derived scheme.

Set $\Orb:=Ge$. We write $\B_e$ for the Springer fiber $\tilde{\g}\times^L_{\g}\{e\}$ that we
view with its natural derived scheme structure. Consider the Slodowy slice
$S=e+\mathfrak{z}_{\g}(f)\subset \g$, it is transverse to $\Orb$.
The reductive group $Q:=Z_G(e,h,f)$ acts on $S$.
We consider the  $\C^\times$-action on $S$ given by $t.s:=t^{-2}\gamma(t)s$.
It commutes with $Q$.

The purpose of this section is to recall some results from \cite{BM}.
In Section \ref{SS_tilting_gen} we record some generalities about tilting
generators for categories of coherent sheaves. Then in  Section
\ref{SS_tilting_BM} we recall some properties of a tilting generator
for $\tilde{\g}$ constructed and studied in \cite{BM}.

\subsection{Generalities on tilting generators}\label{SS_tilting_gen}
Let $X$ be a smooth Calabi-Yau algebraic variety with a projective morphism
to an affine variety. Recall that a
vector bundle $\Tilt$ on $X$ is called
a {\it tilting generator} (for $\operatorname{Coh}(X)$ or simply for $X$)
if it has no higher self-extensions and the algebra $\End(\Tilt)$ has finite homological dimension.
Note that $R\Gamma(\mathcal{T}\otimes\bullet):D^b(\Coh(X))\rightarrow D^b(\End(\Tilt)\operatorname{-mod})$
is an equivalence, \cite[Proposition 2.2]{BK}.
This equivalence defines a new t-structure on $D^b(\Coh(X))$, where $\Tilt^*$
is a projective generator and $R\Gamma(\mathcal{T}\otimes\bullet)$ is t-exact.

\begin{defi}
We say that this  t-structure is defined by $\Tilt^*$.
\end{defi}

Set $\A:=\End(\Tilt)$.

\begin{Lem}\label{Lem:tilting_gen_prop}
Let $\Tilt$ be a tilting generator on $X$. Then $\Tilt^*$ is also a tilting generator.
Moreover, $\A$ is a Gorenstein algebra.
\end{Lem}
\begin{proof}
Clearly the higher self-extensions of $\Tilt^*$ vanish. Moreover, the algebra $\End(\Tilt^*)=\A^{opp}$
has finite homological dimension. Indeed, $\A$ is a finitely generated module over a finitely
generated commutative algebra. Hence $\A$ and $\A^{opp}$ are Noetherian. By \cite[Theorem 3.2.7]{Weibel},
every finitely generated flat (left or right) $\A$-module is projective. Hence the homological dimension
of $\A$ equals to the maximum of all $n$ such that $\operatorname{Tor}^{\A}_n(M,N)\neq 0$
for finitely generated left $\A$-module $M$ and right $\A$-module $N$. The latter clearly
is clearly symmetric with respect to $\A$ vs $\A^{opp}$ and hence
equals also to the homological dimension of $\A^{opp}$.


Now we proceed to proving that the algebra $\A$ is Gorenstein. Recall that this means,
by definition, that $\A$ has finite injective dimension both as a left and as a right
$\A$-module. It is enough to only consider the left module case. What we need to prove
is that there is an integer $n$ such that $\operatorname{Ext}^i_\A(M,\A)=0$ for all
$M\in \A\operatorname{-mod}$ and $i>n$. Let $d$ denote the homological dimension of
$\A$. We write $\varphi$ for the derived equivalence $R\Gamma(\mathcal{T}\otimes\bullet):
D^b(\Coh(X))\xrightarrow{\sim} D^b(\A\operatorname{-mod})$. Note that $\varphi(\Tilt^*)=\A$.
Hence
$$\operatorname{Ext}^i_\A(M,\A)=\Hom_{D^b(\Coh(X))}(\varphi^{-1}(M)[-i],\Tilt^*).$$
By the first paragraph of the proof, the homology of $\varphi^{-1}(M)$ are in homological degrees $0,\ldots,d$. And for a
coherent sheaf $\mathcal{F}$ we have $\operatorname{Ext}^j_{\Str_X}(\mathcal{F},\Tilt^*)=0$ for
$j>\dim X$. So we have $\operatorname{Ext}^i_\A(M,\A)=0$ for $i>d+\dim X$.
\end{proof}

Now let $Y$ be an affine smooth algebraic variety and let $X\rightarrow Y$ be  a projective morphism.
Suppose that $\A$ is flat over $\C[Y]$. The derived scheme $X\times^L_Y X$ comes with a vector bundle
$\Tilt\otimes \Tilt^*$ so that
\begin{equation}\label{eq:tilting_equiv} R\Gamma(\Tilt\otimes \Tilt^*,\bullet):
D^b(\Coh(X\times^L_Y X))\rightarrow D^b(\A\otimes_{\C[Y]}\A^{opp}\operatorname{-mod})
\end{equation}
is an equivalence. Note that it maps $\Tilt^*\otimes \Tilt$ to $\A\otimes_{\C[Y]}\A^{opp}$.

Here is a basic property of this equivalence. Note that $D^b(\Coh(X\times^L_Y X))$ is a monoidal
category with respect to convolution of coherent sheaves.

\begin{Lem}\label{Lem:equiv_monoidal}
The equivalence (\ref{eq:tilting_equiv})
is monoidal with respect to the convolution on $D^b(\Coh(X\times^L_Y X))$
and the tensor product of bimodules on $D^b(\A\otimes_{\C[Y]}\A^{opp}\operatorname{-mod})$.
\end{Lem}

There is also a  module analog of this lemma. Namely, pick $y\in Y$ and let $X_y$ be the fiber
of $X$ over $y$ viewed as a derived scheme. Then the functor $R\Gamma(\mathcal{T}\otimes\bullet)$
defines an equivalence $D^b(\Coh(X_y))\xrightarrow{\sim} D^b(\A_y\operatorname{-mod})$, where $y$
is the fiber of $\A$ over $y$.   This equivalence is compatible with the equivalence
from Lemma \ref{Lem:equiv_monoidal}.

\subsection{Bezrukavnikov-Mirkovic tilting bundle}\label{SS_tilting_BM}
We will need a $G\times \C^\times$-equivariant vector
bundle $\Tilt_{\h}$ on $\tilde{\g}$  with remarkable properties that was constructed in  \cite{BM}.
Set $\Tilt:=\Tilt_\h\otimes_{\C[\h]}\C_0$.

Here are two crucial properties of this bundle established in \cite[Section 2.5]{BM} that we will need:
\begin{Lem}\label{Lem:tilting_properties}
The following claims are true:
\begin{enumerate}
\item The bundle $\Tilt_\h$ is a tilting generator and $\operatorname{End}(\mathcal{T}_\h)$ is flat over
$\g$. Moreover, $\Tilt$ is a tilting generator
for $\tilde{\Nilp}$.
\item The bundle $\Tilt_\h$ is defined over a finite localization of $\Z$.
\end{enumerate}
\end{Lem}

We write $\A_\h$ for $\operatorname{End}(\Tilt_\h)$ and $\A$
for $\operatorname{End}(\Tilt)$. Note that, by the construction,
the algebras $\A_\h,\A$ are Gorenstein  by Lemma \ref{Lem:tilting_gen_prop}.

Thanks to (1) of Lemma \ref{Lem:tilting_properties}, we have the following derived equivalence
$$R\Gamma(\Tilt_\h\otimes\bullet): D^b(\Coh(\tilde{\g}))\xrightarrow{\sim} D^b(\A_\h\operatorname{-mod}),$$
as well as the similarly defined equivalences between categories of equivariant objects  with respect
to algebraic subgroups of $G$.

We record some related equivalences that we will need below. Note that $\A=\A_\h\otimes_{\C[\h^*]}\C_0$.

First of all, since $\St_{\h,\h}, \St_{\h,0}$ are complete intersections in $\tilde{\g}\times
\tilde{\g}$ and $\tilde{\g}\times \tilde{\Nilp}$, respectively, we have the following derived
equivalences:
\begin{align}\label{eq:non_comm_Springer_equiv}
&D^b(\Coh^G(\St_{B}))\xrightarrow{\sim} D^b(\A_\h\otimes_{\C[\g]}\A^{opp}\operatorname{-mod}^G),\\\label{eq:non_comm_Springer_equiv1}
&D^b(\Coh^G(\St_\h))\xrightarrow{\sim} D^b(\A_\h\otimes_{\C[\g]}\A_\h^{opp}\operatorname{-mod}^G).
\end{align}

Note that (\ref{eq:non_comm_Springer_equiv1}) is an equivalence of monoidal categories and
then (\ref{eq:non_comm_Springer_equiv}) is an equivalence of module categories. We have
equivalences of module categories:

\begin{equation}\label{eq:non_comm_Springer_fiber}
D^b(\Coh^{Z_G(e)}_{\B_e}(\tilde{\g}))\xrightarrow{\sim} D^b(\A_\h\operatorname{-mod}^{Z_G(e)}_e),
D^b(\Coh^{Z_G(e)}(\B_e))\xrightarrow{\sim} D^b(\A_{\h,e}\operatorname{-mod}^{Z_G(e)}).
\end{equation}
Moreover, we can replace the equivariance with
respect to $Z_G(e)$ with that with respect to any algebraic subgroup of $Z_G(e)$.
Here and below we write $\A_{\h,e}$ for the fiber of $\A_\h$ at $e$.

Now we explain a Koszulity property.
Consider the restriction $\A_\h|_{S}:=\C[S]\otimes_{\C[\g]}\A_\h$. This algebra is acted on by
$Q\times \C^\times$.  The following is one of the main results of \cite{BM},
see Section 5.5 there.

\begin{Thm}\label{Thm:Koszulity}
There is a Koszul grading on $\A_{\h}|_S$ compatible with the grading on $\C[S]$.
\end{Thm}

\begin{Rem}\label{Rem:Koszulity}
Note that the Koszul grading on $\A_\h|_{S}$ comes from a $\C^\times$-equivariant
structure on the restriction of $\Tilt$ to $S\times_{\g}\tilde{\g}$,
see \cite[Section 5.5]{BM}. The group
$Q\times\C^\times$ acts on $\A_{\h}|_S$ because $\Tilt|_{S\times_{\g}\tilde{\g}}$
is $Q\times \C^\times$-equivariant. We claim that
we can choose the grading to be $Q\times \C^\times$-stable.

Let $I$ denote the intersection of the kernels of the irreducible
$\A_\h|_S$ modules supported at $0\in S$.
Note that the descending filtration $(\A_{\h}|_S)_{\geqslant d}$
on $\A_\h|_{S}$ coming from a Koszul grading does not depend on the
choice of a Koszul grading: $(\A_{\h}|_S)_{\geqslant d}=I^d$. Any grading
splitting this filtration is Koszul, because the corresponding graded algebra
is isomorphic to the associated graded with respect to the filtration above.

Each $(\A_{\h}|_S)_{\geqslant d}$ is a $Q\times \C^\times$-stable $\C[S]$-submodule.
For each $d>0$, the group $F_d$ of $\C[S]$-linear automorphisms
of $\A_{\h}|_S/ (\A_{\h}|_S)_{\geqslant d}$ that are the identity on the associated
graded algebra is algebraic and unipotent. The set $\mathcal{G}_d$ of  gradings that split the filtration
is a torsor over $F_d$.
The group $Q\times \C^\times$ normalizes $F_d$ and acts on $\mathcal{G}_d$ in a compatible way.
It follows that there is a fixed point. Moreover, the fiber of $\mathcal{G}_{d+1}$ over a fixed
point in $\mathcal{G}_d$ is an affine space with an affine action of $Q\times \C^\times$.
It follows that we can compatibly choose $Q\times \C^\times$-fixed points in each $\mathcal{G}_d$.
This choice gives a $Q\times \C^\times$-stable grading that splits the filtration  (hence
 Koszul) and is compatible with the grading on $\C[S]$.
\end{Rem}

Since $\Tilt_{\h}$ is defined over a finite localization of $\Z$, we can reduce it
mod $p$ for $p$ large enough. We will view the reduction, to be denoted by $\Tilt_{\h,\F}$,
as a vector bundle on $\tilde{\g}_\F^{(1)}$. This vector bundle is
still a tilting generator.  Now we discuss a connection between $\Tilt_{\h,\F}$
and the splitting bundles considered in Section \ref{S_der_loc}.

The following result is essentially \cite[Corollary 1.6.8]{BM}. We provide a proof
for reader's convenience.

\begin{Lem}\label{Lem:reduct_mod_p}
The indecomposable summands of $\Tilt_{\h,\F}$ restricted to the formal neighborhood
of $\B^{(1)}_\chi$ are precisely the indecomposable summands
of the bundle $\mathcal{V}^\chi_0(\rho)$, where
$\mathcal{V}^\chi_0$ is defined by  (\ref{eq:splitting_Springer}).
\end{Lem}
\begin{proof}
According to \cite[Section 1.5.1]{BM}, the following two conditions uniquely specify a tilting generator $\mathcal{V}$ on $\tilde{\g}^{(1)}_\F\times_{\g_\F^{*(1)}}\g_\F^{*(1)\wedge_\chi}$
 (up to changing the multiplicities
of the indecomposable summands):
\begin{itemize}
\item Braid positivity: the  action of the affine braid monoid on $D^b_{\B^{(1)}_\chi}(\Coh(\tilde{\g}^{(1)}_\F))$
is by right $t$-exact functors with respect to the t-structure given by $\mathcal{V}^*$.
\item Normalization: $\mathcal{O}$ is a direct summand in $\mathcal{V}$, equivalently,
$R\Gamma$ is exact in the t-structure given by $\mathcal{V}^*$.
\end{itemize}
The restriction of $\Tilt_{\h,\F}$ satisfies these two properties by the construction,
see \cite[Section 1.5.1]{BM}.
We need to show that $\mathcal{V}^\chi_0(\rho)$ does. The braid positivity follows
because $R\Gamma(\mathcal{V}^\chi_0(\rho)\otimes\bullet)$ intertwines the braid group
actions and the action on the category  $\U_\F\operatorname{-mod}^\chi_0$ is
by right $t$-exact functors, Proposition
\ref{Prop:braid_equivariance}. The normalization follows because $R\Gamma$ is, up to a t-exact category
equivalence, the translation functor to $-\rho$, \cite[Section 2.2.5]{BMR_sing}.
\end{proof}

Finally, let us  discuss a tilting bundle on $\tilde{\Nilp}_P$.
Set \begin{equation}\label{eq:tilting_parabolic}
\Tilt_P:=\varpi_*\iota^* (\Tilt(-\rho)).
\end{equation}

The following claim was established in \cite[Sections 4.1, 4.2]{BM}.

\begin{Lem}
The complex $\Tilt_P$ is a vector bundle in homological degree $0$.
Moreover, it is a tilting generator for $\tilde{\Nilp}_P$.
\end{Lem}

\section{Constructible realization}\label{S_constr_realiz}
\subsection{Notation and content}\label{SSS_constructible_notation}
We continue to work over $\C$.
The notation $G,\g,B,T,\rho,\bor,\h,\B$, $\Nilp$ has the same meaning as in
Section \ref{SS_loc_notation}, and $\St_\h,\St_B,\St_0$
have the same meaning as in Section \ref{SS_tilting_notation}.
As in Section \ref{SS_notation_basic}, we write $W^a$ for the extended affine
Weyl group $W\ltimes \Xfr(T)$.

Let $G^\vee$ be the Langlands dual group of $G$.
Consider the  Cartan and Borel
subalgebras $\h^\vee\subset \bor^\vee\subset \g^\vee$
so that $\h^\vee=\h^*$ and the positive roots for
$\bor^\vee$ are the positive coroots for $\bor$.
Let $I^\vee$ be the Iwahori subgroup
of $G^\vee$, the preimage of $B^\vee\subset G^\vee$ under the projection $G^\vee[[t]]\twoheadrightarrow
G^\vee$. Consider the affine flag variety
$\Fl$ for $G^\vee$, $\Fl=G^\vee((t))/I^\vee$.
Let $I^\circ$ denote the pro-unipotent radical of $I^\vee$, the kernel of $I^\vee
\twoheadrightarrow T^\vee$, where $T^\vee$ denotes the maximal torus in $G^\vee$ corresponding to
$\h^\vee$.

The goal of this section is to
review some constructions and results from \cite{B_Hecke}
as well as some modifications. In Section
\ref{SS_derived_equiv} we recall the main result from \cite{B_Hecke}
on an equivalence between two geometric categorifications of an affine
Hecke algebra. Then in Section \ref{SS_perv_equiv} we discuss the compatibility
of this equivalence with t-structures. In Section
\ref{SS_completed_cats} we will discuss completions of some of categories
involved. Finally, in  Section \ref{SS_braid} we discuss a  compatibility of equivalences from
\cite{B_Hecke} with affine group actions.

\subsection{Derived equivalence}\label{SS_derived_equiv}
We need to relate the equivariant coherent derived categories for the versions
of the Steinberg varieties introduced in Section \ref{SS_tilting_notation}
to equivariant constructible derived categories
for the affine flag variety $\Fl$.

On the coherent side, we consider the categories
$D^b_{\Nilp}(\Coh^G(\St_\h)),D^b(\Coh^G(\St_B))$, and
$D^b(\Coh^G(\St_0))$. The first category consists of the complexes of coherent sheaves with homology
set theoretically supported at the preimage of $\Nilp$ in $\St_\h$.
Note that $D^b_{\Nilp}(\Coh^G(\St_\h)),$ $
D^b(\Coh^G(\St_0))$ are tensor categories (with respect to convolution
of coherent sheaves),
and  $D^b(\Coh^G(\St_B))$ is a bimodule category with a left action of
$D^b_{\Nilp}(\Coh^G(\St_\h))$ and a right action of
$D^b(\Coh^G(\St_0))$.

We  consider the constructible equivariant derived categories
$D^b_{un}(I^\circ\backslash G^\vee((t))/I^\circ), D^b_{I^\circ}(\Fl),$ and $
D^b_{I^\vee}(\Fl)$, where the first category is that of monodromic equivariant
constructible sheaves with unipotent monodromy.
Again, $D^b_{un}(I^\circ\backslash G^\vee((t))/I^\circ), D^b_{I^\vee}(\Fl)$
are tensor categories with respect to the $!$-convolution, while
$D^b_{I^\circ}(\Fl)$ is a bimodule category.

\begin{Thm}[Theorem 1 in \cite{B_Hecke}]\label{Thm:derived_equiv}
We have tensor equivalences
$$\tau: D^b_{un}(I^\circ\backslash G^\vee((t))/I^\circ)\xrightarrow{\sim} D^b_{\Nilp}(\Coh^G(\St_\h)),
D^b_{I^\vee}(\Fl)\xrightarrow{\sim} D^b(\Coh^G(\St_0))$$
and a bimodule equivalence
$$\tau: D^b_{I^\vee}(\Fl)\xrightarrow{\sim} D^b(\Coh^G(\St_B)).$$
\end{Thm}

\begin{Rem}\label{Rem:der_equi_compat}
Note that we have  functors given by partially forgetting equivariance
$$D^b_{I^\vee}(\Fl)\rightarrow D^b_{I^\circ}(\Fl)\rightarrow
D^b_{un}(I^\circ\backslash G^\vee((t))/I^\circ).$$
We also have inclusions of derived schemes $\St_0\hookrightarrow \St_B\hookrightarrow \St_\h$.
They give rise to push-forward functors
$$D^b(\Coh^G(\St_0))\rightarrow D^b(\Coh^G(\St_B))
\rightarrow D^b_{\Nilp}(\Coh^G(\St_\h)).$$
These functors intertwine the equivalences $\tau$. For the first pair of  arrows
this follows from \cite[Lemma 44(b)]{B_Hecke} as discussed in the proof
of \cite[Corollary 45]{B_Hecke}. For the second pair of arrows this
follows from \cite[Lemma 43]{B_Hecke}.
\end{Rem}

\subsection{Perverse equivalence}\label{SS_perv_equiv}
On the derived category  $D^b_{I^\circ}(\Fl)$ we have the usual perverse $t$-structure with heart
$\Perv_{I^\circ}(\Fl)$ consisting of perverse sheaves. We want to compare it with the t-structure
on $D^b(\Coh^G(\St_B))$ coming from (\ref{eq:non_comm_Springer_equiv}). We assume that $G$ is semisimple and simply connected.

In fact,  equivalence (\ref{eq:non_comm_Springer_equiv}) is compatible with certain filtrations on the categories
indexed by nilpotent orbits in $\g$. Let us explain what filtrations we consider and
state the corresponding result about equivalences.

Let us start with  $D^b(\Coh^G(\St_B))$.
The Steinberg variety $\mathsf{St}_B$ maps to $\mathcal{N}$
via $\mathsf{St}_B\rightarrow \tilde{\Nilp}\rightarrow \Nilp$. For a nilpotent
orbit $\Orb\subset \g$ let $\mathsf{St}_{\leqslant \Orb}$ denote the preimage
of $\overline{\Orb}$ in $\mathsf{St}_{B}$. Then we can consider the full subcategory
$D^b_{\leqslant \Orb}(\Coh^G(\St_B))$ of all
complexes with cohomology supported on $\mathsf{St}_{\leqslant \Orb}$.
We also can  consider the quotient category
$$D^b_{\Orb}(\Coh^G(\St_B)):=
D^b_{\leqslant \Orb}(\Coh^G(\mathsf{St}_{B}))/
D^b_{<\Orb}(\Coh^G(\mathsf{St}_{B})).$$
We  have a similarly defined filtration
$$D^b_{\leqslant \Orb}(\A_\h\otimes_{\C[\g]} \A^{opp}\operatorname{-mod}).$$
The derived global section functor restricts to
$$D^b_{\leqslant \Orb}(\Coh^G(\St_B))
\xrightarrow{\sim}D^b_{\leqslant \Orb}(\A_\h\otimes_{\C[\g]} \A^{opp}\operatorname{-mod})$$
and so gives an equivalence
$$D^b_{\Orb}(\Coh^G(\mathsf{St}_{B}))
\xrightarrow{\sim}D^b_{\Orb}(\A_\h\otimes_{\C[\g]} \A^{opp}\operatorname{-mod}).$$
The target category has a natural t-structure whose heart is
the subquotient category $\A_\h\otimes_{\C[\g]} \A^{opp}\operatorname{-mod}_{\Orb}$.

Let us proceed to a filtration on $D^b_{I^\circ}(\Fl)$.
Recall that the simples in $\operatorname{Perv}_{I^\circ}(\Fl)$ are indexed by the elements of the affine
Weyl group $W^{a}$. We have the two-sided cell filtration on $W^a$.
By a result of Lusztig, \cite{Lusztig_affine}, the two-sided cells in $W^{a}$ are in a natural one-to-one
correspondence with the nilpotent orbits in $\g$. So we can consider the category
$D^b_{I^\circ,\leqslant \Orb}(\Fl)$ of all objects with perverse homology in the Serre
subcategory $\operatorname{Perv}_{I^\circ,\leqslant \Orb}(\Fl)$ spanned by the simples from two-sided
cells corresponding to orbits contained in $\overline{\Orb}$. Again we have the
quotient $D^b_{I^\circ, \Orb}(\Fl)$ and its heart $\operatorname{Perv}_{I^\circ,\Orb}(\Fl)$.

We consider $D^b(\Coh^G(\St_B))$ with the  t-structure given
by $\Tilt_\h^*\otimes \Tilt$. We also consider
the induced t-structures on the subquotient categories $D^b_{\Orb}(\Coh^G(\St_B))$.

A final ingredient to state the result is as follows. Define
$D^{b,\leqslant 0}_{perv}(\A_\h\otimes_{\C[\g]}\A^{opp}\operatorname{-mod}^G)$ as the full
subcategory of $D^b(\A_\h\otimes_{\C[\g]}\A^{opp}\operatorname{-mod}^G)$
consisting of all objects $M$ such that $\dim \operatorname{Supp}H^i(M)\leqslant
\dim \Nilp-2i$.

\begin{Thm}[Theorems 54,55 in \cite{B_Hecke}]\label{Thm:perv_equiv}
For each nilpotent orbit $\Orb$, the equivalence
$$\tau: D^b_{I^\circ}(\Fl)\xrightarrow{\sim}D^b(\Coh^G(\St_{B}))$$
restricts to an equivalence $D^b_{I^\circ,\leqslant \Orb}(\Fl)\xrightarrow{\sim}D^b_{\leqslant \Orb}(\Coh^G(\St_{B}))$.
Moreover, for the induced equivalence
$$\tau_\Orb: D^b_{I^\circ,\Orb}(\Fl)\xrightarrow{\sim}
D^b(\Coh^G_{\Orb}(\St_B))$$
we have that $\tau[\frac{1}{2}\operatorname{codim}_{\Nilp}\Orb]$ is t-exact (with respect to
the perverse t-structure on the source and the t-structure on the target described in
the previous paragraph). Finally,  the image under $\tau$ of the negative part
$D^{b,\leqslant 0}_{I^\circ}(\Fl)$ of the perverse t-structure
coincides with $D^{b,\leqslant 0}_{perv}(\A_\h\otimes_{\C[\g]}\A^{opp}\operatorname{-mod}^G)$.
\end{Thm}

\begin{Rem}\label{Rem:correction}
We use this opportunity to correct the statement of \cite[Theorem 54]{B_Hecke}.
We use the notation of the present paper.
Everywhere in part (a) of that Theorem the expression $\A_\h\otimes_{\C[\g]} \A$
should be replaced by $\A_\h\otimes_{\C[\g]} \A^{opp}$, where the equivalence
$D^b(\A_\h\otimes_{\C[\g]} \A^{opp}\operatorname{-mod})\xrightarrow{\sim}
D^b(\Coh^G(\St_B))$ given by the tilting bundle
$\Tilt^*_\h\otimes \Tilt$ as above. The proof of this corrected statement is parallel to that
of \cite[Theorem 6.2.1]{BM} as asserted in \cite{B_Hecke}, however, one needs to take into account
that the images of the standard generators $T_\alpha$ under the right action
of $\Br^a$ on $D^b(\Coh^G(\St_B))$ are inverse to the image of $T_\alpha$ obtained by
viewing $\St$ as a base change of $\tilde{\g}\times \tilde{\Nilp}$.
\end{Rem}

\begin{Rem}\label{Rem:perv_equiv_compat}
Note that the simples in the hearts of t-structures of
$$D^b_{I^\vee}(\Fl), D^b_{I^\circ}(\Fl), D^b_{un}(I^\circ\backslash G^\vee((t))/I^\circ)$$
are the same, the functors in Remark \ref{Rem:der_equi_compat} are t-exact and
give the identity on the simple objects. The same holds for the categories
$D^b(\Coh^G(\St_0)), D^b(\Coh^G(\St_B)),D^b_{\Nilp}(\Coh^G(\St_\h))$
and the t-structures given by $\mathcal{T}_\h^*\otimes \mathcal{T}_\h$.
In particular, $$\tau:D^b_{I^\vee}(\Fl)\xrightarrow{\sim}
D^b(\Coh^G(\St_0)), D^b_{un}(I^\circ\backslash G^\vee((t))/I^\circ)
\xrightarrow{\sim} D^b_{\Nilp}(\Coh^G(\St_\h))$$ have t-exactness properties
similar to those explained in Theorem \ref{Thm:perv_equiv}.
\end{Rem}

\begin{Rem}\label{Rem:perverse_bimod}
Theorem \ref{Thm:perv_equiv} implies, in particular,
that $D^{b,\leqslant 0}_{perv}(\A_\h\otimes_{\C[\g]}\A^{opp}\operatorname{-mod}^G)$
is the negative part of a t-structure. The heart of this t-structure
will be denoted by $\operatorname{Perv}(\A_\h\otimes_{\C[\g]}\A^{opp}\operatorname{-mod}^G)$.
The objects in the heart of this t-structure
are called {\it perverse bimodules}. One can also see that it is the negative
part of a t-structure by using the construction of perverse coherent sheaves
from \cite{Arinkin_B}. We give a more detailed review of this t-structure in
the beginning of Section \ref{SS_abel_equiv}.
\end{Rem}

\subsection{Completed version}\label{SS_completed_cats}
We start by introducing a ``completed'' version of the equivalence
$\tau: D^b_{un}(I^\circ\backslash G^\vee((t))/I^\circ)\xrightarrow{\sim} D^b_{\Nilp}(\Coh^G(\St_\h))$.

The completed version of $D^b_{\Nilp}(\Coh^G(\St))$ is easy to define, this is the category
$D^b(\Coh^G(\St^{\wedge}_{\h}))$, where $\St_{\h}^\wedge$ is the formal neighborhood of
$\St_0$ in $\St_\h$. The corresponding completion of
$D^b_{un}(I^\circ\backslash G^\vee((t))/I^\circ)$ was constructed in
\cite[Appendix A]{BY}, it was shown in \cite[Corollary A.4.7]{BY}
to be   the derived category of a completed category of
perverse sheaves. We denote the completion by
$D^b_{pu}(I^\circ\backslash G^\vee((t))/I^\circ)$, with ``pu''
for ``pro-unipotent''. Note that both
$D^b(\Coh^G(\St^\wedge_\h))$ and $D^b_{pu}(I^\circ\backslash G^\vee((t))/I^\circ)$
are monoidal categories. The equivalence $\tau$ extends to
$$D^b_{pu}(I^\circ\backslash G^\vee((t))/I^\circ)
\xrightarrow{\sim} D^b(\Coh^G(\St^{\wedge}_\h)),$$
see \cite[Section 9.2]{B_Hecke}. We note that we have two commuting actions
of $\C[[\h]]$ on both $D^b_{pu}(I^\circ\backslash G^\vee((t))/I^\circ), D^b(\Coh^G(\St^{\wedge}_\h))$
(on the left and on the right) and the equivalence $\tau$ is bilinear.

Inside the category  $\operatorname{Perv}_{pu}(I^\circ\backslash G^\vee((t))/I^\circ)$
we have the full additive monoidal subcategory of free-monodromic tilting objects, denote it
by $\mathsf{Tilt}$. This category has the following properties.

\begin{enumerate}
\item The indecomposable objects in $\mathsf{Tilt}$ are indexed by $W^a$,
\cite[Proposition 11(a)]{B_Hecke}, let us write $\mathfrak{T}_x$ for the
indecomposable object labelled by $x$. All $\mathfrak{T}_x$ are flat
over $\C[[\h^*]]$ (both for the left and for the right action).
\item For $\mathcal{F}\in \mathsf{Tilt}$, the functor of convolution with
$\mathfrak{T}_x$ is  t-exact with respect to the perverse t-structure, see
\cite[Proposition 11(a)]{B_Hecke}. This convolution functor is  biadjoint to the convolution with $\mathfrak{T}_{x^{-1}}$. This is because the monoidal category $\mathsf{Tilt}$ is rigid,
the duality functor is obtained from the Verdier duality functor (extended to the pro-completion)
and the pushforward along $g\mapsto g^{-1}$. Every tilting is self-dual with respect
to the monoidal duality.
\item $K^b(\mathsf{Tilt})\xrightarrow{\sim} D^b_{pu}(I^\circ\backslash G^\vee((t))/I^\circ)$,
see \cite[Proposition 7]{B_Hecke}.
\end{enumerate}

We will also need the following lemma. Set
$\A_\h^{\wedge}:=\C[\h^*]^{\wedge_0}\widehat{\otimes}_{\C[\h^*]}\A_\h$.
Then $R\Gamma((\Tilt_\h\otimes \Tilt_\h^*)\otimes\bullet)$ gives an
equivalence
$$D^b(\Coh^G(\St_\h^\wedge))\xrightarrow{\sim}
D^b(\A_\h^\wedge\widehat{\otimes}_{\C[\g^*]}\A_\h^{\wedge,opp}\operatorname{-mod}^G).$$

\begin{Lem}\label{Lem:bimod_image}
The image of $\mathsf{Tilt}$ in $D^b(\Coh^G(\St^{\wedge}_\h))$
consists of $\A_\h^{\wedge}$-bimodules (in homological degree $0$).
\end{Lem}

In the proof we will need the following construction. Let $k\geqslant 0$.
Consider the
quotient category $D^b_{\Nilp}(\A_\h\otimes_{\C[\g^*]}\A_\h^{opp}\operatorname{-mod}^G)_{\geqslant 2k}$
by the full subcategory of all complexes with dimension of support (in $\Nilp$) $< 2k$.
This quotient category comes with the induced perverse bimodule t-structure whose heart is the quotient
$\operatorname{Perv}_{\Nilp}(\A_\h\otimes_{\C[\g^*]}\A_\h^{opp}\operatorname{-mod}^G)_{\geqslant 2k}$
of the category of perverse bimodules.
We note that for  $\mathcal{F}_0\in \A_\h\otimes_{\C[\g^*]}\A_\h^{opp}\operatorname{-mod}_\Orb$
with $\dim \Orb=2k$, we have  $\mathcal{F}_0[ 2k-\dim\mathcal{N}]\in\operatorname{Perv}_{\Nilp}(\A_\h\otimes_{\C[\g^*]}\A_\h^{opp}\operatorname{-mod}^G)_{\geqslant 2k}$.

\begin{proof}[Proof of Lemma \ref{Lem:bimod_image}]
By Theorem \ref{Thm:perv_equiv} and Remark \ref{Rem:perverse_bimod}, the image of
$\operatorname{Perv}_{pu}(I^\circ\backslash G^\vee((t))/I^\circ)$
in $D^b(\A_\h^\wedge\widehat{\otimes}_{\C[\g^*]}\A_\h^{\wedge,opp}\operatorname{-mod}^G)$
coincides with the category of perverse bimodules, to be denoted by
$\operatorname{Perv}(\A_\h^{\wedge}\widehat{\otimes}_{\C[\g^*]}\A_\h^{\wedge,opp}
\operatorname{-mod}^G)$.
We claim that if $\B\in \operatorname{Perv}(\A_\h^{\wedge}\widehat{\otimes}_{\C[\g^*]}\A_\h^{\wedge,opp}
\operatorname{-mod}^G)$ is such that
\begin{itemize}
\item $\B$ is flat over $\C[[\h]]$, and
\item
$\B\otimes^L_{\A_\h^{\wedge}}\bullet$ is  t-exact in the perverse
bimodule t-structure,
\end{itemize}
 then $\B$ is concentrated in homological degree $0$ (for the usual bimodule t-structure).
 Since $\tau(\mathfrak{T}_x)$ has both these properties by (1) and (2)
 before the lemma, this claim implies the lemma.

First of all, we prove that $H^i(\B)=0$ for $i>0$, where the cohomology is taken
with respect to the usual bimodule t-structure. Let $i>0$ be maximal such that
$H^i(\B)\neq 0$. Let $2j=\dim\operatorname{Supp} H^i(\B)$.
Note that $\B\otimes^L_{\A^{\wedge}_\h}\bullet$ induces a t-exact (in the perverse
bimodule t-structure) endofunctor of
$D^b_{\Nilp}(\A_\h\otimes_{\C[\g^*]}\A_\h^{opp}\operatorname{-mod}^G)_{\geqslant 2j}$.
We can find $\Orb$ of dimension $2j$ and  $\mathcal{F}_0\in \A_\h\otimes_{\C[\g^*]}\A_\h^{opp}\operatorname{-mod}_\Orb$
with $H^i(\mathcal{B})\otimes_{\A_\h}\mathcal{F}_0$ giving a nonzero object
in $\A_\h\otimes_{\C[\g^*]}\A_\h^{opp}\operatorname{-mod}_\Orb$.
Recall that $\mathcal{F}_0[2j-\dim \Nilp]$ lies in
$\operatorname{Perv}_{\Nilp}(\A_\h\otimes_{\C[\g^*]}\A_\h^{opp}\operatorname{-mod}^G)_{\geqslant 2j}$.
It follows that $\B\otimes^L_{\A_\h}\mathcal{F}_0[2j-\dim \Nilp]$ has positive cohomology
in the perverse bimodule t-structure. A contradiction with the t-exactness of
$\B\otimes^L_{\A_\h}\bullet$.

We proceed to proving that $H^i(\mathcal{B})=0$ for $i<0$ (in the usual t-structure).
Now let $\mathfrak{A}$ be a finite dimensional quotient of $\C[[\h]]$.
Note that $\B_{\mathfrak{A}}:=\B\otimes_{\C[[\h]]}\mathfrak{A}$
is still perverse and satisfies $H^i(\B_{\mathfrak{A}})=0$ for $i>0$.
We claim that $H^i(\B_{\mathfrak{A}})=0$ for $i<0$. Take smallest
$i$ such that $H^i(\B_{\mathfrak{A}})\neq 0$. Let $2k$ denote the
dimension of the support of $H^i(\B_{\mathfrak{A}})$.

So on the one hand the image of $\B_{\mathfrak{A}}$ in $\operatorname{Perv}_{\Nilp}(\A_\h\otimes_{\C[\g^*]}\A_\h^{opp}\operatorname{-mod}^G)_{\geqslant 2k}$ lies in the heart of the perverse bimodule t-structure. On the other hand
it admits a nonzero homomorphism from
$$\left(H^i(\B_{\mathfrak{A}})[2k-\dim \Nilp]\right)[\dim \Nilp-2k+i].$$
Note that  $H^i(\B_{\mathfrak{A}})[2k-\dim \Nilp]$ lies in
the heart of the perverse bimodule t-structure on the quotient category, and $\dim \Nilp-2k+i>0$.
We arrive at a contradiction
and conclude that  $H^i(\B_{\mathfrak{A}})=0$ for $i< 0$.
Since this holds for all $\mathfrak{A}$, we see that $H^i(\mathcal{B})=0$
for all $i< 0$. This completes the proof.
\end{proof}

\subsection{Compatibility with affine braid groups}\label{SS_braid}
We have homomorphisms
from $\Br^a$ to the monoidal categories
$$D^b(\Coh^G(\St^\wedge_\h)),D^b(\Coh^G(\St_0)),
D^b_{pu}(I^\circ\backslash G^\vee((t))/I^\circ),D^b_{I^\vee}(\Fl).$$
The homomorphisms to the  two constructible categories are standard: the generator
$T_s$ for each simple affine root $s$ goes to the costandard object labelled
by $s$. The homomorphisms to the  two coherent categories where constructed
in \cite{BR}, see Theorem 1.3.1 there, in particular.

In this section we discuss the compatibility of these homomorphisms
with the equivalences $\tau$.

\begin{Prop}\label{Prop:braid_compatible}
The equivalences
$$\tau: D^b_{I^\vee}(\Fl)\xrightarrow{\sim}
D^b(\Coh^G(\St_0)), D^b_{pu}(I^\circ\backslash G^\vee((t))/I^\circ)
\xrightarrow{\sim} D^b(\Coh^G(\St^\wedge_\h))$$
intertwine the homomorphisms from $\Br^a$.
\end{Prop}
\begin{proof}
One of the defining properties of $\tau$ in \cite{B_Hecke}, see
\cite[Section 4.2]{B_Hecke}, in particular, is that
it intertwines the homomorphisms from the lattice $\mathfrak{X}(T)$ in $\Br^a$
to the categories of interest. So what we need to prove is that
$\tau$ intertwines the generators $T_{s}$ for simple Dynkin
reflections $s$. Further it is enough to prove this claim
for  $$\tau:D^b_{pu}(I^\circ\backslash G^\vee((t))/I^\circ)
\xrightarrow{\sim} D^b(\Coh^G(\St^\wedge_\h))$$
only. This is because the equivalences $\tau$ are intertwined
by the pullback functors $$D^b_{pu}(I^\circ\backslash G^\vee((t))/I^\circ)
\rightarrow D^b_{I^\vee}(\Fl), D^b(\Coh^G(\St^\wedge_\h))
\rightarrow D^b(\Coh^G(\St_0)),$$
and the pullbacks also map the generators $T_{s}$ in
$D^b_{pu}(I^\circ\backslash G^\vee((t))/I^\circ), D^b(\Coh^G(\St^\wedge_\h))$
to those generators in $D^b_{I^\vee}(\Fl),  D^b(\Coh^G(\St_0))$,
by the construction of the actions.

The object $T_s$ in $D^b_{pu}(I^\circ\backslash G^\vee((t))/I^\circ)$ is
the free monodromic costandard object labelled by $s$, denote it by
$\nabla_{\h^{\wedge}}(s)$. In particular, it is flat over $\h^\wedge:=\operatorname{Spec}(\C[\h^*]^{\wedge_0})$
(on the either side).
The object $T_s$ in  $D^b(\Coh^G(\St^\wedge_\h))$ is
described as follows. Consider the locus $Z_{\h^{reg}}(s)\subset
\St_\h^{reg}:=\g^{reg}\times_{\g}\St_\h$ consisting of all triples $(x,\mathfrak{b},
\mathfrak{b}')$ such that $\mathfrak{b},\mathfrak{b'}$ are in the relative
position $s$. Let $Z_\h(s)$ be the closure of of $Z_{\h^{reg}}(s)$
in $\St_\h$. This is  the unique closed subscheme $Z_\h(s)$ of $\St_\h$
that is flat over $\h^*$ whose intersection with $\St_{\h}^{reg}$ is
$Z_{\h^{reg}}(s)$. Set $Z_{\h^\wedge}(s):=\St_\h^{\wedge}\cap Z_s$. Then
$T_s=\Str_{Z_{\h^\wedge}(s)}$,
see \cite[Theorem 1.3.1]{BR}. So we reduce to proving
\begin{equation}\label{eq:tau_braid_intertw}
\tau(\nabla_{\h^{\wedge}}(s))=\Str_{Z_{\h^\wedge}(s)}.
\end{equation}
This is a special case of \cite[Example 57]{B_Hecke}, but since that example
does not feature a proof, we are going to prove (\ref{eq:tau_braid_intertw})
here.

Consider the tilting object
$P_{\h^\wedge}(w_0)(=\mathfrak{T}_{w_0})\in \operatorname{Perv}_{pu}(I^\circ\backslash G^\vee((t))/I^\circ)$.
Then $\nabla_{\h^\wedge}(s)$ is a quotient of $P_{\h^{\wedge}}(w_0)$. It can be
characterized as follows. Note that, by  Soergel theory,
(see  \cite[Proposition 6.4]{BR_Soergel} or \cite[Proposition 4.7.3(1)]{BY})
$P_{\h^\wedge}(w_0)$ has commuting
actions of $\C[\h^\wedge]$ from the left and from the right so that
the actions of $\C[\h^{\wedge}]^W$ agree.
Set $\h^{\wedge,reg}:=\h^{\wedge}\cap \h^{reg}$.
Consider the base change $P_{\h^{\wedge,reg}}(w_0):=
\C[\h^{\wedge,reg}]\otimes_{\C[\h^{\wedge}]}P_{\h^\wedge}(w_0)$.
Then $P_{\h^{\wedge,reg}}(w_0)$ canonically splits into the direct
sum $P_{\h^{\wedge,reg}}(w_0)=\bigoplus_{w\in W}\nabla_{\h^{\wedge,reg}}(w)$,
where  $\nabla_{\h^{\wedge,reg}}(w)$ is the summand characterized by the property
that the left action of $\C[[\h]]$ is obtained from the right one by twisting
with $w$.  Now $\nabla_{\h^{\wedge}}(s)$ is the unique $\C[\h^{\wedge}]$-flat
quotient of $P_{\h^\wedge}(w_0)$ whose base change to $\h^{\wedge,reg}$ coincides
with $\nabla_{\h^{\wedge,reg}}(s)$.

By the first paragraph in \cite[Section 6]{B_Hecke},
$$\tau(P_{\h^{\wedge}}(w_0))=\Str_{\St_\h^{\wedge}}.$$
Note that the description of $\Str_{Z_{\h^\wedge}(s)}$
as the quotient of $\Str_{\St_\h^{\wedge}}$ mirrors the
description of $\nabla_{\h^{\wedge}(s)}$ as the quotient of
$P_{\h^{\wedge}}(w_0)$. Since $\tau$ is exact on the perverse
sheaves supported on $G^\vee/B^\vee$, see \cite[Corollary 42(a)]{B_Hecke},
and $\C[[\h]]$-bilinear,
(\ref{eq:tau_braid_intertw}) follows.
\end{proof}

We note that the argument in the proof proves the costandard part of
\cite[Example 57]{B_Hecke} for an arbitrary $w\in W$.

\section{Equivalence in parabolic setting}\label{S_parab_equiv}
\subsection{Notation and content}
The meaning of $G,\g,\bor,\h,\rho, B,T,\tilde{\g}, G^\vee, \g^\vee,\bor^\vee,\h^\vee,\Nilp, \St_\h,
\St_B,$ $\St_0,I^\vee,I^\circ, \Fl$ is the same as in Section \ref{SSS_constructible_notation}.
The notation $P, \mathcal{P}, \tilde{\Nilp}_P, Z,\iota,\varpi$ has the same meaning as in Section \ref{SS_tilting_notation}. Let $\Nilp_P$ denote $\operatorname{Spec}(\C[\tilde{\Nilp}_P])$
and let $\Nilp'_P$ denote its image in $\Nilp$.
We assume that $G$ is semisimple and simply connected.

We write $\tau$ for each of the derived equivalences of Theorem \ref{Thm:derived_equiv}.
Let $\Tilt_\h$ denote the Bezrukavnikov-Mirkovic tilting bundle on
$\tilde{\g}$, $\Tilt$ be its specialization to $0\in \h^*$,
$\A_\h:=\End(\Tilt)$, and $\A:=\End(\Tilt)$.
Recall that in the end of Section \ref{SS_tilting_BM}
we have introduced the tilting bundle $\Tilt_P:=\varpi_*\iota^* (\Tilt(-\rho))$.
We write $\A_P$ for the endomorphism algebra of $\Tilt_P$.

We also consider the parabolic version of the Steinberg variety, the derived scheme
$\St_{P}:=\tilde{\g}\times^L_{\g}\tilde{\Nilp}_P$ as well as the derived scheme
$\widehat{Z}:=\tilde{\g}\times^L_{\g}Z$. We write $\tilde{\iota},\tilde{\varpi}$
for the induced morphisms $\widehat{Z}\hookrightarrow \St_B$
and $\widehat{Z}\twoheadrightarrow \St_P$.

Let $P^\vee$ denote the parabolic subgroup of $G^\vee$ that contains
$B^\vee$ and corresponds to $P$. We write $J^\vee\subset G^\vee((t))$
for the preimage of $P^\vee$ in $G^\vee[[t]]$, and $\Fl_P$ for $G^\vee((t))/J^\vee$.
Let $\eta$ denote the projection $\Fl\twoheadrightarrow \Fl_P$. It is a locally
trivial fibration with fibers $P^\vee/B^\vee$.

This section is organized as follows. In Section \ref{SS_parab_statements} we state
the main results of this section that are generalizations of Theorems
\ref{Thm:derived_equiv} and Theorem \ref{Thm:perv_equiv} to the parabolic
setting as well as their corollary that is a crucial ingredient in
the proof of Theorem \ref{Thm:disting_dim}. The subsequent sections
are devoted to proving these results, their content is described in
more detail below.

\subsection{Statements of results}\label{SS_parab_statements}
Now we proceed to stating the main result of this section. We first construct  functors
$\varphi_1: D^b(\Coh^G(\St_P))\rightarrow D^b(\Coh^G(\St_B))$ and $\varphi_2:
D^b_{I^\circ}(\Fl_P)\rightarrow D^b_{I^\circ}(\Fl)$. The latter is given by
\begin{equation}\label{eq:varphi_2_def}
\varphi_2(\bullet):=\eta^*[\dim P/B]
\end{equation}
so that, in particular, it is
Verdier self-dual. Both $D^b_{I^\circ}(\Fl_P),D^b_{I^\circ}(\Fl)$ are module
categories over $D^b_{pu}(I^\circ\backslash G^\vee((t))/I^\circ)$ and the functor
$\varphi_2$ is  an equivariant functor.

Now let us describe the functor $\varphi_1$. Let $\rho_L$ denote half the sum of positive
roots  in $\lf$. We set
\begin{equation}\label{eq:coherent_functor}\varphi_1(\bullet):=\tilde{\iota}_*[\tilde{\varpi}^*(\bullet)(-\rho,\rho-2\rho_L)]:
D^b(\Coh^G(\mathsf{St}_P)) \rightarrow D^b(\Coh^G(\mathsf{St}_{B})),\end{equation}
where the expression in round brackets means the twist with the line bundle $\mathcal{O}(-\rho, \rho-2\rho_L)$
on $\St_B$, the pullback of $\mathcal{O}(-\rho)\boxtimes \mathcal{O}(\rho-2\rho_L)$. Both source and target categories are module categories over $D^b(\Coh^G(\St^\wedge_\h))$
however the functor $\varphi_1$ is not equivariant due to the $(-\rho)$-twist in the first copy
of $\tilde{\g}$. It becomes equivariant if we redefine the action of   $D^b(\Coh^G(\St^\wedge_\h))$
on $D^b(\Coh^G(\St_P))$ by conjugating it with $\mathcal{O}(\rho)$:
\begin{equation}\label{eq:twisted_convolution}
\mathcal{F}*_{\rho}\mathcal{G}:=(p_{13})_*\left(p_{12}^*(\mathcal{F})\otimes p_{23}^*(\mathcal{G}(-\rho))\right)(\rho).
\end{equation}

Here are the main results of this section. The first one should be thought as a parabolic analog of
Theorem \ref{Thm:derived_equiv}. It establishes a special case of \cite[Conjecture 59]{B_Hecke}.

\begin{Thm}\label{Prop:parabol_equiv_der}
We have an equivalence $\tau_P:D^b_{I^\circ}(\Fl_P)\xrightarrow{\sim}
D^b(\Coh^G\mathsf{St}_P)$ of triangulated categories
\begin{itemize}
\item
making the following diagram commutative,

\begin{picture}(100,30)
\put(2,22){$D^b_{I^\circ}(\Fl_P)$}
\put(2,2){$D^b_{I^\circ}(\Fl)$}
\put(52,22){$D^b(\Coh^G(\St_P))$}
\put(52,2){$D^b(\Coh^G(\mathsf{St}_{B}))$}
\put(8,21){\vector(0,-1){14}}
\put(65,21){\vector(0,-1){14}}
\put(20,23){\vector(1,0){31}}
\put(18,3){\vector(1,0){33}}
\put(9,13){\tiny $\varphi_2$}
\put(66,13){\tiny $\varphi_1$}
\put(35,24){\tiny $\sim$}
\put(35,20){\tiny $\tau_P$}
\put(35,4){\tiny $\sim$}
\put(35,1){\tiny $\tau$}
\end{picture}
\item
mapping $\underline{\C}_{P^\vee/P^\vee}$ to $\mathcal{O}_{Z^{diag}}$ (where $Z^{diag}$
is the image of $Z$ in $\St_P$ under the diagonal map),
\item
and intertwining the actions of $$D^b_{pu}(I^\circ\backslash G^\vee((t))/I^\circ)
\xrightarrow{\sim} D^b(\Coh^G(\St^{\wedge}_\h))$$ (where the action on $D^b(\Coh^G(\St_P))$
is given by (\ref{eq:twisted_convolution})).
\end{itemize}
\end{Thm}

By  Proposition \ref{Prop:braid_compatible}, $\tau_P$ intertwines the actions
of $\Br^a$.


Our second result is an analog  of Theorem \ref{Thm:perv_equiv}.
As in the case of $P=B$, we have the cell filtration
on $D^b_{I^\circ}(\Fl_P)$ and the nilpotent orbit filtration
on $D^b(\Coh^G(\St_P))$. We shift the numeration by
$\dim P/B$ so that the filtration degree $0$ quotient  functor for
$D^b(\Coh^G(\St_P))$ is isomorphic to the restriction to
$\Orb$. As usual, we consider the perverse t-structure on
$D^b_{I^\circ}(\Fl_P)$. The t-structure on
$D^b(\Coh^G(\St_P))$ is given by $\Tilt_\h(-\rho)^*\otimes \Tilt_P(-2\rho)$
and the heart is $\A_\h\otimes_{\C[\g]} \A_P^{opp}\operatorname{-mod}^G$.
Recall that $\Nilp'_P$ is the image of $\tilde{\Nilp}_P$ in $\Nilp$.

\begin{Thm}\label{Thm:parab_perv_equiv}
The equivalence
$$\tau_P: D^b_{I^\circ}(\Fl_P)\xrightarrow{\sim}D^b(\Coh^G(\St_{P}))$$
restricts to an equivalence  $D^b_{I^\circ,\leqslant \Orb'}(\Fl_P)\xrightarrow{\sim}D^b_{\leqslant \Orb'}(\Coh^G(\St_{P}))$ for all  $\Orb'\subset \Nilp'_P$). Moreover, for the induced equivalence
$$\tau_{\Orb'}: D^b_{I^\circ,\Orb'}(\Fl_P)\xrightarrow{\sim}
D^b_{\Orb'}(\Coh^G(\St_P))$$
we have that $\tau_{\Orb'}[\frac{1}{2}\operatorname{codim}_{\Nilp_P}\Orb']$ is t-exact
(with respect to t-structures described in the previous paragraph).
\end{Thm}

Theorems \ref{Prop:parabol_equiv_der} and \ref{Thm:parab_perv_equiv} will be proved
simultaneously. The proof goes as follows. We will introduce and study left adjoint
functors $\psi_1,\psi_2$ of $\varphi_1,\varphi_2$ in Section \ref{SS_fun_adjoint}.
We will describe compositions $\varphi_j\psi_j$ and $\psi_j \varphi_j$. In particular,
we will see that $\varphi_j\psi_j$ is given by convolving on the right with certain
objects. In Section \ref{SS_image_coinc} we will see that the equivalence $\tau$
intertwines those objects and use this to deduce that $\tau$ intertwines
the full Karoubian subcategories generated by the images of $\varphi_1,\varphi_2$
(below these subcategories will be called {\it full images}).
In Section
\ref{SS_abel_equiv} we will combine results of the  previous two sections
and show that $\tau$ restricts to an equivalence of abelian categories
between $\operatorname{Perv}_{I^\circ}(\Fl_P)$ and the heart
of the perverse t-structure on  $D^b(\A_\h\otimes_{\C[\g]}\A_P^{opp}\operatorname{-mod}^G)$.
We will use this equivalence to prove Theorems \ref{Prop:parabol_equiv_der} and
\ref{Thm:parab_perv_equiv} in Section \ref{SS_parab_complete}. This will, in particular,
imply that $\tau_P$ is t-exact with respect to the perverse t-structure
on $D^b_{I^\circ}(\Fl_P)$ and the perverse bimodule t-structure on
$D^b(\Coh^G(\St_P))$.

Below we will need a corollary of these two theorems.
Recall that $\Orb$ denotes the dense orbit in $\Nilp'_P$
and  we write $\B_e$ for the Springer fiber of $e$ with its natural
derived scheme structure (of a derived subscheme in $\tilde{\g}$).
Note that the preimage of $\Orb$ in $\St_P$ is naturally identified with
$G\times^{Z_P(e)}\B_e$ (an isomorphism of derived schemes). So
\begin{equation}\label{eq:equiv_cat_equiv}
D^b_{\Orb}(\Coh^G(\St_P))\xrightarrow{\sim} D^b(\Coh^{Z_P(e)}\B_e),
\end{equation}
Consider the t-structure on $D^b(\Coh^{Z_P(e)}\B_e)$ given by $\mathcal{T}(-\rho)^*$.
Its heart is $\A_{\h,e}\operatorname{-mod}^{Z_P(e)}$, where, recall, we write
$\A_{\h,e}$ for the fiber of $\A_\h$ at $e$.

Inside $\B_e$ consider the (ordinary) subvariety $\B_{\mathfrak{m}}$ consisting of
all Borel subalgebras containing $\mathfrak{m}^{(1)}_\F$. It is naturally identified with $P/B$.
Note that $\B_{\mathfrak{m}}$ is an irreducible component in $\B_e$.



\begin{Cor}\label{Thm:abelian_quotient}
The quotient functor $D^b_{I^\circ}(\Fl_P)\twoheadrightarrow D^b(\Coh^{Z_P(e)}\B_e)$
has the following properties:
\begin{enumerate}
\item it is t-exact for the t-structures above,
\item it intertwines the actions of $D^b_{pu}(I^\circ\backslash G^\vee((t))/I^\circ)
\xrightarrow{\sim} D^b(\Coh^G(\St^\wedge_\h))$ (hence the  $\Br^a$-actions),
\item it maps $\underline{\C}_{P^\vee/P^\vee}$ to $\mathcal{O}_{\B_{\mathfrak{m}}}$.
\end{enumerate}
\end{Cor}

\begin{Rem}\label{Rem:simply_conn}
Above we have assumed that $G$ is simply connected, which is needed
for $\rho\in \mathfrak{X}(T)$. We will need to weaken this assumption
to the case when we have a character $\rho'$ of $T$ that coincides with $\rho$ on the coroots:
this holds when our group is a Levi of  a simply connected semisimple group, which is
precisely the situation we need. Results in this section continue to hold with easy modifications,
for example, in the definition of $\varphi_1$ we need to twist with $(-\rho',\rho'-2\rho_L)$.
\end{Rem}

\subsection{Adjoint functors}\label{SS_fun_adjoint}
In this section we are going to introduce and study left adjoint functors
of $\varphi_1,\varphi_2$ to be denoted by $\psi_1,\psi_2$, respectively.

Let us start with $\psi_2$, which is easier. The left adjoint of
$\eta^!=\varphi_2[\dim P/B]$ is $\eta_!$ so that
\begin{equation}\label{eq:psi_2_defn}
\psi_2=\eta_![\dim P/B].
\end{equation}

\begin{Lem}\label{Lem:psi_2_properties}
The following claims are true.
\begin{enumerate}
\item We have $\psi_2\varphi_2\cong \operatorname{id}\otimes H^*(P^\vee/B^\vee,\C)[2\dim P/B]$,
where we view $H^*(P^\vee/B^\vee,\C)$ as the complex  with zero differential,
$H^*(P^\vee/B^\vee,\C)=\bigoplus_{i} H^i(P^\vee/B^\vee,\C)[-i]$.
\item We have $\varphi_2\psi_2\cong \bullet* \underline{\C}_{P^\vee/B^\vee}[2\dim P/B]$.
\end{enumerate}
\end{Lem}
\begin{proof}
We have $\psi_2\varphi_2=\eta_!\eta^!$, which implies (1). The proof of
(2) is similar.
\end{proof}

The functor $\varphi_1$ admits a  left adjoint functor as well. Namely, $\tilde{\iota}^*$
is the left adjoint functor to $\tilde{\iota}_*$. Also the relative canonical
bundle of $Z\rightarrow \tilde{\Nilp}_P$ is $\mathcal{O}(-2\rho_L)$ hence,
by the Serre duality, $\varpi_*$ is
left adjoint of $\varpi^*(\bullet)(-2\rho_L)[\dim P/B]$. We conclude that the left adjoint $\psi_1$ of
$\varphi_1$ is given by
\begin{equation}\label{eq:psi_1_defn}
\psi_1(\bullet):=\tilde{\varpi}_*\left(\tilde{\iota}^*(\bullet)(\rho,-\rho)\right)[\dim P/B].
\end{equation}

\begin{Lem}\label{Lem:psi_1_properties}
The following claims are true.
\begin{enumerate}
\item We have $\psi_1\varphi_1\cong \operatorname{id}\otimes H^*(P/B,\C)[2\dim P/B]$.
\item We have $$\varphi_1\psi_1\cong \bullet* \mathcal{O}_{Z\times_{\tilde{\Nilp}_P}Z}(-\rho,\rho-2\rho_L)[\dim P/B].$$
\end{enumerate}
\end{Lem}
\begin{proof}
Let us prove part (1). Note that, for $\mathcal{G}\in D^b(\Coh^G(\St_P))$,
$$\psi_1 \varphi_1(\mathcal{G})=\tilde{\varpi}_* \left(\tilde{\iota}^*\circ \tilde{\iota}_*
\left[\tilde{\varpi}^*\mathcal{G}\right](-2\rho_L))\right)[\dim P/B].$$
The composition $\iota^*\circ \iota_*(\bullet)$ is tensoring with $\mathcal{O}_{Z}\otimes^L_{\mathcal{O}_{\tilde{\Nilp}}}\mathcal{O}_{Z}$, where we view
$Z$ as a closed subscheme of $\tilde{\Nilp}$ via $\iota$.
Tensoring this complex by $\mathcal{O}(-2\rho_L)$
and applying the shift by $\dim P/B$ we get
$R\mathcal{E}nd_{\mathcal{O}_{\tilde{\Nilp}}}(\mathcal{O}_{Z})[2\dim P/B]$.
We have $\varpi_* R\mathcal{E}nd_{\mathcal{O}_{\tilde{\Nilp}}}(\mathcal{O}_{Z})\cong
\mathcal{O}_{\tilde{\Nilp}_P}\otimes H^*(P/B,\C)$. This boils down to proving
$R\mathcal{E}nd_{\mathcal{O}_{\tilde{\Nilp}}}(\mathcal{O}_\B,\mathcal{O}_\B)\cong
H^*(\B,\C)$ as a $G$-module. The latter is a consequence of the Hodge theorem:
for a smooth variety $X$, the self-Ext's
of the structure sheaf of the zero section in $T^*X$  coincide
with the Hodge cohomology of $X$.
So (1) is proved.


To prove part (2), we first note that
$$\varphi_1\psi_1(\bullet)=\left[\tilde{\iota}_*\circ\tilde{\varpi}^*\circ \tilde{\varpi}_*\circ\tilde{\iota}^*
(\bullet(\rho,-\rho))\right](-\rho,\rho-2\rho_L)[\dim P/B]$$

Consider the variety $Z\times_{\tilde{\Nilp}_P}Z$ and let $\kappa_i: Z\times_{\tilde{\Nilp}_P}Z
\rightarrow T^*\B$ denote the projection
to the $i$th factor composed with the inclusion $\iota: Z\hookrightarrow T^*\B$.
Let $\tilde{\kappa}_i$ denote the induced morphism $\tilde{\g}\times^L_{\g}(Z\times_{\tilde{\Nilp}_P}Z)
\rightarrow \St_B$. Note that
$$\tilde{\iota}_*\circ\tilde{\omega}^*\circ \tilde{\omega}_*\circ\tilde{\iota}^*=\tilde{\kappa}_{2*}
\circ \tilde{\kappa}_1^*,$$
so we get
\begin{align*}\varphi_1\psi_1(\mathcal{F})=&\left(\tilde{\kappa}_{2*}
\circ \tilde{\kappa}_1^*
(\mathcal{F}(\rho,-\rho))\right)(-\rho,\rho-2\rho_L)[\dim P/B]=\\
&\left(\tilde{\kappa}_{2*}
\circ \tilde{\kappa}_1^*
(\mathcal{F}(0,-\rho))\right)(0,\rho-2\rho_L)[\dim P/B]=\\
&\mathcal{F}*\mathcal{O}_{Z\times_{\tilde{\Nilp}_P}Z}(-\rho,\rho-2\rho_L)[\dim P/B].\end{align*}
This proves part (2).
\end{proof}


\subsection{Coincidence of full images}\label{SS_image_coinc}
Let $\Fim\varphi_j$ denote the Karoubian envelope of the full subcategory
in $D^b_{I^\circ}(\Fl)$ (for $j=2$) or $D^b(\Coh^G(\St_B))$ (for $j=1$)
generated by the objects in the image of $\varphi_j$.

A crucial step in the proof of the claim that $\tau$ intertwines
$\Fim\varphi_2$ with $\Fim\varphi_1$ is the following lemma.

\begin{Lem}\label{Lem:der_equiv_image}
The image of $\underline{\C}_{P^\vee/B^\vee}[\dim P^\vee/B^\vee]$
in $D^b(\Coh^G(\St_0))$ is $\mathcal{O}_{Z\times_{\tilde{\Nilp}_P}Z}(-\rho,\rho-2\rho_L)$.
\end{Lem}

We remark that the same result holds for $\underline{\C}_{P^\vee/B^\vee}[\dim P^\vee/B^\vee]$
viewed as an object of $D^b_{I^\circ}(\Fl)$ and $\mathcal{O}_{Z\times_{T^*(G/P)}Z}(-\rho,\rho-2\rho_L)$
viewed as an object in $D^b(\Coh^G(\St_B))$. This follows from Remark
\ref{Rem:perv_equiv_compat}.

\begin{proof}
The proof of this lemma is in three steps.

{\it Step 1}. First, consider the case when $P=G$. Here we need to prove that
$\tau(\underline{\C}_{\B^\vee}[\dim \B])=\mathcal{O}_{\B\times \B}(-\rho,-\rho)$.
We have
\begin{equation}\label{eq:RGamma_identification}
R\Gamma\left( (\Tilt\otimes \Tilt^*)\otimes \mathcal{O}_{\B\times \B}(-\rho,-\rho)\right)=R\Gamma(\Tilt(-\rho))\otimes R\Gamma(\Tilt^*(-\rho)).\end{equation}
Thanks to Lemma
\ref{Lem:reduct_mod_p}, we have a
Morita equivalence $(\A_{\h,0})_\F\cong \U^0_{(0), \F}$. Using that lemma, we also see that
under that Morita equivalence the object
$R\Gamma(\Tilt(-\rho))\otimes R\Gamma(\Tilt^*(-\rho))$ becomes
$R\Gamma(\mathcal{V}^0_0)\otimes R\Gamma((\mathcal{V}_0^0)^*(-2\rho))$.
By (1) of Lemma \ref{Lem:split_properties}, the first factor is
$\Gamma(\mathcal{O}_{\mathcal{B}_\F})=\F$, the trivial $G_\F$-module.
By the Serre duality, the second factor is $\F^*[-\dim \B]$.
So (\ref{eq:RGamma_identification}) is a simple $G$-equivariant $\A_{\h,0}$
module in cohomological degree $\dim \B$.

On the other hand, by Theorem \ref{Thm:perv_equiv}, $\tau(\underline{\C}_{\B^\vee}[\dim \B])$
is also a simple $G$-equivariant $\A$-bimodule shifted by $\dim \B$.
So what we need to show is that the $K_0$-classes of   $\tau(\underline{\C}_{\B^\vee}[\dim \B])$
and $\mathcal{O}_{\B\times \B}(-\rho,-\rho)$ coincide. Under the standard identification
of $K_0(D^b_{I^\circ}(\Fl))$ with $\Z W^a$ (where the class of the standard
object labelled by $x\in W^a$ goes to $x\in \Z W^a$), the class of
$\underline{\C}_{\B^\vee}[\dim \B]$ is $\sum_{w\in W}(-1)^{\ell(w_0 w)}w$.
So we need to show that the class of $\mathcal{O}_{\B\times \B}(-\rho,-\rho)$ in
$K_0^{G}(\St_B)$ (or equivalently $K_0^{G}(\St_0)$) is $\sum_{w\in W}(-1)^{\ell(w_0w)}w$. For this we consider the action of
$K_0^{G}(\St_0)$ on $K_0^T(T^*\B)$ by $\rho$-twisted convolution
(\ref{eq:twisted_convolution}).
After localizing  $K_0^T(\operatorname{pt})$, the module
$K_0^T(T^*\B)$ gets the fixed point basis. The  elements of this basis
are naturally indexed by the elements of $W^a$. The elements of $W$ correspond
to the skyscraper sheaves at the fixed points with trivial $T$-actions.
In this basis the action of
$W^a$ is by left multiplications. Convolving $\mathcal{O}_{\B\times \B}(-\rho,-\rho)$
with the skyscraper sheaf at $1B$ we get $\mathcal{O}_{\B}$. Decomposing with respect to the $T$-fixed point basis
we get $\sum_{w\in W} (-1)^{\ell(w_0w)}w$. This finishes the proof of
$\tau(\underline{\C}_{\B^\vee}[\dim \B])\cong \mathcal{O}_{\B\times \B}(-\rho,-\rho)$.

{\it Step 2}. In this step we consider the case of general $P$.

We are going to construct a  functor
$\xi_P: D^b(\Coh^{L}(\St_{L,0}))\rightarrow D^b(\Coh^G(\St_0))$,
where $\St_{L,0}$ is the analog of $\St_0$ for $L$.
Consider the derived scheme $\mathsf{S}$ of pairs $x\in \mathfrak{n}$
and $\mathfrak{b}'\in \B$ with $x\in \mathfrak{b}'$ (with its natural derived scheme
structure). Note that $D^b(\Coh^G(\St_0))$ is naturally identified with
$D^b(\Coh^B(\mathsf{S}))$. Consider the derived subscheme $\mathsf{S}_{\mathfrak{m}}\subset \mathsf{S}$ consisting of all pairs $(x,\mathfrak{b'})\in \mathsf{S}$ such that $\mathfrak{m}\subset \mathfrak{b}'$.
It projects to $\mathsf{S}_L$, an analog of $\mathsf{S}$ for $L$. The functor
$D^b(\Coh^{L}(\St_{L,0}))\rightarrow D^b(\Coh^G(\St_0))$ we need is the
pull-push functor via $\mathsf{S}_L\leftarrow \mathsf{S}_{\mathfrak{m}}\rightarrow \mathsf{S}$.

By the construction of the homomorphism from $\operatorname{Br}^a$ to
$D^b(\Coh^G(\St_0))$ given in \cite{BR} (see the proof of Proposition
\ref{Prop:braid_compatible}), $\xi_P$ is equivariant
for the action of the braid group $\Br_P$ for $W_P$. The functor $\xi_P$ maps the structure sheaf
of the diagonal to the structure sheaf of the diagonal. It also maps
$\mathcal{O}_{P/B\times P/B}(-\rho,\rho-2\rho_L)$ to  $\mathcal{O}_{Z\times_{\tilde{\Nilp}_P}Z}(-\rho,\rho-2\rho_L)$.

{\it Step 3}. Let $\tilde{L}$ denote the product of a torus and a simply connected
semisimple group that is a cover of $L$.
Let us write $\tau_{\tilde{L}}$ for the equivalence $D^b_{I^\vee_{\tilde{L}}}(\Fl_{\tilde{L}})\xrightarrow{\sim}
D^b(\Coh^{\tilde{L}}(\St_{\tilde{L},0}))$, where $\Fl_{\tilde{L}}$
is the affine flag variety for $\tilde{L}^\vee$
and $I^\vee_{\tilde{L}}$ is the standard Iwahori subgroup in $\tilde{L}^\vee((t))$.
By Step 1, $\tau_{\tilde{L}}(\underline{\C}_{P^\vee/B^\vee}[\dim P/B])=\mathcal{O}_{P/B\times P/B}(-\rho_L,-\rho_L)$.
Note that $\tau_{\tilde{L}}$ restricts to an
equivalence $\tau_L: D^b_{I^\vee_{L}}(\Fl_{L})\xrightarrow{\sim}
D^b(\Coh^{L}(\St_{L,0}))$.
The $L$-equivariant sheaves $\mathcal{O}_{P/B\times P/B}(-\rho_L,-\rho_L)$ and $
\mathcal{O}_{P/B\times P/B}(-\rho,\rho-2\rho_L)$
coincide. Hence
$\tau_{L}(\underline{\C}_{P^\vee/B^\vee}[\dim P/B])=\mathcal{O}_{P/B\times P/B}(-\rho,\rho-2\rho_L)$.

 The simple labelled by $x\in W^a$ in $\operatorname{Perv}_{I^\vee}(\Fl)$ is
the image of a unique (up to rescaling) nonzero homomorphism from $T_{x}^{-1}\underline{\C}_{B^\vee/B^\vee}$ (the standard object labelled by $x$)
to  $T_{x}\underline{\C}_{B^\vee/B^\vee}$ (the costandard object labelled
by $x$).  It was proved in \cite[Corollary 42]{B_Hecke} that $\tau$
maps simples in $\operatorname{Perv}_{B^\vee}(G^\vee/B^\vee)$ to
coherent sheaves. For $x\in W$, both $T_{x}^{-1}\underline{\C}_{B^\vee/B^\vee}$ and
$T_{x}\underline{\C}_{B^\vee/B^\vee}$ lie in $\operatorname{Perv}_{B^\vee}(G^\vee/B^\vee)$.
So $\tau(\underline{\C}_{P^\vee/B^\vee}[\dim P/B])$
is the image of a unique nonzero morphism $T_{w_{0,L}}^{-1}\tau(\underline{\C}_{B^\vee/B^\vee})\rightarrow
T_{w_{0,L}}\tau(\underline{\C}_{B^\vee/B^\vee})$. The functor $\xi_P$ is  exact and faithful
on the heart of the usual t-structure of coherent sheaves. This is because it is the composition of the pull-back under a locally trivial fibration and the push-forward under a closed embedding.
The functor $\xi_P$ is $\operatorname{Br}_{W_P}$-equivariant hence maps
$T_{w_{0,L}}^{-1}\tau_L(\underline{\C}_{B^\vee/B^\vee}),
T_{w_{0,L}}\tau_L(\underline{\C}_{B^\vee/B^\vee})$
to $T_{w_{0,L}}^{-1}\tau(\underline{\C}_{B^\vee/B^\vee}),
T_{w_{0,L}}\tau(\underline{\C}_{B^\vee/B^\vee})$, respectively.
It follows that it maps $\tau_L(\underline{\C}_{P^\vee/B^\vee}[\dim P/B])$
to $\tau(\underline{\C}_{P^\vee/B^\vee}[\dim P/B])$. This finishes the proof.
\end{proof}

\begin{Rem}
We would like to sketch an alternative proof of Step 1. Let $B_{\St}$ denote the simple
objects in $\A_{0}\otimes \A_0^{opp}\operatorname{-mod}^G$ corresponding to
$\Str_{\B\times \B}(-\rho,-\rho)[\dim \B]$.  Let $B_{\Fl}$ denote the simple
object in $\Perv_{I^\vee}(\Fl)$ with $\tau(B_{\Fl})=B_{\St}[-\dim \B]$.
We need to show that $\tau(B_{\Fl})=\underline{\C}_{G^\vee/B^\vee}[\dim \B]$. For this we make two observations
about $B_{\St}$. First, it is homologically shifted when we apply $T_\alpha$'s for simple
Dynkin roots both on the left and on the right. It follows that $B_{\Fl}$
is pulled from a simple object in the Satake category $\Perv_{G^\vee((t))}(\Fl_G)$.
Second, $B_{\St}$ remains irreducible in the non-equivariant category
$\A_{0}\otimes \A_0^{opp}\operatorname{-mod}$. It follows that $B_{\St}\otimes V\in
\A_0\otimes \A_0^{opp}\operatorname{-mod}^G$ is irreducible for every irreducible
$G$-module $V$.  There is only one object in the Satake category with this property:
the sky-scraper sheaf $\underline{\C}_{pt}$. Its pullback to $\Fl$ is $\underline{\C}_{G^\vee/B^\vee}$.
\end{Rem}

\begin{Cor}\label{Cor:fim_coinc}
The equivalence $\tau:D^b_{I^\circ}(\Fl)\xrightarrow{\sim} D^b(\Coh^G(\mathsf{St}_{B}))$
restricts to an equivalence between  $\operatorname{Fim}\varphi_2$ and $\operatorname{Fim}\varphi_1$.
\end{Cor}
\begin{proof}
Recall, Theorem \ref{Thm:derived_equiv}, that $\tau$ is equivariant with respect to the action of
$D^b_{I^\vee}(\Fl)\cong D^b(\Coh^G(\mathsf{St}_{0}))$ by convolutions on the right.
It follows from Lemma \ref{Lem:der_equiv_image} that, under the equivalence
$\tau:D^b_{I^\vee}(\Fl)\xrightarrow{\sim} D^b(\Coh^G(\mathsf{St}_{0}))$, we have
 $$\underline{\C}_{P^\vee/B^\vee}[\dim P^\vee/B^\vee]\mapsto \mathcal{O}_{Z\times_{\tilde{\Nilp}_P}Z}(-\rho,\rho-2\rho_L).$$
It follows from Lemma \ref{Lem:psi_2_properties}, that
$\operatorname{Fim}\varphi_2$ is the Karoubian envelope of
$$\{\mathcal{F}* \underline{\C}_{P^\vee/B^\vee}[\dim P/B]| \mathcal{F}\in D^b_{I^\circ}(\Fl)\}.$$
Similarly, it follows from Lemma \ref{Lem:psi_1_properties} that
$\operatorname{Fim}\varphi_1$ is the Karoubian envelope of
$$\{\mathcal{G}* \mathcal{O}_{Z\times_{\tilde{\Nilp}_P}Z}(-\rho,\rho-2\rho_L)| \mathcal{G}\in D^b(\Coh^G(\St_B))\}.$$
This finishes the proof.
\end{proof}

\subsection{Abelian equivalence}\label{SS_abel_equiv}

The tilting bundles $\Tilt^*\otimes \Tilt$ on $\St_B$ and  $\mathcal{T}(-\rho)^*\otimes \mathcal{T}_P(-2\rho)$ on $\St_P$  give equivalences
$$D^b(\Coh^G(\St_B))\xrightarrow{\sim} D^b(\A_\h\otimes_{\C[\g^*]}\A^{opp}\operatorname{-mod}^G),
D^b(\Coh^G(\St_P))\xrightarrow{\sim} D^b(\A_\h\otimes_{\C[\g^*]}\A_P^{opp}\operatorname{-mod}^G).$$
On  $D^b(\A_\h\otimes_{\C[\g^*]}\A^{opp}\operatorname{-mod}^G), D^b(\A_\h\otimes_{\C[\g^*]}\A_P^{opp}\operatorname{-mod}^G)$ we have perverse
bimodule t-structures, see Remark \ref{Rem:perverse_bimod}. In more detail,
notice that the algebras $\A_\h,\A_P^{opp}$ are Gorenstein, as the algebras
of endomorphisms of tilting bundles on $\tilde{\g},\tilde{\Nilp}_P$, see
Lemma \ref{Lem:tilting_gen_prop}.
Being a complete intersection in a Gorenstein algebra, the algebra
$\tilde{\A}_\h\otimes_{\C[\g^*]}\A_P^{opp}$ is Gorenstein as well.

For a $G$-equivariant Gorenstein $\C[\Nilp_P]$-algebra $\A'$ we have a unique
{\it perverse} t-structure on $D^b(\A'\operatorname{-mod}^G)$ whose $\leqslant 0$-part
is given by
$$\{M\in D^b(\A'\operatorname{-mod}^G)| \operatorname{codim}_{\Nilp_P}\operatorname{Supp}
H^i(M)\geqslant 2i.\}$$
The $\geqslant 0$-part is obtained from the $\leqslant 0$ for the category of
right $\A'$-modules by applying the functor $R\Hom_{\A'}(\bullet,\A')$.
The proof copies that in \cite[Section 3]{Arinkin_B}. Moreover, since $\A'$ is Gorenstein,
we see that the forgetful
functors $\A'\operatorname{-mod},\A'^{opp}\operatorname{-mod}
\rightarrow \C[\Nilp_P]\operatorname{-mod}$ intertwine $R\Hom_{\A'}(\bullet,\A')$
with $R\Hom_{\C[\Nilp_P]}(\bullet,\C[\Nilp_P])$. So the forgetful
functor $D^b(\A'\operatorname{-mod}^G)\rightarrow D^b(\C[\Nilp_P]\operatorname{-mod}^G)$
is t-exact.

\begin{Prop}\label{Prop:varphi_1_perverse_exact}
The functor $\varphi_1$ is t-exact with respect to the perverse bimodule t-structures.
\end{Prop}
\begin{proof}
Let $\zeta$ denote the natural morphism $\Nilp_P\rightarrow \Nilp$, it is
finite. Arguing as in \cite[Lemma 3.3]{Arinkin_B}, we see that $\zeta_*[\dim P/B]$ is t-exact
with respect to the perverse t-structures. Now note that the forgetful
functors $$D^b(\A_\h\otimes_{\C[\g^*]} \A_P^{opp}\operatorname{-mod}^G)
\rightarrow D^b(\C[\Nilp_P]\operatorname{-mod}^G),
D^b(\A_\h\otimes_{\C[\g^*]} \A^{opp}\operatorname{-mod}^G)
\rightarrow D^b(\C[\Nilp]\operatorname{-mod}^G)$$
intertwine $\varphi_1$ with $\zeta_*[\dim P/B]$.
The forgetful functors are faithful and t-exact. So since $\zeta_*[\dim P/B]$
is t-exact, we see that $\varphi_1$ is t-exact.
\end{proof}

Let $\operatorname{Perv}(\A_\h\otimes_{\C[\g^*]}\A_P^{opp}\operatorname{-mod}^G)$
denote the heart of the perverse t-structure. We view it as a full subcategory
of $D^b(\Coh^G(\St_P))$.

The main result of this section
is the following proposition.

\begin{Prop}\label{Prop:abel_equiv_parab}
The following claims are true.
\begin{enumerate}
\item
The functors $$\varphi_1: \operatorname{Perv}(\A_\h\otimes_{\C[\g^*]}\A_P^{opp}\operatorname{-mod}^G)
\rightarrow D^b(\Coh^G(\St_B)),\quad\varphi_2:
\operatorname{Perv}_{I^\circ}(\Fl_P)\rightarrow D^b_{I^\circ}(\Fl)$$ are full embeddings.
\item The equivalence $$\tau: D^b_{I^\circ}(\Fl)\xrightarrow{\sim}D^b(\Coh^G(\St_{B}))$$ restricts to an
equivalence $$\varphi_2(\operatorname{Perv}_{I^\circ}(\Fl_P))\xrightarrow{\sim}
\varphi_1(\operatorname{Perv}(\A_\h\otimes_{\C[\g^*]}\A_P^{opp}\operatorname{-mod}^G)).$$
\end{enumerate}
\end{Prop}
\begin{proof}
We start by proving (1). We will consider the case of $\varphi_1$, the other case is similar.

First of all, we claim that for all $\mathcal{F},\mathcal{G}\in D^b(\Coh^G(\St_P))$ we have
\begin{equation}\label{eq:fin_dimension}
\dim \Hom_{D^b(\Coh^G(\St_P))}(\mathcal{F},\mathcal{G})<\infty.
\end{equation}
For this, we note that the Hom in the non-equivariant
category is a finitely generated $\C[\g^*]$-module supported
on $\Nilp$. The Hom in the equivariant category is the
$G$-invariants in that module. Since $\Nilp$ has only
finitely many $G$-orbits, the space of invariants is finite dimensional.

Recall, Lemma \ref{Lem:psi_1_properties}, that $\psi_1\varphi_1=\operatorname{id}\otimes H^*(P/B,\C)[2\dim P/B]$. Hence $\psi_1\varphi_1$ is faithful. So is $\varphi_1$.

Note that
$$\Hom_{D^b(\Coh^G(\St_B))}(\varphi_1 \mathcal{F}, \varphi_1 \mathcal{G})=\Hom_{D^b(\Coh^G(\St_P))}(\psi_1\varphi_1 \mathcal{F}, \mathcal{G}),$$
for all $\mathcal{F},\mathcal{G}\in D^b(\Coh^G(\St_P))$. Since $\varphi_1$ is faithful,
we see that
\begin{equation}\label{eq:counit_map}
\Hom_{D^b(\Coh^G(\St_P))}(\mathcal{F}, \mathcal{G})\hookrightarrow \Hom_{D^b(\Coh^G(\St_P))}(\psi_1\varphi_1 \mathcal{F}, \mathcal{G}),
\end{equation}

Now assume that $\mathcal{F},\mathcal{G}$ are perverse bimodules.
Since $\mathcal{F},\mathcal{G}$ lie in the heart of a t-structure, we have
$\Hom_{D^b(\Coh^G(\St_P))}(\mathcal{F}[i], \mathcal{G})=0$ for $i>0$.
Since $\psi_1\varphi_1\cong \operatorname{id}\otimes H^*(P/B,\C)[2\dim P/B]$,
we see that
$$\dim\Hom_{D^b(\Coh^G(\St_P))}(\mathcal{F}, \mathcal{G})=
\dim\Hom_{D^b(\Coh^G(\St_P))}(\psi_1\varphi_1 \mathcal{F}, \mathcal{G})$$
Combining this with (\ref{eq:fin_dimension}) and (\ref{eq:counit_map}),
we see that (\ref{eq:counit_map}) is an isomorphism. This finishes the
proof of (1).
%

Now we prove 2).  Proposition \ref{Prop:varphi_1_perverse_exact} says  that $\varphi_1$
is t-exact with respect to the perverse bimodule t-structures. And $\varphi_2$ is t-exact as well.
So we have
\begin{equation}\label{eq:image_relation}
\begin{split}
&\varphi_1(\operatorname{Perv}(\A_\h\otimes_{\C[\g^*]}\A_P^{opp}\operatorname{-mod}^G))=
\operatorname{Fim}\varphi_1\cap \operatorname{Perv}(\A_\h\otimes_{\C[\g^*]}\A^{opp}\operatorname{-mod}^G),\\
&\varphi_2(\operatorname{Perv}_{I^\circ}(\Fl_P))=
\operatorname{Fim}\varphi_2\cap \operatorname{Perv}_{I^\circ}(\Fl).
\end{split}
\end{equation}
By Theorem \ref{Thm:perv_equiv}, $\tau$ is t-exact with respect to the perverse t-structures.
Combining this with Corollary \ref{Cor:fim_coinc} and (\ref{eq:image_relation}), we get 2).
\end{proof}

So we get an equivalence $\Perv_{I^\circ}(\Fl_P)\xrightarrow{\sim} \Perv(\A_\h\otimes_{\C[\g^*]}\A_P^{opp}\operatorname{-mod}^G)$ that we denote
by $\tau_P$.

\subsection{Completion of proofs of main results}\label{SS_parab_complete}
First, we need to produce a functor $\tau_P: D^b_{I^\circ}(\Fl_P)
\xrightarrow{\sim} D^b(\Coh^G(\St_P))$ that makes the  diagram
in Theorem \ref{Prop:parabol_equiv_der} commutative. Note that
$D^b(\Perv_{I^\circ}(\Fl_P))\xrightarrow{\sim} D^b_{I^\circ}(\Fl_P)$
by \cite[Proposition 1.5]{BBM}.
Also, since $\Perv(\A_\h\otimes_{\C[\g^*]}\A_P^{opp}\operatorname{-mod}^G)$
is the heart of a t-structure on $D^b(\Coh^G (\St_P))$, by \cite[A6]{Beilinson} the embedding $\Perv(\A_\h\otimes_{\C[\g^*]}\A_P^{opp}\operatorname{-mod}^G)
\rightarrow D^b(\Coh^G (\St_P))$
extends to a triangulated realization functor $D^b(\Perv(\A_\h\otimes_{\C[\g^*]}\A_P^{opp}\operatorname{-mod}^G))
\rightarrow D^b(\Coh^G(\St_P))$ (note that $D^b(\Coh^G(\St_P))$ carries a canonical
filtered lifting as in \cite[Example A2]{Beilinson}).
We denote the composed functor
$$D^b_{I^\circ}(\Fl_P)\xrightarrow{\sim} D^b(\Perv(\A_\h\otimes_{\C[\g^*]}\A_P^{opp}\operatorname{-mod}^G))
\rightarrow D^b(\Coh^G(\St_P))$$
by $\tau_P$. That the diagram of Theorem
\ref{Prop:parabol_equiv_der} is commutative follows directly from the construction
of $\tau_P$.

\begin{Lem}\label{Lem:tau_P_equivalence}
The functor $\tau_P$ is an equivalence.
\end{Lem}
\begin{proof}
We need to show that $\tau_P$ is fully faithful and essentially surjective.

As we have pointed out in the proof of Proposition \ref{Prop:abel_equiv_parab},
the functor $\varphi_2$ is faithful.
On the other hand, it is the composition $\tau^{-1}\varphi_1\tau_P$. So
$\tau_P$ is faithful.

Now we show $\tau_P$ is full. 

Thanks to (\ref{eq:fin_dimension}), we need to show that $\dim \Hom(\mathcal{F}',\mathcal{G}')=
\dim \Hom(\tau_P\mathcal{F}',\tau_P\mathcal{G}')$
for all $\mathcal{F}',\mathcal{G}'\in D^b_{I^\circ}(\Fl_P)$. We have an isomorphism
$$\Hom(\varphi_1\tau_P\mathcal{F}',\varphi_1\tau_P\mathcal{G}')\xrightarrow{\sim} \Hom(\varphi_2\mathcal{F}',\varphi_2\mathcal{G}')$$
given by $\tau^{-1}$.
Since $\psi_i \varphi_i\cong \operatorname{id}\otimes H^*(P/B,\C)[2\dim P/B]$, see Lemma
\ref{Lem:psi_1_properties}, \ref{Lem:psi_2_properties}, we then get
an isomorphism
$$\Hom(\tau_P\mathcal{F}'\otimes H^*(P/B,\C),\tau_P\mathcal{G}')\xrightarrow{\sim} \Hom(\mathcal{F}'\otimes H^*(P/B,\C),\mathcal{G}').$$
Since $\dim \Hom(\mathcal{F}'[i],\mathcal{G}')\leqslant \Hom(\tau_P\mathcal{F}'[i],\tau_P\mathcal{G}')$
for all $i$ ($\tau_P$ is faithful), we see that $$\dim \Hom(\mathcal{F}',\mathcal{G}')= \Hom(\tau_P\mathcal{F}',\tau_P\mathcal{G}').$$
Therefore, $\tau_P$ is fully faithful.

Now let us show that $\tau_P$ is essentially surjective. Note that the perverse bimodule
t-structure is bounded:  for each $\mathcal{F}'\in
D^b(\Coh^G(\St_P))$ there are $i>j$ such that $\mathcal{F}'$ lies in the intersection
of $\leqslant i$ and $\geqslant j$ subcategories. The existence of $i$ is straightforward from
the definition of a perverse bimodule t-structure. Since the $\geqslant 0$-category is
obtained from $\leqslant 0$ by taking $R\Hom_{\A'}(\bullet,\A')$, where $\A'$
is a Gorenstein algebra, the existence of  $j$ follows as well.
Now we easily  prove that $\mathcal{F}'$ lies in the image of $\tau_P$ by induction on $i-j$.
The base, $i-j=0$, follows because $\tau_P$ is an equivalence between the hearts by the construction.
The induction step follows because $\tau_P$ is fully faithful.
\end{proof}

Theorem \ref{Thm:parab_perv_equiv} follows.
To finish the proof of Theorem \ref{Prop:parabol_equiv_der},
it remains to show two things: that
$\tau_P:D^b_{I^\circ}(\Fl_P)\xrightarrow{\sim} D^b(\Coh^G(\St_P))$
is equivariant for the action of $D^b_{pu}(I^\circ\backslash G^\vee((t))/I^\circ)
\cong D^b(\Coh^G(\St^\wedge_{\h}))$ and that
$\tau_P(\underline{\C}_{P^\vee/P^\vee})\cong\mathcal{O}_{Z^{diag}}$.

\begin{Lem}
The equivalence $\tau_P:D^b_{I^\circ}(\Fl_P)\xrightarrow{\sim} D^b(\Coh^G(\St_P))$
is equivariant.
\end{Lem}
\begin{proof}
Recall the algebra $\A_\h^\wedge:=\C[\h^*]^{\wedge_0}\widehat{\otimes}_{\C[\h^*]}\A_\h$.

Consider the negative parts
$D^{b,\leqslant 0}(\A^{\wedge}_{\h}\widehat{\otimes}_{\C[\g^*]}\A_\h^{\wedge,opp}\operatorname{-mod}^G)$
and $D_{perv}^{b,\leqslant 0}(\A_{\h}\widehat{\otimes}_{\C[\g^*]}\A_P^{opp}\operatorname{-mod}^G)$.
It is straightforward from the definition of the perverse bimodule t-structure that
$$D^{b,\leqslant 0}(\A^{\wedge}_{\h}\widehat{\otimes}_{\C[\g^*]}\A_\h^{\wedge,opp}\operatorname{-mod}^G)\otimes^L_{\A_\h}
D_{perv}^{b,\leqslant 0}(\A_{\h}\otimes_{\C[\g^*]}\A_P^{opp}\operatorname{-mod}^G)
\subset D_{perv}^{b,\leqslant 0}(\A_{\h}\otimes_{\C[\g^*]}\A_P^{opp}\operatorname{-mod}^G).$$
We can carry the action of $D^{b,\leqslant 0}(\A^{\wedge}_{\h}\widehat{\otimes}_{\C[\g^*]}\A_\h^{\wedge,opp}\operatorname{-mod}^G)$
to   $D^b_{I^\circ}(\Fl)$ using the equivalence $\tau$.

So we get the action of $\A^{\wedge}_{\h}\widehat{\otimes}_{\C[\g^*]}\A_\h^{\wedge,opp}\operatorname{-mod}^G$
on $$\Perv_{I^\circ}(\Fl), \Perv(\A_{\h}\otimes_{\C[\g^*]}\A^{opp}\operatorname{-mod}^G),
\Perv(\A_{\h}\otimes_{\C[\g^*]}\A_P^{opp}\operatorname{-mod}^G)$$
by right t-exact functors.
The full embedding $$\varphi_1:\Perv(\A_{\h}\otimes_{\C[\g^*]}\A_P^{opp}\operatorname{-mod}^G)
\hookrightarrow \Perv(\A_{\h}\otimes_{\C[\g^*]}\A^{opp}\operatorname{-mod}^G)$$
is equivariant by the construction. Since $\varphi_2:D^b_{I^\circ}(\Fl_P)
\hookrightarrow D^b_{I^\circ}(\Fl)$ is $D^b(\Coh^G\St^{\wedge}_\h)$-equivariant, we see
that $\varphi_2(\Perv_{I^\circ}(\Fl_P))\subset \Perv_{I^\circ}(\Fl)$ is
closed under the action of the category $\A^{\wedge}_{\h}\widehat{\otimes}_{\C[\g^*]}\A_\h^{\wedge,opp}\operatorname{-mod}^G$.
So $\Perv_{I^\circ}(\Fl_P)$ becomes an $\A^{\wedge}_{\h}\widehat{\otimes}_{\C[\g^*]}\A^{\wedge,opp}_\h\operatorname{-mod}^G$-module
category and $$\tau_P: \Perv_{I^\circ}(\Fl_P)\xrightarrow{\sim}
\Perv(\A_{\h}\otimes_{\C[\g^*]}\A_P^{opp}\operatorname{-mod}^G)$$
is $\A^{\wedge}_{\h}\otimes_{\C[\g^*]}\A^{\wedge,opp}_\h\operatorname{-mod}^G$-equivariant.
It follows from Lemma \ref{Lem:bimod_image} that $\tau_P$ is $\mathsf{Tilt}$-equivariant.
Since $K^b(\mathsf{Tilt})\xrightarrow{\sim}
D^b_{pu}(I^\circ\backslash G^\vee((t))/I^\circ)$, see property (3) in
Section \ref{SS_completed_cats}, we are done.
%
\end{proof}

\begin{Lem}
We have $\tau_P(\underline{\C}_{P^\vee/P^\vee})\cong \mathcal{O}_{Z^{diag}}$.
\end{Lem}
\begin{proof}
We have the following commutative diagram:

\begin{picture}(100,30)
\put(2,22){$D^b_{I^\circ}(\Fl_P)$}
\put(2,2){$D^b_{I^\circ}(\Fl)$}
\put(52,22){$D^b(\Coh^G(\St_P))$}
\put(52,2){$D^b(\Coh^G(\mathsf{St}_{B}))$}
\put(8,7){\vector(0,1){14}}
\put(65,7){\vector(0,1){14}}
\put(20,23){\vector(1,0){31}}
\put(18,3){\vector(1,0){33}}
\put(9,13){\tiny $\psi_2$}
\put(66,13){\tiny $\psi_1$}
\put(35,24){\tiny $\sim$}
\put(35,21){\tiny $\tau_P$}
\put(35,4){\tiny $\sim$}
\end{picture}

Consider the object $\underline{\C}_{1B^\vee}\in D^b_{I^\circ}(\Fl)$.
We have $\psi_2(\underline{\C}_{1B^\vee})=\underline{\C}_{1P^\vee}[\dim P/B]$.
Furthermore, under the equivalence $\tau:D^b_{I^\circ}(\Fl)\xrightarrow{\sim}
D^b(\Coh^G(\St_B))$ the object $\underline{\C}_{1B^\vee}$
goes to $\mathcal{O}_{\tilde{\Nilp}_{diag}}$, where we write
$\tilde{\Nilp}_{diag}$ for the diagonal in $\St_B$, this is a special
case of Lemma \ref{Lem:der_equiv_image} (for $P=B$).

So we need to   prove that
$\psi_1(\mathcal{O}_{\tilde{\Nilp}^{diag}})=\mathcal{O}_{Z^{diag}}[\dim P/B]$, equivalently,
$$\tilde{\varpi}_*\circ \tilde{\iota}^*(\mathcal{O}_{\tilde{\Nilp}^{diag}}(\rho,-\rho))=
\mathcal{O}_{Z^{diag}}.$$

First, note that $\mathcal{O}_{\tilde{\Nilp}^{diag}}(\rho,-\rho)=\mathcal{O}_{\tilde{\Nilp}^{diag}}$.
Next, let us compute the pull-back of $\mathcal{O}_{\tilde{\Nilp}^{diag}}$
to $\widehat{Z}=\tilde{\g}\times_{\g}Z$. The intersection of $\tilde{\Nilp}^{diag}$
and $\widehat{Z}$ is transversal and equals to $Z^{diag}$
so the pull-back, $(\operatorname{id}\boxtimes \iota)^*(\mathcal{O}_{\tilde{\Nilp}^{diag}})$,
is $\mathcal{O}_{Z^{diag}}$. Finally, the restriction of $\tilde{\varpi}$ to
$Z^{diag}$ is the closed embedding $Z^{diag}\hookrightarrow \St_{P}$.
So we see that $\psi_1(\mathcal{O}_{\tilde{\Nilp}^{diag}})=\mathcal{O}_{Z^{diag}}[\dim P/B]$.
\end{proof}

This finishes the proof of Theorem \ref{Prop:parabol_equiv_der}.

\begin{proof}[Proof of Corollary \ref{Thm:abelian_quotient}]
Let us prove (1). We write $\tilde{\Orb}$ for the open orbit in $\tilde{\Nilp}_P$,
this is a $G$-equivariant cover of $\Orb$.

We first check  that the equivalence
\begin{equation}\label{eq:one_more_equiv}
D^b(\Coh^{Z_P(e)}\B_e)\xrightarrow{\sim}D^b(\Coh^G(\tilde{\g}\times_{\g}\tilde{\Orb}))
\end{equation}
is t-exact. The heart of the t-structure on the source
is $\A_{\h,e}\operatorname{-mod}^{Z_P(e)}$, while the heart
in the target is naturally identified with $\A_{\h,e}\otimes \A_{P,e}^{opp}\operatorname{-mod}^{Z_P(e)}$. Both  categories $$D^b(\Coh^{Z_P(e)}\B_e),D^b(\Coh^G(\tilde{\g}\times_{\g}\tilde{\Orb}))$$ are the derived categories
of their hearts.
But $\A_{P,e}=\End(\mathcal{T}_{P,e}(-2\rho))$ and the restriction of (\ref{eq:one_more_equiv})
to the heart is just tensoring with $\mathcal{T}_{P,e}(-2\rho)^*$. So it is t-exact.

What remains to check to prove (1) is that
$D^b_{I^\circ}(\Fl_P)\twoheadrightarrow D^b(\Coh^G(\tilde{\g}\times_{\g}\tilde{\Orb}))$
is t-exact. This follows directly from Theorem \ref{Thm:parab_perv_equiv}.

(2) is a direct corollary of the equivariance part from Theorem
\ref{Prop:parabol_equiv_der}. To prove (3) note that, thanks to
Theorem \ref{Prop:parabol_equiv_der}, the image of $\underline{\C}_{P^\vee/P^\vee}$ in
$D^b(\Coh^{Z_G(e)}\B_e)$ is the restriction of $\Str_{Z^{diag}}$ to $\B_e$. But the (derived scheme theoretic)
intersection of $Z^{diag}$ with $\B_e$ is precisely $\B_{\mathfrak{m}}$.
\end{proof}

\section{Duality}\label{S_duality}
\subsection{Notation and content}
The notation $G,\g,\h,\mathfrak{b},e,h,f, \nu,L, \g^i,\underline{P},P, \mathfrak{m},\mathfrak{m}^-,T_0$
has the same meaning as in Section \ref{SS_notation_basic}.
 We fix root generators $e_\alpha, f_\alpha\in \g$
for each simple root $\alpha$. Recall ${\underline{Q}}:=Z_{\underline{G}}(e,h,f)$.

Recall that we write $W^a$ for the (extended) affine Weyl group of $G$. By $\rho_L$ we denote
half the sum of positive roots for $L$.

We continue to write $p$ for a sufficiently large prime and $\chi\in \g_\F^{(1)*}$
for the reduction of $(e,\cdot)$ mod $p$. Recall the splitting bundle
$\mathcal{V}^\chi_\mu$ on $\tilde{\g}_\F^{(1)}\times_{\g_\F^{*(1)}}\g_\F^{*(1)\wedge_\chi}$,
see (\ref{eq:splitting_Springer}). It carries a ${\underline{Q}}_\F$-equivariant structure as explained
in Section \ref{SS_splitting_equivar}.

In this section we introduce a duality functor, a contravariant t-exact self-equivalence $\D$
of $D^b(\U^\chi_\F\operatorname{-mod}^{\underline{Q}})$ and study its properties. Section
\ref{SS_dual_basic} defines this functor and states (and mostly proves) its
basic properties. The most important property is that $\D$ fixes the
$K_0$-classes of $\chi$-Weyl modules, Proposition \ref{Prop:K_0_identity}.
This proposition is proved in the next two sections: in Section
\ref{SS_K_0_distinguished} we treat the case when $\chi$ is distinguished
and then in Section \ref{SS_K_0_general} we deal with the general case.
Finally, in Section \ref{SS_loc_dual} we study an interplay between
$\D$ and the derived localization equivalence.

\subsection{Duality functor: construction and basic properties}\label{SS_dual_basic}
For the time being, we are working over $\C$.
Consider the  standard anti-involution $\sigma$ of $\g$ defined by $\sigma_{\h}=\operatorname{id},
\sigma(e_\alpha)=f_\alpha, \sigma(f_\alpha)=e_\alpha$
for each simple root $\alpha$. As was checked in \cite[Section 2.6]{W_dim}, one can replace $(e,h,f)$
with a conjugate triple in such  a way that
$\sigma(e)=f, \sigma(f)=e, h$ is dominant. If $n$ is the image in $G$
of the matrix $\begin{pmatrix}0&i\\i&0\end{pmatrix}\in \operatorname{SL}_2$, then $\varsigma:x\mapsto \operatorname{Ad}(n)\sigma(x)$
is still an anti-involution that now fixes $e,f$ (and maps $h$ to $-h$). Also note that
$\varsigma$ lifts to an anti-involution of $G$. We can assume that $T_0$ is
$\varsigma$-stable.  Clearly, $\varsigma$
maps $\g^i$ to $\g^{-i}$. It also fixes the parabolic subalgebra $\underline{\p}\subset
\underline{\g}$.
And, of course, we can reduce $\varsigma$ mod $p$.
Note also that $\varsigma$ fixes ${\underline{Q}}$.

For a module $M\in (\U^\chi_\F)^{opp}\operatorname{-mod}^{\underline{Q}}$, consider its twist
with $\varsigma$ and denote it by $\,^\varsigma\! M$, it is an object of $\U^\chi_\F\operatorname{-mod}^{\underline{Q}}$.
Also consider the ${\underline{Q}}_\F$-module $\Lambda^{top}(\underline{\g}_\F/\underline{\mathfrak{p}}_\F)$.
We set $\,^{\tw}\! M:=(\,^\varsigma\!M)\otimes \Lambda^{top}(\underline{\g}_\F/\underline{\mathfrak{p}}_\F)$. We note that tensoring with
$\Lambda^{top}(\underline{\g}_\F/\underline{\mathfrak{p}}_\F)$ does not affect
the action of $(\U^\chi_\F)^{opp}$.
The similar definition of $\,^{\tw}\! M'$ makes sense for $M'\in \U^\chi_\F\operatorname{-mod}^{\underline{Q}}$.

%

\begin{Lem}\label{Lem:involution_naive}
We have $\,^{\tw}\!(\,^{\tw}\!M)\cong M$.
\end{Lem}
\begin{proof}
Since $\varsigma^2=\operatorname{id}$, the claim of the lemma reduces to checking  that
\begin{equation}\label{eq:rep_iso}\Lambda^{top}(\underline{\g}_\F/\underline{\mathfrak{p}}_\F)^*\cong \,^\varsigma\! \Lambda^{top}(\underline{\g}_\F/\underline{\mathfrak{p}}_\F)
\end{equation}
First, we notice that (\ref{eq:rep_iso}) reduces to $G=\underline{G}$.
Note that $Z(G_\F)$ acts trivially on $\Lambda^{top}(\underline{\g}_\F/\underline{\mathfrak{p}}_\F)$.
We can replace $G_\F$ with $G_\F/Z(G_\F)$ and assume that $G_\F$ is semisimple and of adjoint type.
Here ${\underline{Q}}=A$. Also (\ref{eq:rep_iso}) reduces to the case when $G_\F$ is simple.

 The group $A$ is either the sum of
several copies of $\Z/2\Z$ (classical types) or $S_1,S_2,S_3,S_4,S_5$
(exceptional types), see \cite[Sections 6.1,8.4]{CM}. For all these groups the square
of any one-dimensional representation is trivial. For $A=S_i$, there is a
unique nontrivial one-dimensional representation and (\ref{eq:rep_iso}) follows.  It remains
to prove that if $\g$ is of classical type (hence $A$ is abelian),
then
\begin{itemize}\item[(*)] $a\mapsto \varsigma(a)^{-1}$ is an inner
automorphism of $A$.
\end{itemize} This will imply $\varsigma(a)=a^{-1}$ and (\ref{eq:rep_iso}) will follow.

To prove (*) we consider the standard antiautomorphism $\sigma':\g\rightarrow\g^{opp},
x\mapsto -x$. Assume first that  anti-automorphism $\sigma^{-1}\sigma'$ is an inner
automorphism of $\g$. Then consider the element $n'\in G$, the image of
$\operatorname{diag}(i,-i)\in \operatorname{SL}_2$ and set
$\varsigma':=\operatorname{Ad}(n')\circ \sigma'$. Note that $\varsigma'(e)=\varsigma(e)=e,\varsigma'(f)=f,
\varsigma'(h)=-h$. It follows that $\varsigma=\operatorname{Ad}(a')\circ \varsigma'$ for
some $a'\in A$. Since $A$ is commutative and $\varsigma'(a)=a^{-1}$ for all $a\in A$,
(*) follows.

Now we need to consider the situation when $\sigma^{-1}\sigma'$  is an outer automorphism.
In type A, the group $A$ is trivial, so we only need to consider type $D_n$ (with odd $n$).
Here $\sigma^{-1}\sigma'$ is induced by an element of $\operatorname{O}_{2n}$. Now we can
run the argument in the previous paragraph replacing $G$ with $\operatorname{O}_{2n}$
and still arrive at the same conclusion.
\end{proof}

Since $\U_\F^{\chi}\operatorname{-mod}$ consists of finite dimensional $\U_\F$-modules,
we see that the functor $$R\Hom_{\U_\F}(\bullet,\U_\F)[\dim \g]$$ is $t$-exact.
Define $\D: \U^\chi_\F\operatorname{-mod}^{\underline{Q}}\rightarrow \U^\chi_\F\operatorname{-mod}^{{\underline{Q}},opp}$
 by
\begin{equation}\label{eq:D_definition}
\D(M)=\,^{\tw}\! R\Hom_{\U_\F}(M,\U_\F)[\dim \g].
\end{equation}
The following lemmas establish basic properties of the functor $\D$.

\begin{Lem}\label{Lem:dual_square}
We have $\D^2\cong \operatorname{id}$.
\end{Lem}
\begin{proof}
We have $\!R\Hom_{\U_\F^{opp}}(\,^{\varsigma}\!M,\,^\varsigma\U_\F)\cong
\,^{\varsigma}\!R\Hom_{\U_\F}(M,\U_\F)$, where we write $\,^\varsigma\U_\F$
for the bimodule with both left and right action twisted. But $\U_\F\cong \,^\varsigma\U_\F$
as a $\U_\F$-bimodule because $1\in \,^\varsigma\U_\F$ is a central
generator. We deduce that $\!R\Hom_{\U_\F^{opp}}(\,^{\tw}\!M,\,\U_\F)\cong
\,^{\tw}\!R\Hom_{\U_\F}(M,\U_\F)$
Combining  this with the fact
that $R\Hom_{\U_\F}(\bullet,\U_\F)$ is an involution and Lemma
\ref{Lem:involution_naive}
we get $\D^2\cong \operatorname{id}$.
\end{proof}

\begin{Lem}\label{Lem:dual_HC_character}
$\D$ maps $\U^\chi_{(\lambda),\F}\operatorname{-mod}^{{\underline{Q}}}$
to $\U^\chi_{(\lambda),\F}\operatorname{-mod}^{{\underline{Q}}}$ for every HC character
$\lambda$.
\end{Lem}
\begin{proof}
This follows from the classical fact that the principal anti-involution $\sigma$
acts trivially on the HC center.
\end{proof}

Here is another useful property of $\D$ describing the interaction of
this functor with the categorical braid group action.

\begin{Prop}\label{Prop:duality_braid_action}
We have $\D\circ T_x\cong T^{-1}_{x^{-1}}\circ \D$ for any $w\in W^{a}$.
\end{Prop}
\begin{proof}
Consider the case when $x$ is a simple affine reflection $s$,
in which case $T_s$ is a classical
wall-crossing functor. By what was recalled in Section \ref{SS_transl_braid},
the functor $T_s^{-1}$ is
given by $\mathsf{T}^*\circ \mathsf{T}(\bullet)\rightarrow \bullet$, with the
target functor  in  cohomological degree $1$. So we need to show that
$\D\circ\mathsf{T}\cong \mathsf{T}\circ \D$. The functor $\mathsf{T}$
has the form $\operatorname{pr}_{\lambda'}(V\otimes \bullet)$, for a suitable
central character $\lambda'$ (singular with the singularity corresponding to $s$)
and a suitable finite dimensional irreducible $G$-module $V$. By  Lemma
\ref{Lem:dual_HC_character}, $\D$ commutes with the functors $\operatorname{pr}_{\lambda'}$
so it remains to show that $V\otimes\D(\bullet)\cong \D(V\otimes \bullet)$.
Clearly, $\D(V\otimes \bullet)\cong \,^\varsigma\!(V^*)\otimes \D(\bullet)$, where
in the right hand side $V^*$ is viewed as a right $G$-module and so
the twist $\,^\varsigma\!(V^*)$ is a left $G$-module. And
$\,^\varsigma\!(V^*)\cong \,^\sigma\!(V^*)\cong V$. This finishes the proof.

Similarly, one checks that $\D$ commutes with the length $0$ elements in $W^a$,
those act by translation equivalences.
Since the simple affine reflections $s$ and the length zero elements generate
$W^a$, we are done.
\end{proof}

Finally, let us state the main result of this section.
In particular, it explains a  reason why we twist with $\varsigma$ and not with some other anti-involution of $\g$.

\begin{Prop}\label{Prop:K_0_identity}
Let $\lambda=0$. For all $x\in W^{a,P}$, we have $[\D W^\chi_\F(\mu_x)]=[W^\chi_\F(\mu_x)]$.
\end{Prop}
Recall that here $\mu_x=x^{-1}\cdot (-2\rho)$.

Below, Section \ref{SS_canonical}, we will see that this proposition implies that
$\D$ gives the identity map on $K_0(\U^\chi_{(0),\F}\operatorname{-mod}^{\underline{Q}})$.

This proposition will be proved in the next two sections.

\subsection{Behavior on $K_0$, distinguished case}\label{SS_K_0_distinguished}
Throughout this section we assume that $\chi$ is distinguished.

We start with a series of lemmas.

\begin{Lem}\label{Lem:disting_K_0_D}
Let $\lambda=0$.
The following  claims are equivalent:
\begin{enumerate}
\item $\D W^\chi_\F(2\rho_L-2\rho)\cong W^\chi_\F(2\rho_L-2\rho)$.
\item $[\D W^\chi_\F(2\rho_L-2\rho)]\cong [W^\chi_\F(2\rho_L-2\rho)]$.
\item $[\D W^\chi_\F(\mu_x)]\cong [W^\chi_\F(\mu_x)]$
for all $x\in W^{a,P}$.
\end{enumerate}
\end{Lem}
\begin{proof}
We have $\dim W^\chi_\F(0)=p^{\dim \Pcal}$. Since $\dim \Pcal=\frac{1}{2}\dim G_\F\chi$,
we use the main result of \cite{Premet} (proving the Kac-Weisfeiler conjecture) to deduce that $W^\chi_\F(0)$ is irreducible.
So (1) and (2) are equivalent.

Proposition \ref{Prop:duality_braid_action} implies that
$[\D]$ acts on $K_0(\U^\chi_{(0),\F}\operatorname{-mod}^{\underline{Q}})$
by a $W^a$-linear automorphism. The equivalence
(2)$\Leftrightarrow$(3) now follows from
Lemma \ref{Lem:W_aff_action}.
\end{proof}

Here is another technical result that we are going to need.

\begin{Lem}\label{Lem:chi_Weyl_coincide}
We have $W^\chi_\F(2\rho_L-2\rho)\xrightarrow{\sim} W^\chi_\F(0)$.
\end{Lem}
\begin{proof}
Note that $W^\chi_\F(\mu)$ is the specialization of the parabolic Verma module
$\Delta^P_\F(\mu)$ to $\chi\in \g^{(1)*}_\F$. Consider the parabolic Verma
$\Delta^P_{\C}(2\rho_{L}-2\rho):=U(\g)\otimes_{U(\p)}\C_{2\rho_{L}-2\rho}$.

We claim that
$\Delta^P_\C(2\rho_{L}-2\rho)\hookrightarrow \Delta^P_\C(0)$ and the quotient
does not have  full support in $\mathfrak{m}^{-*}$. Note that $\Delta^P_\C(2\rho_{L}-2\rho)$ is simple.
The annihilators of  $\Delta^P_\C(2\rho_{L}-2\rho), \Delta^P_\C(0)$ in $U(\g)$
coincide because both coincide with the kernel $J$ of $U(\g)\twoheadrightarrow D(G/P)$,
see \cite[Section 3.6]{BorBr}. The highest weight $2\rho_{L}-2\rho$ corresponds
to the longest element $\underline{w}_0$ of $W_{\underline{G}}$. This is a Duflo involution.
So, according to \cite[Satz 7.11]{Jantzen},  the socle of $\Delta_\C(0)/J\Delta_\C(0)$ is $\Delta^P_\C(2\rho_{L}-2\rho)$
and the quotient of $\Delta_\C(0)/J\Delta_\C(0)$ by the socle has GK dimension smaller
than that of $\Delta^P_\C(2\rho_{L}-2\rho)$. Since  the GK dimensions of
$\Delta^P_\C(2\rho_{L}-2\rho)$ and $\Delta^P_\C(0)$ are the same,
it follows that the natural epimorphism $\Delta_\C(0)/J\Delta_\C(0)\twoheadrightarrow \Delta^P_\C(0)$
is actually an isomorphism. This yields the required embedding
$\Delta^P_\C(2\rho_{L}-2\rho)\hookrightarrow \Delta^P_\C(0)$.

The embedding $\Delta^P_{\C}(2\rho_{L}-2\rho)\hookrightarrow \Delta^P_\C(0)$ is defined over
a finite localization of $\Z$. Since $p$ is large enough, we get $\Delta^P_{\F}(2\rho_{L}-2\rho)
\hookrightarrow \Delta^P_\F(0)$ and the quotient does not have   full support in $\mathfrak{m}_\F^{-,(1)*}$.
This support is  closed and $P_\F^{(1)}$-stable. So the support
of the quotient does not contain $\chi$. We conclude that $W^\chi_\F(2\rho_{L}-2\rho)
\hookrightarrow W^\chi_\F(0)$. Since the dimensions coincide, this embedding is an isomorphism.
\end{proof}

\begin{proof}[Proof of Proposition
\ref{Prop:K_0_identity} for distinguished $\chi$]

The proof will be in several steps.

{\it Step 1}. First, we prove that $\D M=\,^\tw\! \Hom_{\U^\chi_\F}(M,\U^\chi_\F)$.

By the 5-lemma, $$R\Hom_{\U_\F}(M,\U_\F)[\dim \g]=\Hom(M, R\Hom_{\U_\F}(\U^\chi_\F,\U_\F)[\dim \g])$$
It follows that $$\D M=\,^\tw\!\Hom(M, R\Hom_{\U_\F}(\U^\chi_\F,\U_\F)[\dim \g]).$$
So we just need to prove that the $\U^\chi_\F$-bimodules
$\U^\chi_\F$ and $R\Hom_{\U_\F}(\U^\chi_\F,\U_\F)[\dim \g]$ are $A$-equivariantly isomorphic.
Let $n:=\dim \g$ and let $V$ denote the subspace in $S(\g_{\F}^{(1)})$ generated
by the elements $x-\langle \chi,x\rangle$ for $x\in \g_\F^{(1)}$. This subspace is
$A$-stable. Consider the Koszul complex for $V$ acting on $\U_\F$ by multiplication,
denote it by $K_\bullet(V,\U_\F)$. It is an $A$-equivariant bimodule resolution of $\U^\chi_\F$
by free left $\U_\F$-modules.  Then we have an $A$-equivariant
$\U_\F$-bilinear isomorphism
\begin{equation}\label{eq:iso_Koszul} R\Hom_{\U_\F}(\U^\chi_\F, \U_\F)[\dim\g]\cong K_\bullet(V, \U_\F\otimes \Lambda^{top}V^*)\end{equation}
Now we observe that the action of $A$ on $V$ is via a homomorphism
$A\rightarrow \operatorname{SO}(V)$. This is because the action is isomorphic
to the action on $A$ on $\g_\F^{(1)}$, which factors through $G_\F^{(1)}$.
It follows that $\Lambda^{top}V^*$ is the trivial $A$-module. So the right hand side
of (\ref{eq:iso_Koszul}) is nothing else but $\U^\chi_\F$. Hence
$\U^\chi_\F\cong R\Hom_{\U_\F}(\U^\chi_\F,\U_\F)[\dim \g]$.

{\it Step 2}. Both $W^\chi_\F(0), \D W^\chi_\F(0)$ are simple modules.
By Lemma \ref{Lem:chi_Weyl_coincide}, $W^\chi_\F(0)=W^\chi_\F(2\rho_{L}-2\rho)$.
To prove the proposition we need to show that
$$\Hom_{\U^\chi_\F}\left(W^\chi_\F(2\rho_{L}-2\rho), \D W^\chi_\F(0)\right)^A\neq 0.$$
We write $\F_0$ for the one-dimensional trivial $\p_\F$-module.
By Step 1,
\begin{align*}&\D W^\chi_\F(0)=\,^\tw\! \Hom_{\U^\chi_\F}(\U^\chi_\F\otimes_{U^0(\p_\F)}\F_{0},\U^\chi_\F)=\\
&^\varsigma\! \Hom_{U^0(\p_\F)}(\F_{0},\U^\chi_\F)\otimes \Lambda^{top}(\g_F/\p_\F)=
\Hom_{U^0(\p_\F)^{opp}}(\F_{0}, \U^\chi_\F)\otimes \Lambda^{top}(\g_\F/\p_\F).
\end{align*}
The last equality holds because $^\varsigma \U^\chi_\F\cong \U^\chi_\F$
(as discussed in the proof of Lemma \ref{Lem:dual_square}) and $\,^\varsigma\F_{0}\cong \F_{0}$. So
\begin{align*}&\Hom_{\U^\chi_\F\operatorname{-mod}^A}(W^\chi_\F(2\rho_{L}-2\rho), \D W^\chi_\F(0))=\\&\left[\Hom_{U^0(\p_\F)}\left(\F_{2\rho_{L}-2\rho}, \Hom_{U^0(\p_\F)^{opp}}(\F_{0}, \U^\chi_\F)\right)\otimes \Lambda^{top}(\g_\F,\p_\F)\right]^A=\\
&\left[\Hom_{U^0(\p_\F)\otimes U^0(\p_\F)^{opp}}\left(\F_{2\rho_{L}-2\rho}\otimes \F_{0},\U^\chi_\F\right)\otimes \Lambda^{top}(\g_\F/\p_\F)\right]^A.
\end{align*}
Note that $\Lambda^{top}(\g_\F/\p_\F)\cong \Lambda^{top}(\p_\F)$. Since $U^0(\p_\F)$
is an $A$-stable subalgebra of $\U^\chi_\F$, the space
$\Hom(W^\chi_\F(2\rho_{L}-2\rho), \D W^\chi_\F(0))$ contains
$$\left[\Hom_{U^0(\p_\F)\otimes U^0(\p_\F)^{opp}}(\F_{2\rho_{L}-2\rho}\otimes \F_{0},U^0(\p_\F))\otimes \Lambda^{top}(\p_\F)\right]^A.$$
Below in this proof we will see that the latter space is nonzero.

{\it Step 3}. Let $\mathfrak{u}$ be the Lie algebra of a unipotent algebraic group $U$
over $\F$. We claim that
there is a unique (up to rescaling) element $x(\mathfrak{u})\in U^0(\mathfrak{u})$ annihilated by $\mathfrak{u}$ on the left and on the right. The proof is by induction on $\dim \mathfrak{u}$.
Namely, set $\mathfrak{u}_1:=\mathfrak{u}/\mathfrak{z}(\mathfrak{u})$. Then
$U^0(\mathfrak{u})\twoheadrightarrow U^0(\mathfrak{u}_1)$. Let $x'$ be a lift
of $x(\mathfrak{u}_1)$ to $U^0(\mathfrak{u})$ and let $y_1,\ldots,y_k$ be a basis
in $\mathfrak{z}(\mathfrak{u})$. The element $x(\mathfrak{u}):=x'\prod_{i=1}^k y_i^{p-1}$
is independent of the choice of $x'$.  It is annihilated by $\mathfrak{u}$
on the left and on the right. On the other hand, if $x$ is annihilated by
$\mathfrak{u}$ on the left and on the right it must have the form
$x''\prod_{i=1}^ky_i^{p-1}$, where $x''\in U^0(\mathfrak{u})$ is such that
the projection of $x''$ to $U^0(\mathfrak{u}_1)$
is annihilated by $\mathfrak{u}_1$ on the left and on the right. So it must be proportional
to $x(\mathfrak{u}_1)$. This implies the claim in the beginning of the paragraph.

{\it Step 4}. We still assume that $\mathfrak{u}$ is the Lie algebra of  a unipotent
algebraic group. Let $y_1,\ldots, y_N$ be a basis of $\mathfrak{u}$. Then
$x(\mathfrak{u})$ is proportional to  $\prod_{i=1}^N y_i^{p-1}$. Indeed, this follows from
Step 3 for  a special choice of basis and it is easy to see that $\prod_{i=1}^N y_i^{p-1}$ is
independent of the choice of  a basis up to a scalar multiple. It follows that if $S$ is an algebraic
group acting on $\mathfrak{u}$ by algebraic Lie algebra automorphisms, then
$$s x(\mathfrak{u})=\chi_{\Lambda^{top}(\mathfrak{u})}(s)^{p-1} x(\mathfrak{u}),$$
where $\chi_{\Lambda^{top}(\mathfrak{u})}$ is the character of the $S$-action
in $\Lambda^{top}(\mathfrak{u})$.

{\it Step 5}.  Now consider an algebraic group $F=S\ltimes U$ over $\F$, where $S$
is connected reductive and $U$ is unipotent. Assume that there is
$x(\mathfrak{s})\in U^0(\mathfrak{s})$ that is annihilated by $\mathfrak{s}$
on the left and on the right. It follows from Step 4, that the element
$x(\mathfrak{f}):=x(\mathfrak{u})x(\mathfrak{s})\in U^0(\mathfrak{f})$ is annihilated by
$\mathfrak{f}$ on the right, while for any $y\in \mathfrak{f}$ we have
$$y x(\mathfrak{f})=-\langle \chi_{\Lambda^{top}(\mathfrak{u})},y\rangle x(\mathfrak{f}).$$
So to prove the claim in the end of Step 2 (where $\mathfrak{f}=\mathfrak{p}$),
it remains to check that
$x(\mathfrak{s})$ indeed exists.

{\it Step 6}.  Let $\mathfrak{s}=\mathfrak{n}^-\oplus \mathfrak{h}\oplus \mathfrak{n}$
be the triangular decomposition for $\mathfrak{s}$. Let $z_1,\ldots,z_r$ be an integral basis in
$\mathfrak{h}$. For $k\in \F_p$, set
$F_k(z):=(\prod_{i=1}^{p}(z-i))/(z-k)\in \F[z]$.  We claim that  $$x(\mathfrak{s}):=x(\mathfrak{n})(\prod_{i=1}^r F_{\langle 2\rho,z_i\rangle}(z_i))x(\mathfrak{n}^-)$$
satisfies the required properties.  Note that $\prod_{i=1}^r F_{\langle 2\rho,z_i\rangle}(z_i)$
does not depend on the choice of $z_1,\ldots,z_r$ (up to rescaling): it is the unique
element in $U^0(\mathfrak{h})$ annihilated by $z-\langle 2\rho,z\rangle$ for all
$z\in \mathfrak{h}$.

It is enough to show that $x(\mathfrak{s})$
is annihilated by the Cartan generators $e_i,f_i$ and also by $\mathfrak{h}$
on the left and on the right. For $\mathfrak{h}$, this is clear.
Let $\mathfrak{n}_0$ be the $H$-stable complement of $\F e_i$ in $\mathfrak{n}$
and $\mathfrak{n}_0^-$ have the similar meaning.  By Step 4, we have
$$x(\mathfrak{s})=x(\mathfrak{n}_0)\left[e_i^{p-1}(\prod_{i=1}^r F_{\langle 2\rho,z_i\rangle}(z_i))f_i^{p-1}\right]x(\mathfrak{n}^-_0).$$
The elements $x(\mathfrak{n}_0), x(\mathfrak{n}^-_0)$ commute with both
$e_i,f_i$ by Step 4. So we need to check that $e_i,f_i$ annihilate the middle bracket.
This reduces the computation to the case of $\mathfrak{s}=\mathfrak{sl}_2$:
we choose $z_1=\alpha_i^\vee$ and all other $z_i$ vanishing on $\alpha_i$.

In the case of $\mathfrak{sl}_2$ what we need to check is that
$f e^{p-1}(\prod_{i\neq 2} (h-i))f^{p-1}=0$ and $e^{p-1}(\prod_{i\neq 2} (h-i))f^{p-1}e=0$ in $U^0(\mathfrak{sl}_2)$. The first equality easily follows from $fe^{p-1}=e^{p-2}(h-2)+e^{p-1}f$.
The second is analogous.
\end{proof}

\subsection{Behavior on $K_0$, general case}\label{SS_K_0_general}
Now let us discuss a compatibility between $\D$ and $\underline{\D}$,
the similarly defined functor for $\underline{\U}^\chi_\F\operatorname{-mod}^{\underline{Q}}$.
Let $\underline{\Delta}_\nu$ and $\underline{\nabla}_\nu$ denote the baby Verma
and {\it dual baby Verma}
functors $\underline{\U}_\F\operatorname{-mod}^{\underline{Q}}
\rightarrow \U_\F\operatorname{-mod}^{\underline{Q}}$, the latter is defined
by $$\underline{\nabla}_\nu(M):=\Hom_{U^\chi(\mathfrak{g}^{\leqslant 0}_\F)}(\U^\chi_\F,M).$$

\begin{Prop}\label{Prop:D_baby}
We have $\D\circ \underline{\Delta}_\nu\cong \underline{\nabla}_\nu\circ \underline{\D}$,
an isomorphism of functors $\underline{\U}^\chi_\F\operatorname{-mod}^{{\underline{Q}}}
\rightarrow \U^\chi_\F\operatorname{-mod}^{{\underline{Q}},opp}$.
\end{Prop}
\begin{proof}
First, we compute $R\Hom_{\U_\F}(\underline{\Delta}_\nu(\underline{\U}_\F),\U_\F)$. Let $\underline{\Delta}^r_\nu$
denote the analog of the functor $\underline{\Delta}$ for the categories of right modules (with the
same parabolic). We claim that
\begin{equation}\label{eq:duality_Verma}
R\Hom_{\U_\F}(\underline{\Delta}_\nu(\underline{\U}_\F),\U_\F)[2\dim \mathfrak{g}^{>0}]\cong
\underline{\Delta}^r_\nu\left(\underline{\U}_\F\otimes \left[\Lambda^{top}(\mathfrak{g}^{>0}_\F)^{\otimes 1-p}\right]^*\right).\end{equation}
This is an isomorphism of right $(\U_\F,\underline{G}_{\F})$-modules (in particular, we view
$\left[\Lambda(\mathfrak{g}^{>0}_{\F})^{\otimes 1-p}\right]^*$ as a right $\underline{G}_{\F}$-module).
To prove (\ref{eq:duality_Verma}) we first
note that $\underline{\Delta}_\nu(\underline{\U}_\F)$ is the quotient of $\U_\F$ by
the ideal generated by $\g^{>0}_{\F}\oplus \g_{\F}^{<0,(1)}$.
Consider the Chevalley-Eilenberg complex of left modules associated with this Lie subalgebra of
generators.  We  denote this complex by  $\mathsf{CE}_\ell(\g^{>0}_{\F}\oplus\g_{\F}^{<0,(1)}, \U_\F)$.
This is an $L_\F$-equivariant resolution of  $\underline{\Delta}_{\nu}(\underline{\U}_{\F})$
by free left $\U_\F$-modules.

It follows that we have a quasi-isomorphism
$$R\Hom_{\U_\F}(\underline{\Delta}_\nu(\underline{\U}_{\F}),\U_\F)[2\dim \mathfrak{g}^{>0}]\cong\mathsf{CE}_r(\g^{>0}_{\F}\oplus\g_{\F}^{<0,(1)}, \U_\F\otimes \left[\Lambda^{top}(\mathfrak{g}^{>0}_{\F})^{\otimes 1-p}\right]^*),$$
where on the right we have the Chevalley-Eilenberg complex of right modules. This complex
is quasi-isomorphic to the right hand side of (\ref{eq:duality_Verma}).
This proves (\ref{eq:duality_Verma}).

In particular, (\ref{eq:duality_Verma}) implies that
\begin{equation}\label{eq:dual_iso1}
R\Hom_{\U_\F}(\underline{\Delta}_\nu(M),\U_\F)=\underline{\Delta}^r_\nu
\left(R\Hom_{\underline{\U}_{\F}}(M,\underline{\U}_{\F}\otimes
\left[\Lambda^{top}(\mathfrak{g}^{>0}_{\F})^{\otimes 1-p}\right]^*)\right)[2\dim \mathfrak{g}^{>0}].
\end{equation}
Now let us see what twisting by $\varsigma$ does. First of all, $\,^\varsigma\!\underline{\Delta}_\nu^r(\bullet)\cong \underline{\Delta}_{-\nu}(\,^\varsigma\bullet)$, where  $\underline{\Delta}_{-\nu}$ stands for the baby Verma module functor
for the opposite parabolic $\mathfrak{g}^{\leqslant 0}$.  Since
$\varsigma$ is the identity on the character lattice of $\underline{G}_{\F}$,
we have the following isomorphism of left $G_\F$-modules
$$\,^\varsigma\!\left[\Lambda^{top}(\mathfrak{g}^{>0}_{\F})^{\otimes 1-p}\right]^*\cong \Lambda^{top}(\mathfrak{g}^{>0}_{\F})^{\otimes 1-p}.$$

So (\ref{eq:dual_iso1}) yields
\begin{equation}\label{eq:dual_iso2}
\D(\underline{\Delta}_\nu(M))=\underline{\Delta}_{-\nu}(\underline{\D}(M)\otimes \Lambda^{top}(\mathfrak{g}^{>0}_{\F})^{\otimes 1-p}).
\end{equation}

Note that $\underline{\D}(M)$ is the maximal $\nu$-weight subspace
in $\underline{\Delta}_{-\nu}\left(\underline{\D}(M)\otimes \Lambda^{top}(\g^{>0}_{\F})^{\otimes 1-p}
\right)$.
This gives a homomorphism
\begin{equation}\label{eq:delta_nabla_homom}
\underline{\Delta}_{-\nu}(\underline{\D}(M)\otimes \Lambda^{top}(\g^{>0}_{\F})^{\otimes 1-p})\rightarrow
\underline{\nabla}_{\nu}(\underline{\D}(M))\end{equation}
that is the identity on $\underline{\D}(M)$.

Now we show that every  $\nu$-graded $\mathfrak{g}^{>0}_\F$-submodule
of $\underline{\Delta}_{-\nu}\left(\underline{\D}(M)\otimes \Lambda^{top}(\g^{>0}_{\F})^{\otimes 1-p}\right)$
intersects the maximal $\nu$-weight subspace $\underline{\D}(M)$. This will follow if we check
that every $\nu$-graded submodule of $U^0(\g^{>0}_\F)$ has nonzero eigenspace
with eigenvalue $(p-1)\sum_{\alpha| \langle\alpha,\nu\rangle>0}\alpha$.
This claim, in its turn, follows from the analogous claim for $S(\g^{>0}_\F)/(\g^{(1),>0}_\F)$,
where it is obvious.

It follows that (\ref{eq:delta_nabla_homom})
is an injection. Since the dimensions of the source and the target
are the same, we see  that (\ref{eq:delta_nabla_homom}) is an isomorphism.
It follows that
$$\,^\varsigma\!R\Hom_{\U_\F}(\underline{\Delta}_\nu(M),\U_\F)\cong
\underline{\nabla}_{\nu}(\,^\varsigma\!R\Hom_{\underline{\U}_\F}(M,\underline{\U}_\F)).$$
This implies the claim of the proposition.
\end{proof}

\begin{proof}[Proof of Proposition \ref{Prop:K_0_identity} in the general case]
Thanks to the distinguished case of this proposition that we have established
already in Section \ref{SS_K_0_distinguished}, we need to prove that if $[\underline{\D} M]=[M]$, then
$[\D \underline{\Delta}_\nu(M)]=[\underline{\Delta}_\nu(M)]$. Using
Proposition \ref{Prop:D_baby}, we reduce to proving
$[\underline{\Delta}_\nu(M)]=[\underline{\nabla}_\nu(M)]$.
We will prove this for an arbitrary module
$M\in \underline{\U}_\F^\chi\operatorname{-mod}^{\underline{Q}}$.

Note that every module in $\U^\chi_\F\operatorname{-mod}^{\underline{Q}}$
is also a module in $\underline{\U}^\chi_\F\operatorname{-mod}^{\underline{Q}}$
by restriction. The corresponding map between the $K_0$-groups
is injective because the images of simples are linearly independent
thanks to the upper triangularity. To finish the proof observe
that in $K_0(\U^\chi(\ug_\F)\operatorname{-mod}^{\underline{Q}})$ we have
$$[\underline{\Delta}_\nu(M)]=[U^0(\g^{<0}_\F)\otimes M]=[\underline{\nabla}_\nu(M)].$$
\end{proof}

\subsection{Duality and localization}\label{SS_loc_dual}
The goal of this section is to describe the autoequivalence
$$D^b(\Coh^{\underline{Q}}(\B_{\chi}))\xrightarrow{\sim}
D^b(\Coh^{\underline{Q}}(\B_{\chi}))^{opp},$$
induced by $\D: \U^\chi_{(0),\F}\operatorname{-mod}^{\underline{Q}}\xrightarrow{\sim}
\U^\chi_{(0),\F}\operatorname{-mod}^{{\underline{Q}},opp}$ under the derived
localization equivalence
$R\Gamma(\mathcal{V}^\chi_0(\rho)\otimes\bullet)$. Here we assume that
$\rho\in \mathfrak{X}(T)$ (we will explain what to do in the general case in the
end of the section). The bundle $\mathcal{V}^\chi_0$ is equipped
with a $\underline{Q}_\F$-equivariant structure as explained in
Section  \ref{SS_splitting_equivar} and $\mathcal{O}_\F(\rho)$ has a
$\underline{Q}_\F$-equivariant structure obtained by restriction of
the $G_\F$-equivariant structure. So
$\mathcal{V}^\chi_0(\rho)$ becomes a $\underline{Q}_\F$-equivariant
vector bundle.

Consider the Serre duality functor $R\Hom(\bullet, \Omega_{(\tilde{\g}_{\F}^{(1)})_\h})[\dim \g]$,
where  $\Omega_\bullet$ stands for the canonical bundle of $(\tilde{\g}_{\F}^{(1)})_\h$, this
bundle is trivial.
The Serre duality functor gives rise to an equivalence
$$D^b(\Coh^{\underline{Q}}(\B_{\chi}))\xrightarrow{\sim}
D^b(\Coh^{\underline{Q}}(\B_{\chi}))^{opp},$$
where we view $\B_\chi$ as a derived subscheme of $\tilde{\g}_\F^{(1)}$
(equivalently, of $(\tilde{\g}_\F^{(1)})_\h$).

Note that $\varsigma$ gives rise to an automorphism of
the $\h_\F^{(1)*}$-scheme $\tilde{\g}_\F^{(1)}$ such that
$\B_\chi$ is stable (as a derived subscheme).
We  twist the Serre duality functor by $\varsigma$ and then
tensor with the ${\underline{Q}}/{\underline{Q}}^\circ$-module $\Lambda^{top}(\ug_\F/\underline{\mathfrak{p}}_\F)$.
The resulting contravariant autoequivalence of $D^b(\Coh^{\underline{Q}}(\B_{\chi}))$
will be denoted by $\D_{coh}$. The following proposition is similar
to results of \cite[Section 3]{BMR_sing}.

\begin{Prop}\label{Prop:dual_localization}
Recall that we assume that $\rho\in \mathfrak{X}(T)$.
Then we have the following commutative diagram

\begin{picture}(100,30)
\put(2,20){$D^b(\Coh^{\underline{Q}}(\B_\chi))$}
\put(60,20){$D^b(\Coh^{\underline{Q}}(\B_\chi))^{opp}$}
\put(2,2){$D^b(\U^\chi_{(0),\F}\operatorname{-mod}^{\underline{Q}})$}
\put(60,2){$D^b(\U^\chi_{(0),\F}\operatorname{-mod}^{\underline{Q}})^{opp}$}
\put(15,19){\vector(0,-1){13}}
\put(75,19){\vector(0,-1){13}}
\put(33,3){\vector(1,0){26}}
\put(29,21){\vector(1,0){30}}
\put(44,4){\tiny $\D$}
\put(40,23){\tiny $T^{-1}_{w_0}\circ\D_{coh}$}
\put(16,12){\tiny $R\Gamma(\mathcal{V}^\chi_0(\rho)\otimes\bullet)$}
\put(76,12){\tiny $R\Gamma(\mathcal{V}^\chi_0(\rho)\otimes\bullet)$}
\end{picture}
\end{Prop}
\begin{proof}We note that the claim of this proposition reduces
to the case when $G$ is semisimple and simply connected. We are going
to assume this until the end of the proof.

{\it Step 1}. We write $\tilde{D}$ for $\tilde{D}_{\B_\F}$ to simplify the notation.
We have $R\Gamma(\tilde{D})=\U_{\h,\F}$.
The functors $R\Hom_{\U_\F}(\bullet, \U_\F)$ and $R\Hom_{\U_{\h,\F}}(\bullet, \U_{\h,\F})$
are isomorphic on $\U^\chi_{(0),\F}\operatorname{-mod}^{\underline{Q}}$
because the Artin-Schreier map $\h^*_\F\twoheadrightarrow \h^{*(1)}_{\F}$
is unramified. So we have the following commutative diagram

\begin{picture}(100,30)
\put(2,20){$D^b(\Coh^{\underline{Q}}(\tilde{D}|_{\B_\chi}))$}
\put(60,20){$D^b(\Coh^{\underline{Q}}(\tilde{D}^{opp}|_{\B_\chi}))^{opp}$}
\put(2,2){$D^b(\U^\chi_{(0),\F}\operatorname{-mod}^{\underline{Q}})$}
\put(60,2){$D^b(\U^{\chi,opp}_{(0),\F}\operatorname{-mod}^{\underline{Q}})^{opp}$}
\put(15,19){\vector(0,-1){13}}
\put(75,19){\vector(0,-1){13}}
\put(33,3){\vector(1,0){26}}
\put(34,22){\vector(1,0){24}}
\put(38,4){\tiny $R\Hom(\bullet, \U_\F)$}
\put(38,23){\tiny $R\Hom(\bullet, \tilde{D})$}
\put(16,12){\tiny $R\Gamma$}
\put(76,12){\tiny $R\Gamma$}
\end{picture}

Also we have the following commutative diagram, where in the top horizontal row
we take Hom over $\mathcal{O}_{\tilde{\g}_\F^{(1)}}$ and in the bottom arrow
we take Hom over $\tilde{D}_\F$.

\begin{equation}\label{eq:some_loc_diagr}
\begin{picture}(100,30)
\put(2,2){$D^b(\Coh^{\underline{Q}}(\tilde{D}|_{\B_\chi}))$}
\put(60,2){$D^b(\Coh^{\underline{Q}}(\tilde{D}^{opp}|_{\B_\chi}))^{opp}$}
\put(2,20){$D^b(\Coh^{\underline{Q}}(\B_{\chi}))$}
\put(60,20){$D^b(\Coh^{\underline{Q}}(\B_{\chi}))^{opp}$}
\put(15,19){\vector(0,-1){13}}
\put(75,19){\vector(0,-1){13}}
\put(35,3){\vector(1,0){24}}
\put(33,21){\vector(1,0){26}}
\put(38,23){\tiny $R\Hom(\bullet, \mathcal{O})$}
\put(38,4){\tiny $R\Hom(\bullet, \tilde{D})$}
\put(16,12){\tiny $\mathcal{V}^\chi_0(\rho)\otimes\bullet$}
\put(76,12){\tiny $(\mathcal{V}^\chi_0(\rho))^*\otimes \bullet$}
\end{picture}
\end{equation}

Combining the previous two diagram, twisting with $\varsigma$ and tensoring with
$\Lambda^{top}(\ug_\F/\underline{\p}_\F)$,  we get the following commutative diagram.

\begin{picture}(100,30)
\put(2,2){$D^b(\U^\chi_{(0),\F}\operatorname{-mod}^{\underline{Q}})$}
\put(60,2){$D^b(\U^{\chi}_{(0),\F}\operatorname{-mod}^{\underline{Q}})^{opp}$}
\put(2,20){$D^b(\Coh^{\underline{Q}}(\B_{\chi}))$}
\put(60,20){$D^b(\Coh^{\underline{Q}}(\B_{\chi}))^{opp}$}
\put(15,19){\vector(0,-1){13}}
\put(75,19){\vector(0,-1){13}}
\put(35,3){\vector(1,0){24}}
\put(33,21){\vector(1,0){26}}
\put(45,23){\tiny $\D_{coh}$}
\put(45,4){\tiny $\D$}
\put(16,12){\tiny $R\Gamma(\mathcal{V}^\chi_0(\rho)\otimes\bullet)$}
\put(76,12){\tiny $R\Gamma(\,^\varsigma\!(\mathcal{V}^\chi_0(\rho))^*\otimes \bullet)$}
\end{picture}

{\it Step 2}. The following commutative diagram is a consequence of  Proposition \ref{Prop:braid_equivariance}.

\begin{picture}(100,30)
\put(2,2){$D^b(\U^\chi_{(0),\F}\operatorname{-mod}^{\underline{Q}})$}
\put(60,2){$D^b(\U^{\chi}_{(0),\F}\operatorname{-mod}^{\underline{Q}})$}
\put(2,20){$D^b(\Coh^{\underline{Q}}(\B_{\chi}))$}
\put(60,20){$D^b(\Coh^{\underline{Q}}(\B_{\chi}))$}
\put(15,19){\vector(0,-1){13}}
\put(75,19){\vector(0,-1){13}}
\put(33,3){\vector(1,0){26}}
\put(31,21){\vector(1,0){28}}
\put(16,12){\tiny $R\Gamma(\mathcal{V}^\chi_0(\rho)\otimes\bullet)$}
\put(76,12){\tiny $R\Gamma(\mathcal{V}^\chi_0(\rho)\otimes\bullet)$}
\put(43,23){\tiny $T_{w_0}$}
\put(43,4){\tiny $T_{w_0}$}
\end{picture}

On the other hand, by \cite[Theorem 2.1.4]{BMR_sing}, we have

\begin{picture}(100,30)
\put(2,2){$D^b(\U^\chi_{(0),\F}\operatorname{-mod}^{\underline{Q}})$}
\put(60,2){$D^b(\U^{\chi}_{(0),\F}\operatorname{-mod}^{\underline{Q}})$}
\put(2,20){$D^b(\Coh^{\underline{Q}}(\tilde{D}^{-2\rho}|_{\B_\chi}))$}
\put(60,20){$D^b(\Coh^{\underline{Q}}(\tilde{D}|_{\B_\chi}))$}
\put(15,19){\vector(0,-1){13}}
\put(75,19){\vector(0,-1){13}}
\put(35,3){\vector(1,0){24}}
\put(38,21){\vector(1,0){21}}
\put(76,12){\tiny $R\Gamma$}
\put(16,12){\tiny $R\Gamma$}
\put(41,23){\tiny $\mathcal{O}(2\rho)\otimes\bullet$}
\put(43,4){\tiny $T_{w_0}$}
\end{picture}

This commutative diagram implies the following one.

\begin{picture}(100,30)
\put(2,2){$D^b(\U^\chi_{(0),\F}\operatorname{-mod}^{\underline{Q}})$}
\put(60,2){$D^b(\U^{\chi}_{(0),\F}\operatorname{-mod}^{\underline{Q}})$}
\put(2,20){$D^b(\Coh^{\underline{Q}}(\B_{\chi}))$}
\put(60,20){$D^b(\Coh^{\underline{Q}}(\B_{\chi}))$}
\put(15,19){\vector(0,-1){13}}
\put(75,19){\vector(0,-1){13}}
\put(33,3){\vector(1,0){26}}
\put(31,21){\vector(1,0){28}}
\put(76,12){\tiny $R\Gamma(\mathcal{V}^\chi_0(\rho)\otimes\bullet)$}
\put(16,12){\tiny $R\Gamma(\mathcal{V}^\chi_{-2\rho}(\rho)\otimes\bullet)$}
\put(43,23){\tiny $\operatorname{id}$}
\put(43,4){\tiny $T_{w_0}$}
\end{picture}

So we conclude that
\begin{equation}\label{eq:split_intertw}
R\Gamma(\mathcal{V}^\chi_{-2\rho}(\rho)\otimes \bullet)\cong R\Gamma(\mathcal{V}^\chi_{0}(\rho)\otimes
T_{w_0}^{-1}(\bullet)).
\end{equation}

{\it Step 3}. Suppose, for a moment, that we know that $\mathcal{V}^{\chi}_{-2\rho}(\rho)$
and $\,^\varsigma\! (\mathcal{V}^\chi_0(\rho))^*$ are ${\underline{Q}}_\F$-equivariantly isomorphic.
Using (\ref{eq:some_loc_diagr}), we get
$$\D\circ R\Gamma(\mathcal{V}^\chi_0(\rho)\otimes\bullet)\cong
R\Gamma(\mathcal{V}^\chi_{-2\rho}(\rho)\otimes \D_{coh}(\bullet)).$$
By (\ref{eq:split_intertw}), we get
$$R\Gamma(\mathcal{V}^\chi_{-2\rho}(\rho)\otimes \D_{coh}(\bullet))\cong R\Gamma(\mathcal{V}^\chi_{0}(\rho)\otimes T_{w_0}^{-1}\D_{coh}(\bullet)).$$
The last two isomorphisms imply the commutative diagram in the statement
of the proposition.

{\it Step 4}.  It remains to show that we have a ${\underline{Q}}_\F$-equivariant isomorphism $$\mathcal{V}^{\chi}_{-2\rho}(\rho)\cong \,^\varsigma\! (\mathcal{V}^\chi_0(\rho))^*.$$
 We will prove a stronger statement:  there is a ${\underline{Q}}_\F\times \F^\times$-equivariant and
 an $\F[z]$-linear (where $z$ is a coordinate on $\F\chi$) isomorphism
\begin{equation}\label{eq:splitting_iso_chi} \mathcal{V}^{\F\chi}_{-2\rho}(\rho)\cong \,^\varsigma\! (\mathcal{V}^{\F\chi}_0(\rho))^*,\end{equation}
where the bundles $\mathcal{V}^{\F\chi}_{?}$ were introduced in Section
\ref{SS_splitting_equivar}.
Note that the action of $\F^\times$ is contracting (to $(\B_\chi,0)$).
Therefore it is enough to prove that there is a $G_\F$-equivariant
isomorphism
\begin{equation}\label{eq:splitting_iso}\mathcal{V}^{0}_{-2\rho}(\rho)\cong \,^\varsigma\! (\mathcal{V}^{0}_0(\rho))^*.
\end{equation}
Indeed, we restrict (\ref{eq:splitting_iso}) to $\tilde{\g}_\F^{(1)}\times_{\g_\F^{*(1)}}\g_\F^{*(1)\wedge_\chi}$
and get (\ref{eq:splitting_iso_chi}) thanks to the $\F^\times$-equivariance (recall that, by the construction in Section \ref{SS_splitting_equivar} the $\F^\times$-action is also restricted
from $G_\F$).

The left hand side of (\ref{eq:splitting_iso}) is a splitting bundle
for $\tilde{D}^{-2\rho}_{\B,\F}$ restricted to
$$\tilde{\g}_\F^{(1)\wedge_0}=\tilde{\g}_\F^{(1)}\times_{\g_\F^{*(1)}}\g_{\F}^{*(1)\wedge_0}.$$
The right hand side  is a splitting bundle for $\,^\varsigma\!(\tilde{D}_{\B,\F})^{opp}$
restricted to $\tilde{\g}_\F^{(1)\wedge_0}$. We claim that the restrictions
of $\,^\varsigma\!(\tilde{D}_{\B,\F})^{opp}$ and $\tilde{D}^{-2\rho}_{\B,\F}$
to  $\tilde{\g}_\F^{(1)\wedge_0}$
are $G_\F$-equivariantly isomorphic Azumaya algebras.

First of all, note that $D_{\B,\F}$ is an $\operatorname{Ad}(\g_\F)$-equivariant
sheaf of algebras on $\B_\F$. The standard
isomorphism $D^{opp}_{\B,\F}\cong D^{-2\rho}_{\B,\F}$ gives the filberwise
multiplication by $-1$ on $T^*\B_\F^{(1)}$. So we get an isomorphism
$\,^{\sigma'}D^{opp}_{\B,\F}\cong D^{-2\rho}_{\B,\F}$ of Azumaya algebras,
where $\sigma'$ is the standard antiautomorphism of $\g_\F$ given by
$x\mapsto -x$. This isomorphism extends to an $\F[\h^*]$-semilinear
isomorphism $\,^{\sigma'}\tilde{D}^{opp}_{\B,\F}\cong \tilde{D}^{-2\rho}_{\B,\F}$,
where on $\F[\h^*]$ we get an automorphism given by $x\mapsto -x$.
And when we twist with $\varsigma$ instead we get an
$\F[\h^*]$-linear and $G_\F$-equivariant isomorphism $\,^{\varsigma}\tilde{D}^{opp}_{\B,\F}\cong \tilde{D}^{-2\rho}_{\B,\F}$.


So (\ref{eq:splitting_iso}) is true up to a twist with a line bundle.
We need to show that this line bundle is trivial and it is sufficient
to prove the restriction of the line bundle to $\B_\F^{(1)}$ is trivial.
Similarly to the proof of (1) of Lemma \ref{Lem:split_properties},
it is enough to show (\ref{eq:splitting_iso}) on the level of
$K_0$-classes in the non-equivariant K-theory.
Thanks to (1) of Lemma \ref{Lem:split_properties}
this equality reduces to
\begin{equation}\label{eq:splitting_K0_class}[(\Fr_{\B*}\Str)(\rho)]^*=[(\Fr_{\B*}\Str(-2\rho))(\rho)].\end{equation}
Recall that we have an isomorphism
$\operatorname{Fr}_*\Str(-\rho)\cong \Str^{(1)}(-\rho)^{\oplus p^{\dim G/B}}$.
So
$$[(\Fr_{\B*}\Str)(\rho)]^*=p^{\dim G/B}[ \Str(\rho/p)]^*,
[(\Fr_{\B*}\Str(-2\rho))(\rho)]=p^{\dim G/B}[\Str(-\rho/p)].$$
 (\ref{eq:splitting_K0_class}) follows.
\end{proof}

\begin{Rem}\label{Rem:dual_localization}
Now we no longer assume that $\rho\in \mathfrak{X}(T)$ (compare to Remark \ref{Rem:simply_conn}). Instead, let $\rho'$
be a character of $T$ that pairs by $1$ with all simple coroots. Then the argument
of the proof of Proposition \ref{Prop:dual_localization} shows that
the following diagram is commutative:

\begin{picture}(100,30)
\put(2,20){$D^b(\Coh^{\underline{Q}}(\B_\chi))$}
\put(60,20){$D^b(\Coh^{\underline{Q}}(\B_\chi))^{opp}$}
\put(2,2){$D^b(\U^\chi_{(0),\F}\operatorname{-mod}^{\underline{Q}})$}
\put(60,2){$D^b(\U^\chi_{(0),\F}\operatorname{-mod}^{\underline{Q}})^{opp}$}
\put(15,19){\vector(0,-1){13}}
\put(75,19){\vector(0,-1){13}}
\put(33,3){\vector(1,0){26}}
\put(29,21){\vector(1,0){30}}
\put(44,4){\tiny $\D$}
\put(40,23){\tiny $T^{-1}_{w_0}\circ\D_{coh}$}
\put(16,12){\tiny $R\Gamma(\mathcal{V}^\chi_0(\rho')\otimes\bullet)$}
\put(76,12){\tiny $R\Gamma(\mathcal{V}^\chi_0(2\rho-\rho')\otimes\bullet)$}
\end{picture}

Note that $\mathcal{V}^\chi_0(\rho'), \mathcal{V}^\chi_0(2\rho-\rho')$
differ by a twist with a character of $T$.
\end{Rem}

\section{$K_0$-classes of equivariantly simple $\U^\chi_{(0),\F}$-modules}\label{S_K_0_classes}
\subsection{Notation and content}
Assume that $\lambda+\rho$ is regular and pick a representative $\mu^\circ$
of $W\cdot \lambda$ in the anti-dominant $p$-alcove.
The meaning of $G, \underline{G}, L,P, \nu,\chi, W^a, W_P, W^{a,P},\mu_x,\alpha_0,\ldots,\alpha_r$
is as in Section \ref{SS_notation_basic} and the meaning of $\tilde{\Nilp},\tilde{\g}$
is as in Section \ref{SS_loc_notation}. We assume that $G$ is semisimple and simply connected.  We write $\underline{\rho},\rho_L$ for the elements
$\rho$ for the Levi subalgebras $\underline{\g},\lf$.

In Section \ref{SS_affine_Hecke_reminder} and some subsequent sections  we will introduce some additional notation related to (affine) Hecke algebras.

The goal of this section is to finish the proofs of Theorems \ref{Thm:disting_dim}
and \ref{Thm_dim_general} (see Theorems \ref{Thm:distinguished_p_dim} and
\ref{Thm:K0_class_general} below).   Section \ref{SS_splitting_induction}
contains two technical results that describe an interplay between
the parabolic induction and the derived localization equivalences.
Then in Section \ref{SS_char_distinguish} we prove Theorem
\ref{Thm:distinguished_p_dim}, a stronger version of
Theorem \ref{Thm:disting_dim}. Next, in Section \ref{SS_canonical}
we start explaining our approach to proving Theorem
\ref{Thm_dim_general}: it is based on the study of the graded lift
of the contravariant duality functor $\D$ from Section
\ref{S_duality}. In this respect it is similar to what was done
in \cite{BM}, and, in fact, a Koszulity result, Theorem
\ref{Thm:Koszulity}, from that paper is a crucial part of our approach. Unlike in \cite{BM}, we end up with explicit character formulas, which requires a substantial additional work.
Sections \ref{SS_Hecke_module}-\ref{SS_final} contain a proof of
Theorem \ref{Thm_dim_general}. In Section \ref{SS_irred} we discuss a relation
between the equivariantly irreducible modules and usual irreducible modules.
And then in Section
\ref{SS_categorif} we speculate on a categorification of
Theorem \ref{Thm:K0_class_general}.

\subsection{Parabolic induction and splitting bundles}\label{SS_splitting_induction}
The goal of this section is to establish  two results on an interplay of splitting
bundles with different instances of parabolic induction functors.
Proposition \ref{Prop:O_B_global_section} plays an important role in  determining the $K_0$-classes of $A$-equivariantly irreducible
modules in the case when $\chi$ is distinguished. Proposition
\ref{Prop:baby_Verma_local} is
in the case of general $\chi$, it will be used to extend the computation of
irreducible $K_0$-classes from the distinguished case to the general one.

Recall, Section \ref{SS_parab_statements}, that for a distinguished $\chi$, we have an irreducible
component $\B_{\mathfrak{m}}\subset \B_\chi$ that is naturally identified with
$P_\F^{(1)}/B_{\F}^{(1)}$.

\begin{Prop}\label{Prop:O_B_global_section}
We have $R\Gamma(\mathcal{V}^\chi_0|_{\B_{\mathfrak{m}}})=W^\chi_\F(2\rho_L-2\rho)$.
\end{Prop}

Our second result  gives a geometric interpretation of the baby Verma functor
$$\underline{\Delta}_\nu: \underline{\U}^\chi_{(-2\rho),\F}\operatorname{-mod}^{\underline{Q}}
\rightarrow \U^\chi_{(-2\rho),\F}\operatorname{-mod}^{\underline{Q}}.$$
We have a natural embedding $\xi:\tilde{\underline{\g}}_\F^{(1)}
\hookrightarrow \tilde{\g}^{(1)}_\F$ that sends a pair
$(x,\underline{\mathfrak{b}}_\F)\subset \mathfrak{g}_\F^{(1)}$ to
$(x,\underline{\mathfrak{b}}_\F\oplus \g^{>0}_\F)$. This map gives rise to the corresponding
embedding of (derived) Springer fibers at $\chi$ to be also denoted by
$\xi$. For $\mu\in \mathfrak{X}(T)$, we write $\underline{\mathcal{V}}^\chi_\mu$
for the splitting bundle of $\tilde{D}^\mu_{\underline{\B}_\F}$ defined
by the formula analogous to (\ref{eq:splitting_Springer}), where we use
$\underline{\rho}$ instead of $\rho$.

\begin{Prop}\label{Prop:baby_Verma_local}
We have the following commutative diagram.

\begin{picture}(100,30)
\put(2,2){$D^b(\underline{\U}^\chi_{(-2\rho),\F}\operatorname{-mod}^{\underline{Q}})$}
\put(72,2){$D^b(\U^\chi_{(-2\rho),\F}\operatorname{-mod}^{\underline{Q}})$}
\put(2,22){$D^b(\Coh^{\underline{Q}}(\underline{\B}_{\chi}))$}
\put(72,22){$D^b(\Coh^{\underline{Q}}(\B_{\chi}))$}
\put(15,21){\vector(0,-1){14}}
\put(84,21){\vector(0,-1){14}}
\put(40,3){\vector(1,0){31}}
\put(32,23){\vector(1,0){39}}
\put(55,5){\tiny $\underline{\Delta}_\nu$}
\put(55,25){\tiny $\xi_*$}
\put(16,13){\tiny $R\Gamma(\underline{\mathcal{V}}^\chi_{2\underline{\rho}-2\rho}(\rho)\otimes\bullet)$}
\put(85,13){\tiny $R\Gamma(\mathcal{V}^\chi_{0}(\rho)\otimes\bullet)$}
\end{picture}
\end{Prop}

Let us explain a general construction that goes into the proofs of these
two propositions. Let $P'$ be a parabolic subgroup of $G$ containing $B$
and let $L'$ denote the standard Levi subgroup of $P'$. Below we take $P'=P$ (for
Proposition \ref{Prop:O_B_global_section}) and $P'=G^{\geqslant 0}$
(for Proposition \ref{Prop:baby_Verma_local}). Let $M'$ denote the unipotent
radical of $P'$ and $\mathfrak{m}'$ denote the Lie algebra of
$\mathfrak{m}'$. This, in particular, gives
rise to the Verma functor $\Delta^{P'}_\F: U(\lf'_\F)\operatorname{-mod}
\rightarrow \U_\F\operatorname{-mod}$.

We consider the fiberwise lagrangian subvariety $\tilde{Y}^{(1)}_\F\subset \tilde{\g}_{\F}^{(1)}\times\tilde{\lf}'^{(1)}_\F$, given by
$\{(\mathfrak{b}^{(1)}_{\F},x)| \mathfrak{b}^{(1)}_{\F}\supset\mathfrak{m}'^{(1)}_{\F}\}$
under the natural embedding.

The embedding $ \B'_\F:=P'_\F/B_\F\hookrightarrow \B_\F$, to be denoted by $\xi_0$,
gives rise to the D-module pushforward functor
$\xi_{0,*}: \Coh(\tilde{D}_{\B',\F})\rightarrow \Coh(\tilde{D}_{\B,\F})$.
This functor can be viewed as tensoring with the $\tilde{D}_{\B,\F}$-$\tilde{D}_{\B',\F}$-bimodule
$\xi_{0,*}(\tilde{D}_{\B',\F})$. Recall that this bimodule is defined as follows.
Note that $\mathcal{O}_{\B',\F}\otimes_{\mathcal{O}_{\B,\F}}\tilde{D}_{\B,\F}$
is a $\tilde{D}_{\B',\F}$-$\tilde{D}_{\B,\F}$-bimodule (note that $\tilde{D}_{\B',\F}$
is a flat sheaf of $\F[\h^*]$-algebras). We can identify $(\tilde{D}_{\B,\F})^{opp}$
with $\,^{\sigma'}\!(\tilde{D}^{-2\rho}_{\B,\F})$, where $\sigma'$ is the standard antiinvolution
for $\g_\F, x\mapsto -x$, and, similarly,  identify $(\tilde{D}_{\B',\F})^{opp}$
with $\,^{\sigma'}\!(\tilde{D}^{-2\rho}_{\B',\F})$, these are $\F[\h^*]$-semilinear
identifications with respect to $\sigma'$, see the proof of Proposition \ref{Prop:dual_localization}.
So we set
\begin{equation}\label{eq:pushforward}\xi_{0,*}(\tilde{D}_{\B',\F}):=\Omega_{\B'_\F}\otimes_{\Str_{\B_\F}}\tilde{D}_{\B,\F}\otimes_{\Str_{\B_\F}}
\Omega_{\B_\F}^{-1},\end{equation}
this is a $\tilde{D}_{\B,\F}$-$\tilde{D}_{\B',\F}$-bimodule.

We can view $\xi_{0,*}(\tilde{D}_{\B',\F})$
as a vector bundle on $\tilde{Y}^{(1)}_\F$. This is a splitting bundle
for the Azumaya algebra $\left(\tilde{D}_{\B,\F}\otimes \tilde{D}_{\B',\F}^{opp}\right)|_{\tilde{Y}^{(1)}_\F}$.

Let $\chi\in \g_{\F}^{(1)*}$ be a nilpotent element
vanishing on $\mathfrak{m}'^{(1)}_\F$ and let $\underline{\chi}$
be the induced element in $\lf'^{(1)*}_\F$.
Let $(\g_\F^{(1)*}\times \lf'^{(1)*}_\F)^{\wedge_{\chi,\underline{\chi}}}$
denote the spectrum of the completion of $\F[\g_\F^{(1)*}\times \lf'^{(1)*}]$
at $(\chi,\underline{\chi})$. We set
$$\tilde{Y}^{(1)\wedge_{\chi,\underline{\chi}}}_\F:=\tilde{Y}^{(1)}_\F\times_{\g_\F^{(1)*}\times \lf'^{(1)*}_\F}
(\g_\F^{(1)*}\times \lf'^{(1)*}_\F)^{\wedge_{\chi,\underline{\chi}}}.$$

The following lemma is a crucial technical statement that goes into proof
of the two propositions above.

\begin{Lem}\label{Lem:splitting_equiv}
We write $Q'_\F$ for the reductive
part of the centralizer of $\chi$ in $P'_\F$. Then the restriction
of $\xi_{0,*}(\tilde{D}_{\B',\F})$ to  $\tilde{Y}^{(1)\wedge_{\chi,\underline{\chi}}}_\F$
is $Q'_\F$-equivariantly isomorphic
to the restriction of
$$\mathcal{V}^\chi_0\otimes (\underline{\mathcal{V}}^{\underline{\chi}}_0)^*.$$
\end{Lem}
\begin{proof}
Similarly to the proof of Step 4 of Proposition \ref{Prop:dual_localization}
the isomorphism we need to prove will follow once we replace $(\chi,\underline{\chi})$
with the line $\F(\chi,\underline{\chi})$. The latter isomorphism, in its turn,
will follow once we know that there is an $L'_{\F}$-equivariant isomorphism
between
\begin{itemize}
\item[(i)] the restriction of $\xi_{0,*}(\tilde{D}_{\B',\F})$ to  $\tilde{Y}^{(1)\wedge_{0,0}}_\F$,
\item[(ii)] and the restriction of
$\mathcal{V}^0_0\otimes (\underline{\mathcal{V}}^{0}_0)^*$.
\end{itemize}

The  bundles in (i) and (ii) are still splitting bundles
for the same Azumaya algebra, the restriction of $\tilde{D}_{\B,\F}\otimes
\tilde{D}_{\B',\F}^{opp}$ to $\tilde{Y}_\F^{(1)\wedge_{0,0}}$. So (i) and (ii) differ by a twist with an $L'_\F$-equivariant line bundle.
Such a line bundle
is given by a character of the group scheme $B_\F P'^{(1)}_\F$. The character
lattice of this group scheme embeds into the character lattice of $B_\F$.
So the line bundle of interest is given by a suitable character of $B_\F$.
Since $\mathfrak{X}(B_\F)$ is torsion free, to prove that this character is
trivial, we need to verify that the $K_0$-classes of the top exterior powers of
the splitting bundles (i) and (ii) are the same.

Let us start with computing the class for the restriction of $\xi_{0*}(\tilde{D}_{\B',\F})$.
It follows from (\ref{eq:pushforward}) that $$[\xi_{0*}(\tilde{D}_{\B',\F})]=[\operatorname{Fr}_{\tilde{Y},*}\Omega_{\tilde{Y}_\F}].$$
Note that the class of $\Omega_{\tilde{Y}_\F}$ in $K_0^{L'_\F}(\tilde{Y}_\F)=\mathfrak{X}(T_\F)$
is $2(\rho_{L'}-\rho)$. Hence $[\mathcal{O}(\rho_{L'}-\rho)]$ is self-dual under the Serre
duality. Since $\operatorname{Fr}_*$ intertwines the Serre duality functors,
the class of $[\operatorname{Fr}_{\tilde{Y},*}][\mathcal{O}(\rho_{L'}-\rho)]$
is self-dual. Therefore its top exterior power is $[\mathcal{O}(\rho_{L'}-\rho)]$.
Arguing as in the proof of (1) of Lemma \ref{Lem:split_properties},
we see that
\begin{equation}\label{eq:top_power1}
[\Lambda^{top}\operatorname{Fr}_{\tilde{Y},*}\Omega_{\tilde{Y}_\F}]=p^{\dim \tilde{Y}-1}(p-1)(\rho_{L'}-\rho).
\end{equation}
Now we compute the class for the restriction of
$\mathcal{V}^0_0\otimes (\underline{\mathcal{V}}^{0}_{0})^*$.
It coincides with  that of the restriction to $\B'_\F$. Part (1) of Lemma
\ref{Lem:split_properties} implies that this restriction
is
\begin{equation}\label{eq:bundle_restriction}\operatorname{Fr}_{\B,\F}(\mathcal{O}_{\B_\F})\otimes \left( \Fr_{\B',*}(\mathcal{O}_{\B'_\F})\right)^*.
\end{equation}
The canonical bundle of $\B$ is $\mathcal{O}(-2\rho)$ hence the
top exterior power of $\operatorname{Fr}_{\B,\F}(\mathcal{O}_{\B_\F})$
is $-p^{\dim \B-1}(p-1)\rho$.  Similarly,
the equivariant canonical bundle of $\B'$ is $\mathcal{O}(-2\rho_{L'})$,
and the top exterior power of $\Fr_{\B',*}(\mathcal{O}_{\B'_\F})$
$-p^{\dim \B'-1}(p-1)\rho_{L'}$.
Therefore  the top  exterior power of
(\ref{eq:bundle_restriction}) coincides with (\ref{eq:top_power1}).
This finishes the proof.
\end{proof}

We will also need the following standard lemma.

\begin{Lem}\label{Lem:pushforward_global_section}
We have $\Gamma(\xi_{0*}(\tilde{D}_{\B',\F}))=\U_{\h,\F}/ \U_{\h,\F}\mathfrak{m}'_\F\otimes_{\F}
\F_{2\rho-2\rho_{L'}}$, where we write $\F_{2\rho-2\rho_{L'}}$ for the one-dimensional
$T_\F$-module with character $2\rho-2\rho_{L'}$.
\end{Lem}
\begin{proof}
Note that $\Omega_{\B'_\F}\otimes_{} \xi_0^*(\Omega^{-1}_{\B_\F})$ is a trivial
$\Str_{\B'_\F}$-bundle (with a nontrivial $T_\F$-action). Consider the global
section $1$ of this bundle. This is an element of the $L'_{\F}$-equivariant
$\U_{\h,\F}$-$U_{\h}(\lf'_\F)$-bimodule $\Gamma(\xi_{0*}(\tilde{D}_{\B',\F}))$
with weight $2\rho-2\rho_{L'}$.
Note that $\mathfrak{m}'_{\F}$ annihilates $1$. We get
an $L'_{\F}$-equivariant $\U_{\h,\F}$-linear homomorphism
\begin{equation}\label{eq:bimod_homom_gamma}
\U_{\h,\F}/ \U_{\h,\F}\mathfrak{m}'_\F\otimes_{\F}
\F_{2\rho-2\rho_{L'}}\rightarrow \Gamma(\xi_{0*}(\tilde{D}_{\B',\F})).
\end{equation}
An $L'_\F$-equivariant $\U_{\h,\F}$-linear map is automatically also $U_{\h}(\lf'_\F)$-linear.

Now we prove that the homomorphism we have constructed is an isomorphism.
We note that, as a right $\tilde{D}_{\B',\F}$-module, $\xi_{0*}(\tilde{D}_{\B',\F'})$
is freely generated by $U(\mathfrak{m}'^-_\F)\cong U(\mathfrak{m}'^-_\F)1
\subset \Gamma(\xi_{0*}(\tilde{D}_{\B',\F'}))$. To see that this is the case,
we reduce to the associated graded bimodule, where the claim is clear.
In particular,  $\Gamma(\xi_{0*}(\tilde{D}_{\B',\F'}))\xrightarrow{\sim}
U(\mathfrak{m}'^-_\F)\otimes_{\F}U_{\h}(\lf_\F)$. We can identify
$\U_{\h,\F}/\U_{\h,\F}\mathfrak{m}'_\F$ with $U(\mathfrak{m}'^-_\F)\otimes_{\F}U_{\h}(\lf_\F)$
and (\ref{eq:bimod_homom_gamma}) becomes the identity under these identifications.
\end{proof}

Now we proceed to proving the two propositions in the beginning of the section.
Proposition \ref{Prop:baby_Verma_local} is a straightforward corollary
of Lemmas \ref{Lem:splitting_equiv},\ref{Lem:pushforward_global_section}
(note that the twist by $\rho$ doesn't matter in the statement, but matters
for the applications below).

\begin{proof}[Proof of Proposition \ref{Prop:O_B_global_section}]
Thanks to Lemma \ref{Lem:splitting_equiv},
$\mathcal{V}^\chi_0|_{\B_{\mathfrak{m}}}$
is naturally identified with the fiber at $\chi$ of
$$\xi_{0*}(\tilde{D}_{P/B,\F})\otimes_{\tilde{D}_{P/B,\F}}
\underline{\mathcal{V}}^{0}_{0}|_{P_\F^{(1)}/B_\F^{(1)}}.$$
By the proof of Lemma \ref{Lem:pushforward_global_section},
$\xi_{0*}(\tilde{D}_{P/B,\F})$ is a free right
$\tilde{D}_{P/B,\F}$-module. It follows that
$$R\Gamma\left(\xi_{0*}(\tilde{D}_{P/B,\F})\otimes_{\tilde{D}_{P/B,\F}}
\underline{\mathcal{V}}^{0}_{0}|_{P_\F^{(1)}/B_\F^{(1)}}\right)
\xrightarrow{\sim}
\Gamma\left(\xi_{0*}(\tilde{D}_{P/B,\F})\right)\otimes_{\Gamma(\tilde{D}_{P/B,\F})}
R\Gamma\left(\underline{\mathcal{V}}^{0}_{0}|_{P_\F^{(1)}/B_\F^{(1)}}\right).$$
Now we use Lemma \ref{Lem:pushforward_global_section} to show that
$$R\Gamma(\mathcal{V}^\chi_0|_{\B_{\mathfrak{m}}})=
\underline{\Delta}^\chi(R\Gamma\left(\underline{\mathcal{V}}^{0}_{2\rho_L-2\rho}|_{P_\F^{(1)}/B_\F^{(1)}}
\right)).$$
Then we use (1) of Lemma \ref{Lem:split_properties} that says, in particular, that
$$\mathcal{V}^{0}_{2\rho_L-2\rho}|_{P_\F^{(1)}/B_\F^{(1)}}=
\operatorname{Fr}_{P/B,*}\mathcal{O}_{P/B}(2\rho_L-2\rho).$$
Of course, $R\Gamma(\operatorname{Fr}_{P/B,*}\mathcal{O}_{P/B}(2\rho_L-2\rho))$
is the one-dimensional representation of $L_\F$ with character
$2\rho_L-2\rho$. So we see that $R\Gamma(\mathcal{V}^\chi_0|_{\B_{\mathfrak{m}}})=
W^\chi_\F(2\rho_L-2\rho)$.
\end{proof}

\subsection{Reminder on affine Hecke algebras}\label{SS_affine_Hecke_reminder}
Consider the affine Hecke algebra $\Hecke^a_{G}$ for $G$ over $\Z[v^{\pm 1}]$,
where $v$ is an indeterminate.
For $x\in W^a$, let $H_x$ denote the standard basis element of $\Hecke^a_G$.
Recall that the product on $\Hecke^a_{G}$ is determined by
\begin{align*}
& H_{x}H_y=H_{xy}\text { if }\ell(xy)=\ell(x)+\ell(y),\\
& (H_s+v)(H_s-v^{-1})=0,
\end{align*}
where $s$ is runs over the simple affine reflections.

The Hecke algebra $\Hecke^a_{G}$ comes with a $\Z[v^{\pm 1}]$-linear ring involution,
called the bar-involution and denoted by $\bar{\bullet}$, it is given by
$\bar{H}_x:=H^{-1}_{x^{-1}}$. As Kazhdan and Lusztig checked in
\cite{KL}, there is a unique basis $C_x, x\in W^a,$ of $\Hecke^a_G$
with the following two properties:
\begin{itemize}
\item $\bar{C}_x=C_x$ for all $x\in W^a$,
\item and $C_x-H_x\in v^{-1}\operatorname{Span}_{\Z[v^{-1}]}(H_y| y\in W^a)$
for all $x\in W^a$.
\end{itemize}
Then $C_x=\sum_{y\preceq x}c_{xy}(v)H_y$, where $\prec$ stands for the Bruhat order
and $c_{xy}\in \Z[v^{-1}]$ is a Kazhdan-Lusztig polynomial.

We will need a parabolic version of this construction. Let $P$ be a parabolic subgroup
of $G$.
The Hecke algebra $\Hecke_{W_P}$
of the parabolic subgroup $W_P\subset W^a$ embeds into $\Hecke^a_G$. Consider the
sign representation $\operatorname{sgn}_P\cong \Z[v^{\pm 1}]$ of $\Hecke_{W_P}$, where $H_w$ acts via
$(-v)^{\ell(w)}$, and the induced module $\Hecke^{a,P}_G:=\Hecke^a_G\otimes_{\Hecke_{W_P}}
\operatorname{sgn}_P$. For $x\in W^{a,P}$, define $H^P_x$ as $H_x\otimes 1$.
We note that $\operatorname{sgn}_P$ embeds into $\Hecke_{W_P}$ via
$1\mapsto C_{w_{0,P}}=\sum_{w\in W_P}(-v)^{-\ell(ww_{0,P})}H_w$.
This gives rise to an embedding $\Hecke^{a,P}_G\hookrightarrow \Hecke^a_G$
so that $C_x$ for $x\in W^{a,P}$ lies in the image. So, for
$x\in W^{a,P}$ we can expand $C_x$ as $\sum_{y\in W^{a,P}}c^P_{xy}(v)H^P_y$.
The coefficient $c_{xy}^P$ is known as  a parabolic (affine) Kazhdan-Lusztig polynomial.
This construction also implies that if $x\in W^{a,P}$ lies in the two-sided cell
of $w_{0,P}$, then $x$ actually lies in the left cell of $w_{0,P}$.

Let us now recall the representation theoretic meaning of the values
$c_{xy}^P(1)$. For $x\in W^{a,P}$, let $\Delta^P_x,L^P_x$ denote the standard
and simple objects in $\Perv_{I^\circ}(\Fl_P)$. Then we have
a $W^a$-equivariant isomorphism $\C\otimes_{\Z}K_0(\Perv_{I^\circ}(\Fl_P))
\xrightarrow{\sim} \Hecke^{a,P}_G|_{v=1}$
that maps the specializations of $C_x,H^P_y$ to $[L^P_x],[\Delta^P_y]$,
respectively.
In particular, the classes in $K_0$
are related as follows:
\begin{equation}\label{eq:mult_perverse}
[L^P_x]=\sum_{y\in W^{a,P}}c^P_{xy}(1)[\Delta^P_y].
\end{equation}

Now recall that, for $\theta\in \mathfrak{X}(T)$, we have the corresponding
element $X_\theta\in \Hecke^a_G$. We have
\begin{equation}\label{eq:X_bar}
\bar{X}_\theta= H_{w_0}X_{w_0(\theta)}H_{w_0}^{-1}.
\end{equation}
Also we will need a coherent geometric realization of $\Hecke^a_G$ and the corresponding
formula for the bar-involution. Namely, consider the $\C^\times$-action on $\g$ via
$(t,x):=t^{-2}x$. This gives rise to the action of $\C^\times$ on $\St_\h$.
Consider the equivariant $K_0$-group $K_0^{G\times\C^\times}(\St_\h)$, which is
an algebra with respect to convolution. It is also a module over $K_0^{\C^\times}(\operatorname{pt})$
and hence a $\Z[v^{\pm 1}]$-algebra. By a theorem of Kazhdan-Lusztig and Ginzburg,
see \cite[Section 7]{CG} or \cite[Section 8]{Lusztig_K1}, we have a  $\Z[v^{\pm 1}]$-algebra isomorphism $\Hecke^a_G\cong K_0^{G\times\C^\times}(\St_\h)$ that sends $X_\theta$ to the class of the line bundle
$\mathcal{O}(\theta)$ on the diagonal.

Consider the $\varsigma$-twisted and $(-\dim \g)$-shifted Serre duality functor
$$\tilde{\D}_{coh}:=\,^\varsigma R\Hom_{\tilde{\g}\times \tilde{\g}}(\bullet, \Omega_{\tilde{\g}\times \tilde{\g}})[\dim \g]:
D^b(\Coh^{G\times \C^\times}(\St_\h))\xrightarrow{\sim}
D^b(\Coh^{G\times \C^\times}(\St_\h))^{opp},$$
where we consider $\Omega_{\tilde{\g}\times \tilde{\g}}$
with its natural $G\times \C^\times$-equivariant structure.
The following result is \cite[Proposition 9.12]{Lusztig_K1} (note that our $H_x$
is $\tilde{T}_x^{-1}$).

\begin{Lem}\label{Lem:bar_coherent}
Under the identification $\Hecke^a_G\cong K_0^{G\times\C^\times}(\St_\h)$,
we have
$$\bar{a}=v^{-2\ell(w_0)}H_{w_0}[\tilde{\D}_{coh}](a)H_{w_0}^{-1}.$$
\end{Lem}
Here and below we write $[\tilde{\D}_{coh}]$ for the operator on $K_0$
induced by $\tilde{\D}_{coh}$.

\subsection{Character formulas in the distinguished case}\label{SS_char_distinguish}
We assume that $\chi$ is distinguished. Let $w_{0,P}$ be the longest element
in $W_P$. Consider the left cell ${\mathfrak{c}}_P\subset W^a$ containing
$w_{0,P}$. It is contained in $W^{a,P}$.

Let $V_L(\mu_x)$ denote the irreducible representation of $L$ (over $\C$) with highest
weight $\mu_x$. Let   $\hat{d}_L(x)$ be its ${\underline{Q}}$-character. We write $\mathsf{ch}_{\mathfrak{m}^-}$
for the ${\underline{Q}}$-character of $U^0(\mathfrak{m}_\F^-)$.

The following is one of the main results of the paper. In particular, it implies
Theorem \ref{Thm:disting_dim}.

\begin{Thm}\label{Thm:distinguished_p_dim}
Let $\lambda+\rho$ be regular and $\chi$ be distinguished.  The following claims are true:
\begin{enumerate}
\item There is a bijection between $\mathfrak{c}_P$ and $\operatorname{Irr}(\U^\chi_{\lambda,\F}\operatorname{-mod}^{\underline{Q}})$.
    Let $L^\chi_{x,\F}$ denote the simple object corresponding to $x\in \mathfrak{c}_P$.
\item The multiplicity of $L^\chi_{x,\F}$ in $W^\chi_\F(\mu_y)$
coincides with the multiplicity of $L^P_x$ in $\Delta^P_y$.
\item The classes $[W^\chi_\F(\mu_y)], y\in W^{a,P},$ span $K_0(\U^\chi_{\lambda,\F}\operatorname{-mod}^{\underline{Q}})$.
\item If $x\in {\mathfrak{c}}_P$, then $\sum_{y\in W^{a,P}}c^P_{xy}(1) \mathsf{ch}_{\mathfrak{m}^-}[W^\chi_\F(\mu_y)]=[L^\chi_{x,\F}]$.
\item If $x\not\in {\mathfrak{c}}_P$, then $\sum_{y\in W^{a,P}}c^P_{xy}(1)[W^\chi_\F(\mu_y)] =0$.
\end{enumerate}
\end{Thm}
\begin{proof}
In the proof we can assume that $\lambda=0$ thanks to the translation functors.
Recall  that the restriction of $\Tilt_\F(-\rho)$ to
$\tilde{\g}_{\F}^{(1)\wedge_\chi}$ has the same indecomposable
summands as $\mathcal{V}^\chi_0$, Lemma \ref{Lem:reduct_mod_p}.
It follows that $K_0(\U^\chi_{\lambda,\F}\operatorname{-mod}^{\underline{Q}})$
is identified with $K_0(\A_{e,\F}\operatorname{-mod}^{\underline{Q}})$. Both are
identified with $K_0(\Coh^{\underline{Q}}(\B_e))$ via the category
equivalences $R\Gamma(\mathcal{V}^\chi_0\otimes\bullet),
R\Gamma(\Tilt_\F(-\rho)\otimes\bullet)$. These identifications
are $W^a$-equivariant, where we twist the action on $K_0(\Coh^{\underline{Q}}(\B_e))$
as in Theorem \ref{Prop:parabol_equiv_der}.
Thanks to Proposition \ref{Prop:O_B_global_section},
the class of $R\Gamma(\mathcal{T}(-\rho)|_{\B_{\mathfrak{m}}})$ in
$K_0(\A_{e}\operatorname{-mod}^{\underline{Q}})$ is mapped to
$[W^\chi_\F(2\rho_L-2\rho)]\in K_0(\U^\chi_{\lambda,\F}\operatorname{-mod}^{\underline{Q}})$.
Then we use Corollary \ref{Thm:abelian_quotient} to get a
$W^a$-equivariant surjective map
\begin{equation}\label{eq:Wa_equiv} K_0(\Perv_{I^\circ}(\Fl_P))\twoheadrightarrow K_0(\U^\chi_{\lambda,\F}\operatorname{-mod}^{\underline{Q}}).\end{equation}
It maps $[L^P_x]$ to  a class of a simple object if $x\in {\mathfrak{c}}_P$ or zero else. This proves (1).
Thanks to Lemma \ref{Lem:W_aff_action} and the $W^a$-equivariance
of (\ref{eq:Wa_equiv}), the image of $[\Delta^P_y]$
is $[W^\chi_\F(\mu_y)]$ for all $y\in W^{a,P}$. This observation together with the rest of the proof now implies (2)-(5).
\end{proof}

\subsection{Canonical basis}\label{SS_canonical}
Now we assume that $\chi$ is a general nilpotent element.
The goal of this section is to explain a general approach to computing the multiplicities
of the simple objects in $\U^\chi_{(0),\F}\operatorname{-mod}^{\underline{Q}}$ in the $\chi$-Weyl
modules $W^\chi_\F(\mu_w), w\in W^{a,P}$. This general approach is a ramification of what was
used in \cite{BM}.

We will concentrate on a direct summand of the category $\U^\chi_{(0),\F}\operatorname{-mod}^{{\underline{Q}}}$. Let $\kappa$ be a character of
$T_{0}=\underline{Q}^\circ$.
Let $\U^\chi_{(0),\F}\operatorname{-mod}^{{\underline{Q}},\kappa}$ denote the Serre subcategory
of $\U^\chi_{(0),\F}\operatorname{-mod}^{{\underline{Q}}}$ consisting of all modules $M$
such that $\mathfrak{t}_{0,\F}$ acts on the graded component $M_{\upsilon}$
by $\upsilon+\kappa$ mod $p$ for all characters $\upsilon$ of $T_0$.
So $\U^\chi_{(0),\F}\operatorname{-mod}^{{\underline{Q}},\kappa}$ only depends on
$\kappa$ modulo $p\mathfrak{X}(T_0)$.
Note that $$\U^\chi_{(0),\F}\operatorname{-mod}^{\underline{Q}}=
\bigoplus_{\kappa\in \mathfrak{X}(T_0)/p\mathfrak{X}(T_0)}
\U^\chi_{(0),\F}\operatorname{-mod}^{{\underline{Q}},\kappa}.$$
It is straightforward to see that we have $W^\chi_\F(\mu)\in \U^\chi_{(0),\F}\operatorname{-mod}^{\underline{Q},0}$
if and only if $\mu=x^{-1}\cdot (-2\rho)$ for some $x\in W^{a,P}$.

Also note that we have a category equivalence
$\U^\chi_{(0),\F}\operatorname{-mod}^{\underline{Q},0}\xrightarrow{\sim}
\U^\chi_{(0),\F}\operatorname{-mod}^{\underline{Q},\kappa}$ for all $\kappa\in
\mathfrak{X}(T_0)/ p\mathfrak{X}(T_0)$. An equivalence is given by twisting
with a suitable character of $\underline{Q}$. This works for every $\kappa$
because the restriction map $\mathfrak{X}(\underline{Q})\rightarrow
\mathfrak{X}(T_0)$ gives rise to an isomorphism
$\mathfrak{X}(\underline{Q})/p\mathfrak{X}(\underline{Q})\rightarrow
\mathfrak{X}(T_0)/p\mathfrak{X}(T_0)$.


Consider the completion $\hat{K}_0(\U^\chi_{(0),\F}\operatorname{-mod}^{{\underline{Q}},0})$
of $K_0(\U^\chi_{(0),\F}\operatorname{-mod}^{{\underline{Q}},0})$
consisting of all infinite sums $\sum_{\mathcal{L}} a_{\mathcal{L}} [\mathcal{L}]$, where the summation
is taken over all simple objects in  $\U^\chi_{(0),\F}\operatorname{-mod}^{{\underline{Q}},0}$,
such that, for each $N\in \Z$, only finitely many simples occurring with nonzero coefficient have $\nu$-highest weight
bigger than $N$.

\begin{Lem}\label{Lem:K_0_top_bases}
The following claims are true.
\begin{enumerate}
\item The classes $[W^\chi_\F(x\cdot (-2\rho))]$ for $x\in W^{a,P}$
form a topological generating set for $\hat{K}_0(\U^\chi_{(0),\F}\operatorname{-mod}^{\underline{Q}})$.
\item The classes of simples form a topological basis in
$\hat{K}_0(\U^\chi_{(0),\F}\operatorname{-mod}^{{\underline{Q}},0})$.
\item The classes $[\underline{\Delta}_\nu(\underline{\mathcal{L}})]$
where $\underline{\mathcal{L}}$ runs over the set of simples in the
categories of the form $\underline{\U}^\chi_{w\cdot 0}\operatorname{-mod}^{{\underline{Q}},0}$,
form a topological basis in $\hat{K}_0(\U^\chi_{(0),\F}\operatorname{-mod}^{{\underline{Q}},0})$.
\end{enumerate}
\end{Lem}
\begin{proof}
Part (1) follows from (3) of Theorem \ref{Thm:distinguished_p_dim}, part (2) is obvious,
and part (3) follows from upper triangularity.
\end{proof}

Note that the topological bases in (2) and (3) are both labelled
by $\operatorname{Irr}(\U^\chi_{(0),\F}\operatorname{-mod}^{\underline{Q},0})$: we write
$\underline{\Delta}_\mathcal{L}$ for the unique object of the form $\underline{\Delta}_{\nu}(\underline{\mathcal{L}})$
with an epimorphism onto $\mathcal{L}$.

Recall, Theorem \ref{Thm:Koszulity}, that the algebra $\A_\h|_{S}$ is Koszul. This yields
a positive grading on the specialization $\A_{\h,e}$. The grading can be assumed to be $Q$-stable, see
Remark \ref{Rem:Koszulity}, which gives a graded lift of $\A_{\h,e}\operatorname{-mod}^{{\underline{Q}}}$
to be denoted by $\A_{\h,e}\operatorname{-mod}^{{\underline{Q}},gr}$.  Since the grading on $\A_\h|_{S}$
comes from a $\C^\times$-equivariant structure on $\Tilt_\h|_S$, we get a derived equivalence
\begin{equation}\label{eq:derived_graded_equiv}
D^b(\A_{\h,e}\operatorname{-mod}^{{\underline{Q}},gr})\xrightarrow{\sim} D^b(\Coh^{{\underline{Q}}\times \C^\times}(\B_e)).
\end{equation}

Thanks to Lemma \ref{Lem:reduct_mod_p}, the algebras $(\A_{\h,e})_\F$ and $\U^\chi_{(0),\F}$
are Morita equivalent. Moreover, the images of $(\A_{\h,e})_\F\operatorname{-mod},
\U^\chi_{(0),\F}\operatorname{-mod}$ in $D^b(\Coh(\B_\chi))$
are the same, they coincide with the heart of a t-structure. Here and below we use the equivalences $R\Gamma(\mathcal{T}\otimes\bullet)$ and $R\Gamma(\mathcal{V}_0^\chi(\rho)\otimes\bullet)$.

The algebra
$(\A_{\h,e})_\F$ is acted on by $\underline{Q}^{(1)}$.
Then  $(\A_{\h,e})_\F\operatorname{-mod}^{\underline{Q}^{(1)}}$ is the category of
$\underline{Q}^{(1)}$-equivariant objects in that heart, while
$\U^\chi_{(0),\F}\operatorname{-mod}^{\underline{Q}}$ is identified
with the category of $\underline{Q}$-equivariant objects.
Under this identification,  $\U^\chi_{(0),\F}\operatorname{-mod}^{\underline{Q},0}$
is identified with the category of $\underline{Q}^{(1)}$-equivariant objects
in the heart in $D^b(\Coh(\B_\chi))$.

So we get a category equivalence $(\A_{\h,e})_\F\operatorname{-mod}^{\underline{Q}^{(1)}}\xrightarrow{\sim}
\U^\chi_{(0),\F}\operatorname{-mod}^{{\underline{Q}},0}$.
From this equivalence we  get  a graded lift $\U^\chi_{(0),\F}\operatorname{-mod}^{{\underline{Q}},gr}$
of $\U^\chi_{(0),\F}\operatorname{-mod}^{{\underline{Q}},0}$. We can also assume that
the grading on $\U^\chi_{(0),\F}$ corresponding to the positive grading on
$(\A_{\h,e})_\F$ is also positive.
 We get an equivalence
\begin{equation}\label{eq:graded_equiv_p}
D^b(\U^\chi_{(0),\F}\operatorname{-mod}^{{\underline{Q}},gr})\xrightarrow{\sim}
D^b(\Coh^{{\underline{Q}}^{(1)}\times \F^\times}(\B_\chi))
\end{equation}
intertwining the grading shift functors. The grading shift functor will be denoted by $\langle 1\rangle$.

Recall the contravaiant equivalence $\D: \U^\chi_{(0),\F}\operatorname{-mod}^{\underline{Q}}
\rightarrow \U^\chi_{(0),\F}\operatorname{-mod}^{{\underline{Q}},opp}$.

\begin{Lem}\label{Lem:dual_fix_simples}
We have $\D \mathcal{L}\cong \mathcal{L}$ for all simple objects in $\U^\chi_{(0),\F}\operatorname{-mod}^{\underline{Q},0}$.
\end{Lem}
\begin{proof}
This follows from (1) of Lemma \ref{Lem:K_0_top_bases} combined
with Proposition \ref{Prop:K_0_identity}.
\end{proof}

Now we proceed to discussing graded lifts for the objects $\mathcal{L},\underline{\Delta}_\mathcal{L}$,
and an equivalence $\D$. First of all, the equivalences $\D_{coh}$ and $T_{w_0}^{-1}$ both admit natural
graded lifts to $D^b(\Coh^{{\underline{Q}^{(1)}}\times \F^\times}(\B_\chi))$ to be denoted by
$\tilde{\D}_{coh}$ and $\tilde{T}_{w_0}^{-1}$. For the former this is evident and for the latter
this follows from \cite[Theorem 1.3.2]{BR}.

Note that $\tilde{T}_{w_0}^{-1}$ intertwines
the grading shift functors, while $\tilde{\D}_{coh}$ intertwines $\langle 1\rangle$
with $\langle -1\rangle$. Transferring $(\tilde{T}_{w_0}^{-1}\circ \tilde{\D}_{coh})\langle \ell(w_0)\rangle$
to $D^b(\U^\chi_{(0),\F}\operatorname{-mod}^{{\underline{Q}},gr})$ using (\ref{eq:graded_equiv_p})
we get a self-equivalence
$$\tilde{\D}: D^b(\U^\chi_{(0),\F}\operatorname{-mod}^{{\underline{Q}},gr})\xrightarrow{\sim}
D^b(\U^\chi_{(0),\F}\operatorname{-mod}^{{\underline{Q}},gr})^{opp}$$
that intertwines $\langle 1\rangle$ with $\langle -1\rangle$.
By Proposition \ref{Prop:dual_localization}, $\tilde{\D}$ is a graded
lift of $\D$.

Rescaling the grading, if necessary, we achieve, thanks to Lemma
\ref{Lem:dual_fix_simples}, that each
$\mathcal{L}\in \operatorname{Irr}(\U^\chi_{(0),\F}\operatorname{-mod}^{\underline{Q}})$ admits
a unique graded lift $\tilde{\mathcal{L}}\in \operatorname{Irr}(\U^\chi_{(0),\F}\operatorname{-mod}^{{\underline{Q}},gr})$
such that $\tilde{\D}(\tilde{\mathcal{L}})=\tilde{\mathcal{L}}$.
Note that
$W^{\chi}_\F(2\rho_L-2\rho)$ is simple
so it makes sense to speak about its graded lift
$\tilde{W}^{\chi}_\F(2\rho_L-2\rho)$.

\begin{Lem}\label{Lem:stand_graded_lift}
We have a unique graded lift $\tilde{\underline{\Delta}}_\mathcal{L}$ of
$\underline{\Delta}_\mathcal{L}$ with $\tilde{\underline{\Delta}}_\mathcal{L}\twoheadrightarrow
\tilde{\mathcal{L}}$. The simple constituents of  the kernel
$\tilde{\underline{\Delta}}_\mathcal{L}\twoheadrightarrow
\tilde{\mathcal{L}}$ are of the form $\tilde{\mathcal{L}}'\langle -i\rangle$
for $i>0$.
\end{Lem}
\begin{proof}
Let  $\U^\chi_{(0),\F}\operatorname{-mod}^{{\underline{Q}},0}_{\leqslant i}$ denote
the Serre subcategory of $\U^\chi_{(0),\F}\operatorname{-mod}^{{\underline{Q}},0}$
spanned by all simples $\mathcal{L}$ with $\nu$-highest weights $\leqslant  i$.
Let $\pi_{\leqslant i}:
\U^\chi_{(0),\F}\operatorname{-mod}^{{\underline{Q}},0}_{\leqslant i}
\twoheadrightarrow \operatorname{gr}_i\U^\chi_{(0),\F}\operatorname{-mod}^{{\underline{Q}},0}$
be the quotient functor and let $\pi^!_{\leqslant i}$ be its left adjoint.
Then $\underline{\Delta}_{\mathcal{L}}$ is nothing else but
$\pi_{\leqslant i}^! \pi_{\leqslant i}(\mathcal{L})$. The existence
of $\tilde{\underline{\Delta}}_\mathcal{L}$ follows from here. Since
$\underline{\Delta}_\mathcal{L}$ is indecomposable, we have the uniqueness.
And the claim that the simple constituents of $\tilde{\underline{\Delta}}_\mathcal{L}\twoheadrightarrow
\tilde{\mathcal{L}}$ are of the form $\tilde{\mathcal{L}}'\langle -i\rangle$ follows from $\U^\chi_{(0),\F}$
being positively graded.
\end{proof}

Consider the completed $K_0$-group $\hat{K}_0(\U^\chi_{(0),\F}\operatorname{-mod}^{{\underline{Q}},gr})$
defined similarly to $\hat{K}_0(\U^\chi_{(0),\F}\operatorname{-mod}^{{\underline{Q}},0})$.
This is a topological $\Z[v^{\pm 1}]$-module, where $v=[\langle 1\rangle]$.  Both families $[\tilde{\underline{\Delta}}_\mathcal{L}]$
and $[\tilde{\mathcal{L}}]$ are topological bases of $\hat{K}_0(\U^\chi_{(0),\F}\operatorname{-mod}^{{\underline{Q}},gr})$.

Our goal is to express  the basis $[\tilde{\mathcal{L}}]$
via the classes $H_x [\tilde{W}^\chi_\F(2\rho_L-2\rho)]$ for $x\in W^{a,P}$.
To find the expressions we will proceed  as follows. First, we will study an
action of $\Hecke^a_G$ on $\hat{K}_0(\U^\chi_{(0),\F}\operatorname{-mod}^{{\underline{Q}},gr})$
and explicitly describe the module structure and the classes $[\tilde{\underline{\Delta}}_\mathcal{L}]$,
Section \ref{SS_Hecke_module}. Then we describe the action of $[\tilde{\D}]$ on this module,
Section \ref{SS_D_computation}, this is the main part. In Section
\ref{SS_KL} we recall semi-periodic affine Kazhdan-Lusztig polynomials.
Finally, in Section \ref{SS_final} we fully compute the basis of
simples.

\subsection{Module structure}\label{SS_Hecke_module}
The affine Hecke algebra $\Hecke^a_{G}$  contains $\Hecke^a_{\underline{G}}$ as a subalgebra and  is a free right $\Hecke^a_{\underline{G}}$-module with basis $H_u$, where $u$ runs over the elements
of $W$ with $u$ shortest in $uW_{\underline{G}}$. This subset of $W$ will be
denoted by $W^{\underline{G},-}$.

Recall that $\Hecke^a_G$ gets identified with $K_0^{G\times \C^\times}(\St_\h)$,
Section \ref{SS_affine_Hecke_reminder}.
The action of $D^b(\Coh^{G\times \C^\times}(\St_{\h}))$ on $D^b(\Coh^{{\underline{Q}}\times \C^\times}(\B_e))$ gives rise to a left $\Hecke^a_{G}$-action on
$K_0^{{\underline{Q}}\times\C^\times}(\B_e)=
K_0(\U^\chi_{(0),\F}\operatorname{-mod}^{{\underline{Q}},gr})$.
The class $[\tilde{T}_{x}^{-1}]$ acts by $H_x$.

First, we need a result on a compatibility of dualities.
Let $a\mapsto \bar{a}$ denote the standard bar-involution on $\Hecke^a_G$.

\begin{Prop}\label{Prop:dualities_compat}
For $a\in \Hecke^a_G$ and $m\in K_0(\U^\chi_{(0),\F}\operatorname{-mod}^{{\underline{Q}},gr})$.
Then $[\tilde{\D}](am)=\bar{a}([\tilde{\D}]m)$.
\end{Prop}
\begin{proof}
Recall, Lemma \ref{Lem:bar_coherent}, that
$$\bar{a}=v^{-2\ell(w_0)}H_{w_0}[\tilde{\D}_{coh}](a)H_{w_0}^{-1}.$$
So $[\tilde{\D}](am)=\bar{a}([\tilde{\D}]m)$ boils down to
$$[\tilde{\D}_{coh}](am)=v^{-2\ell(w_0)}\left([\tilde{\D}_{coh}]a\right)\left( [\tilde{\D}_{coh}]m\right).$$
This is \cite[Lemma 9.5]{Lusztig_K1}.
\end{proof}

Our next goal is to describe the abelian group  $K_0(\U^\chi_{(0),\F}\operatorname{-mod}^{{\underline{Q}},gr})$
as an $\Hecke^a_{G}$-module and also compute  the classes $[\tilde{\Delta}_\mathcal{L}]$.

Let $\mathfrak{C}_{\underline{P}}$ be the left cell
module over $\Hecke^a_{\underline{G}}$ corresponding to the left cell ${\mathfrak{c}}_{\underline{P}}$. This left
cell labels the simples in $\underline{\U}^\chi_{(2\underline{\rho}-2\rho),\F}\operatorname{-mod}^{{\underline{Q},0}}$, thanks
to (1) of Theorem \ref{Thm:distinguished_p_dim}. For $x\in {\mathfrak{c}}_{\underline{P}}$
we write $C^{\mathfrak{C}}_x$  for the Kazhdan-Lusztig basis element labeled by $x$ in the
natural $\Z[v^{-1}]$-lattice of $\mathfrak{C}_{\underline{P}}$.

Recall that we have fixed a generic element $\nu\in \mathfrak{X}^+(\underline{G})$.
Note that $$\left(W^a_{\underline{G}}/(W^a_{\underline{G}},W^a_{\underline{G}})\right)
\otimes_{\Z}\Q$$
is naturally identified with $\mathfrak{X}(\underline{G})\otimes_{\Z}\Q$. So it makes sense
to talk about the projection of an element in $W^a_{\underline{G}}$ to $\mathfrak{X}(\underline{G})\otimes_{\Z}\Q$. For $x\in W^a_{\underline{G}}$
we abuse the notation and write $\langle\nu,x\rangle$ for the pairing of $\nu$ with the projection of $x$ to $\mathfrak{X}(\underline{G})\otimes_{\Z}\Q$.
We write
$\hat{\mathfrak{C}}_{\underline{P}}$ for the completion of $\mathfrak{C}_{\underline{P}}$, it consists of
all sums $$\sum a_x C_x^{\mathfrak{C}},$$
with $a_x\in \C[v^{\pm 1}]$ such that, for each $N\in \Z$, there are only finitely many
$x$ satisfying $a_x\neq 0$.

Consider the $\Hecke^a_G$-modules $$\,^G\mathfrak{C}_{\underline{P}}:=
\Hecke^a_G\otimes_{\Hecke^a_{\underline{G}}}\mathfrak{C}_{\underline{P}},
\,^G\hat{\mathfrak{C}}_{\underline{P}}:=
\Hecke^a_G\otimes_{\Hecke^a_{\underline{G}}}\hat{\mathfrak{C}}_{\underline{P}}$$
and the elements $M_{u,x}:=H_uC^{\mathfrak{C}}_x\in \,^G\mathfrak{C}_{\underline{P}}$, where
$u\in W^{\underline{G},-}$ and   $x\in {\mathfrak{c}}_{\underline{P}}$.
Note that the elements $M_{u,x}$
form a basis in $\,^G\mathfrak{C}_{\underline{P}}$ and a topological basis in   $\,^G\hat{\mathfrak{C}}_{\underline{P}}$.


Let $\mathcal{L}$ be a simple object in $\U^\chi_{(0),\F}\operatorname{-mod}^{\underline{Q},0}$ and
$\underline{\mathcal{L}}$ be the simple in $$\bigoplus_{w\in W^{\underline{G},-}}\underline{\U}^\chi_{w\cdot(2\underline{\rho}-2\rho),\F}
\operatorname{-mod}^{{\underline{Q}},0}$$
that labels $\mathcal{L}$.
Suppose $\underline{\mathcal{L}}\in \operatorname{Irr}(\underline{\U}^\chi_{2\underline{\rho}-2\rho,\F}\operatorname{-mod}^{{\underline{Q}},0})$.
Let $x\in {\mathfrak{c}}_{\underline{P}}$ be the Weyl group element labelling
$\underline{\mathcal{L}}$.

\begin{Lem}\label{Lem:K_0_Hecke_module}
There is an  $\Hecke_G^a$-linear  map $\upsilon:\,^G\mathfrak{C}_{\underline{P}}\rightarrow K_0(\U^\chi_{(0),\F}\operatorname{-mod}^{{\underline{Q}},gr})$ such that, in the notation of the previous
paragraph, $\upsilon(M_{1,x})=[\tilde{\underline{\Delta}}_\mathcal{L}]$.
\end{Lem}
\begin{proof}
The proof is in two steps.  Recall that we identify $K_0(\underline{\U}^\chi_{(2\underline{\rho}-2\rho),\F}\operatorname{-mod}^{{\underline{Q}},gr})$
with $K_0^{{\underline{Q}^{(1)}}\times \F^\times}(\underline{\B}_\chi)$ (using the splitting
bundle $\underline{\mathcal{V}}^\chi_{2\underline{\rho}-2\rho}(\rho)$).

{\it Step 1}.
We claim that we have an isomorphism of based $\Hecke^a_{\underline{G}}$-modules
$$\underline{\upsilon}:\mathfrak{C}_{\underline{P}}\xrightarrow{\sim} K_0(\underline{\U}^\chi_{(2\underline{\rho}-2\rho),\F}\operatorname{-mod}^{{\underline{Q}},gr}),$$
where in the source we consider the Kazhdan-Lusztig basis and in the target we consider
the basis of simples $\tilde{\mathcal{L}}$.
First of all, note that $C_{w_{0,P}}$ generates
the $\Hecke_{\underline{G}}^a$-module $\mathfrak{C}_{\underline{P}}$
so there can be at most one $\Hecke_{\underline{G}}^a$-linear map $\underline{\upsilon}$
with
$\underline{\upsilon}(C_{w_{0,P}})=[\tilde{W}^\chi_{\F}(2\rho_L-2\rho)]$.
Let us show that such a module homomorphism exists.

First, we show that we have a homomorphism
\begin{equation}\label{eq:Hecke_homom_induced}
\Hecke_{\underline{G}}^a\otimes_{\Hecke_{W_P}}\operatorname{sgn}_P\rightarrow
K_0(\underline{\U}^\chi_{(2\underline{\rho}-2\rho),\F}\operatorname{-mod}^{{\underline{Q}},gr})\end{equation}
sending $C_{w_{0,P}}$ to $[\tilde{W}^\chi_{\F}(2\rho_L-2\rho)]$.
Note that for a simple reflection $s$ in $W_P$, the functor
$T_s$ homologically shifts $\tilde{W}^\chi_{\F}(2\rho_L-2\rho)$
(this is already the case with the corresponding parabolic Verma module).
So $H_s[\tilde{W}^\chi_{\F}(2\rho_L-2\rho)]$ is of the form
$-v^?[\tilde{W}^\chi_{\F}(2\rho_L-2\rho)]$. We conclude
$(H_s+v)[\tilde{W}^\chi_{\F}(2\rho_L-2\rho)]=0$. This gives
(\ref{eq:Hecke_homom_induced}).

Now we prove that (\ref{eq:Hecke_homom_induced}) factors
through the projection  $\Hecke_{\underline{G}}^a\otimes_{\Hecke_{W_P}}\operatorname{sgn}_P
\twoheadrightarrow \mathfrak{C}_{\underline{P}}$. It follows from the $\C^\times$-equivariant
version of \cite[Theorem 55]{B_Hecke} (proved in the same way as the version there) that the $\Hecke_{\underline{G}}^a$-action on
$K_0^{\underline{Q}^{(1)}\times \F^\times}(\underline{\B}_\chi)\cong K_0^{{\underline{Q}}\times \C^\times}(\underline{\B}_e)$
belongs to the two-sided  cell corresponding to $\underline{G}e$. And
$\mathfrak{C}_{\underline{P}}$ is the largest quotient of
$\Hecke_{\underline{G}}^a\otimes_{\Hecke_{W_P}}\operatorname{sgn}_P$
that belongs to that two-sided cell.   So we get the required homomorphism
$\underline{\upsilon}$.


{\it Step 2}. We map $$K_0(\U^\chi_{(2\underline{\rho}-2\rho),\F}\operatorname{-mod}^{{\underline{Q}},gr})\xrightarrow{\sim}
K_0^{{\underline{Q}}\times \C^\times}(\underline{\B}_e)$$ to
$$K_0(\U^\chi_{(0),\F}\operatorname{-mod}^{{\underline{Q}},gr})\xrightarrow{\sim}K_0^{{\underline{Q}}\times \C^\times}(\B_e)$$
via $\xi_*$. The map $\xi_*$ is easily seen to be a $\Hecke^a_{\underline{G}}$-linear.
So $\xi_*$ induces a $\Hecke^a_G$-linear map $\,^G\mathfrak{C}_{\underline{P}}
\rightarrow K_0(\Coh^{{\underline{Q}}\times \C^\times}(\B_e))$. By Proposition
\ref{Prop:baby_Verma_local} combined with Step 1, it indeed sends $M_{1,x}$
to $[\tilde{\underline{\Delta}}_{\mathcal{L}}]$.
\end{proof}

Let $\mathcal{L}$ be such as in the lemma and let $u\in W^{\underline{G},-}$.
Note that the element $u^{-1}\cdot(2\underline{\rho}-2\rho)$ is dominant for $\underline{G}$.
The categories $\underline{\U}^\chi_{(2\underline{\rho}-2\rho),\F}\operatorname{-mod}^{\underline{Q}}$
and $\underline{\U}^\chi_{u^{-1}\cdot(2\underline{\rho}-2\rho),\F}\operatorname{-mod}^{\underline{Q}}$
are identified by means of the translation functor from $2\underline{\rho}-2\rho$
to $u^{-1}\cdot (2\underline{\rho}-2\rho)$. For $\underline{\mathcal{L}}\in \operatorname{Irr}(\underline{\U}^\chi_{(2\underline{\rho}-2\rho),\F}\operatorname{-mod}^{\underline{Q}})$
we write $\underline{\mathcal{L}}^u$ for the corresponding simple object in
$\underline{\U}^\chi_{u^{-1}\cdot(2\underline{\rho}-2\rho),\F}\operatorname{-mod}^{\underline{Q}}$.

Our next result is as follows.

\begin{Prop}\label{Prop:wc_baby_Verma_graded}
Let $u,\mathcal{L},\mathcal{L}^u$ have the same meaning as above. Then
\begin{equation}\label{eq:wc_baby_Verma_graded} \tilde{T}_u^{-1}\tilde{\underline{\Delta}}_{\mathcal{L}}
\cong \tilde{\underline{\Delta}}_{\mathcal{L}^u}.\end{equation}
\end{Prop}
\begin{proof}
First, we prove the ungraded analog of (\ref{eq:wc_baby_Verma_graded}).
We prove that by induction with respect to the Bruhat order
on $W^{\underline{G},-}$. For $u=1$, there is nothing to prove. Now we assume that
we know that claim for all $u'$ such that $u'\prec u$.
Note that there is a simple reflection $s$ such that $su \prec u, su\in W^{\underline{G},-}$.
So we need to check that $\tilde{T}_s^{-1} \underline{\Delta}_{\mathcal{L}^{su}}\cong
\underline{\Delta}_{\mathcal{L}^u}$.

{\it Step 1}.
Set $\lambda_1:=(su)^{-1}\cdot (2\underline{\rho}-2\rho)$
and $\lambda_2:=u^{-1}\cdot (2\underline{\rho}-2\rho)$. Let us write
$\underline{\U}_{\lambda_1\rightarrow \lambda_2}$ for the translation
$\underline{\U}_{\lambda_2}$-$\underline{\U}_{\lambda_1}$-bimodule (over $\C$),
the image of the regular module $\U_{\lambda_1}$ under the
translation equivalence $\U_{\lambda_1}\operatorname{-mod}
\xrightarrow{\sim}\U_{\lambda_2}\operatorname{-mod}$. We claim that
\begin{equation}\label{eq:parab_refl}T_s^{-1} \Delta^P(\underline{\U}_{\lambda_1})\cong \Delta^P(\underline{\U}_{\lambda_1\rightarrow \lambda_2}).\end{equation}
Recall that $T_s^{-1} \Delta^P(\underline{\U}_{\lambda_1})$ is given by the complex
$$\mathsf{T}^*\mathsf{T} \Delta^P(\underline{\U}_{\lambda_1})\rightarrow \Delta^P(\underline{\U}_{\lambda^1}),$$ where $\mathsf{T}$ is a translation to the wall given by $s$ and the source module is in homological
degree $0$. Since $su \prec u$, we see that $\mathsf{T}^*\mathsf{T} \Delta^P(\underline{\U}_{\lambda_1})=\Delta^P(M)$,
where $M$ is a $\underline{\U}$-bimodule that fits to the exact sequence
$$0\rightarrow \underline{\U}_{\lambda_1\rightarrow \lambda_2}\rightarrow M\rightarrow \underline{\U}_{\lambda_1}
\rightarrow 0.$$
Note that a similar claim is classical for the category $\mathcal{O}$ and our statement
follows from that thanks to the Bernstein-Gelfand equivalence between the category
$\mathcal{O}$ and the category of Harish-Chandra bimodules, both for $\underline{\g}$.

The homomorphism $\Delta^P(M)\rightarrow \Delta^P(\U_{\lambda_1})$ induced by the second arrow in the exact sequence above is surjective, so we
see that (\ref{eq:parab_refl}) indeed holds.

{\it Step 2}. Note that (\ref{eq:parab_refl}) is defined over $\mathbb{Q}$ and hence
can be reduced mod $p$ for $p\gg 0$. Note that
\begin{align*} &T^{-1}_{s}\Delta^P(\underline{\mathcal{L}}^{su})\cong
T^{-1}_s \left(\Delta^P(\underline{\U}_{\lambda_1,\F})\otimes^L_{\underline{\U}_{\lambda_1,\F}}
\underline{\mathcal{L}}^{su}\right)\cong
\left(T^{-1}_s \Delta^P(\underline{\U}_{\lambda_1,\F})\right)\otimes^L_{\underline{\U}_{\lambda_1,\F}}
\underline{\mathcal{L}}^{su}\\
&\cong \Delta^P(\underline{\U}_{\lambda_1\rightarrow \lambda_2,\F})\otimes^L_{\underline{\U}_{\lambda_1,\F}}\underline{\mathcal{L}}^{su}
\cong \Delta^P(\underline{\U}_{\lambda_1\rightarrow \lambda_2,\F}
\otimes^L_{\underline{\U}_{\lambda_1,\F}}\mathcal{L}^{su})\cong \Delta^P (\underline{\mathcal{L}}^{u}).
\end{align*}
Specializing to $0\in (\g_\F^{<0})^{(1)*}$, we get
$T^{-1}_{s}\underline{\Delta}_\nu(\underline{\mathcal{L}}^{su})\cong \underline{\Delta}_\nu (\underline{\mathcal{L}}^{u})$.
This is precisely the ungraded version of (\ref{eq:wc_baby_Verma_graded}).

{\it Step 3}. By Step 2, both  $\tilde{T}_u^{-1} \tilde{\underline{\Delta}}_{\mathcal{L}},
\tilde{\underline{\Delta}}_{\mathcal{L}^u}$ are graded lifts of
$\underline{\Delta}_{\mathcal{L}^u}$. The module $\underline{\Delta}_{\mathcal{L}^u}$ is
indecomposable so its graded lifts differ by a shift of grading. Note that
\begin{itemize}
\item[(i)] the difference of the $\nu$-highest weights of $\mathcal{L}^u$ and $\mathcal{L}$
is small comparing to $p$,
\item[(ii)] and if $u_1\prec u_2$, then
the difference of the $\nu$-highest weights of $\mathcal{L}^{u_2}$
and $\mathcal{L}^{u_1}$ is positive.
\end{itemize}
The element $H_u-\overline{H}_u\in \Hecke_W$ is a linear combination
of $H_{u'}$ for $u'\prec u$.
According to Proposition \ref{Prop:dualities_compat},
$[\tilde{\D}]H_u [\tilde{\underline{\Delta}}_{\mathcal{L}}]=
\overline{H}_u[\tilde{\D}][\tilde{\underline{\Delta}}_{\mathcal{L}}]$.
We see that
$[\tilde{\D}]H_u [\tilde{\underline{\Delta}}_{\mathcal{L}}]-H_u [\tilde{\D}][\tilde{\underline{\Delta}}_{\mathcal{L}}]$
is a linear combination of $H_{u'}[\tilde{\D}][\tilde{\underline{\Delta}}_{\mathcal{L}}]$
with $u'\prec u$. Combining this observation with
(i) and (ii), we see that $H_u[\tilde{\underline{\Delta}}_{\mathcal{L}}]$
is the sum of $[\tilde{\mathcal{L}}^u]$ and the classes with smaller $\nu$-highest
weights. We conclude that the shift of grading from $\tilde{T}_u^{-1} \tilde{\underline{\Delta}}_{\mathcal{L}}$ to $\tilde{\underline{\Delta}}_{\mathcal{L}^u}$
is trivial and hence $\tilde{T}_u^{-1} \tilde{\underline{\Delta}}_{\mathcal{L}}\cong\tilde{\underline{\Delta}}_{\mathcal{L}^u}$.
\end{proof}

\begin{Cor}\label{Cor:Hecke_K_0_action}
Under the linear map
$\,^G\mathfrak{C}_{\underline{P}}\rightarrow K_0(\U^\chi_{(0),\F}\operatorname{-mod}^{{\underline{Q}},gr})$
from Lemma \ref{Lem:K_0_Hecke_module}, the element $M_{u,x}$ is mapped to
$[\tilde{\underline{\Delta}}_{\mathcal{L}^u}]$, where $x$ labels $\mathcal{L}$. In particular,
this map induces an isomorphism
$\,^G\hat{\mathfrak{C}}_{\underline{P}}\rightarrow \hat{K}_0(\U^\chi_{(0),\F}\operatorname{-mod}^{{\underline{Q}},gr})$.
\end{Cor}
%

\subsection{Computation of $[\tilde{\D}]$}\label{SS_D_computation}
In this section  we explain how to compute the involution $[\tilde{\D}]$
of $K_0(\U^\chi_{(0),\F}\operatorname{-mod}^{{\underline{Q}},gr})$.
Note that $[\tilde{\D}]$ is continuous so it extends a semi-linear automorphism
of  $\hat{K}_0(\U^\chi_{(0),\F}\operatorname{-mod}^{{\underline{Q}},gr})$.  Corollary
\ref{Cor:Hecke_K_0_action} provides an identification
$$\hat{K}_0(\U^\chi_{(0),\F}\operatorname{-mod}^{{\underline{Q}},gr})\cong \,^G\hat{\mathfrak{C}}_{\underline{P}}.$$
By the definition, as a topological $\Hecke^a_G$-module, the right hand side
is generated by $\mathfrak{C}_{\underline{P}}$. Thanks to Proposition
\ref{Prop:dualities_compat}, it is enough to compute $[\tilde{\D}]$ on
$\mathfrak{C}_{\underline{P}}$.

Let us state the answer. Let $\underline{w}_{0}$ denote the longest element in $W_{\underline{G}}$.
Set $u_{\underline{G}}:=\underline{w}_{0}^{-1}w_0$. For $\theta\in \mathfrak{X}(P)$, the notation
$\theta\rightarrow +\infty$ means that $\langle\theta,\alpha_i^\vee\rangle\rightarrow +\infty$
for all simple roots $\alpha_i$ of $G$ that are not roots of $\underline{G}$. For
$\theta\in \mathfrak{X}(T)$, set $\theta^*:=-w_0(\theta)$.

\begin{Prop}\label{Prop:D_comput}
For $m\in \mathfrak{C}_{\underline{P}}$, the limit
$$\lim_{\theta\rightarrow +\infty}H_{u_{\underline{G}}} X_{-\theta^*}H_{u_{\underline{G}}}^{-1}X_{-\theta}[\underline{\tilde{\D}}]m$$
exists in $\hat{K}_0^{{\underline{Q}}\times \C^\times}(\B_e)$
and is equal to $[\tilde{\D}]m$.
\end{Prop}

The proof will be given after a series of lemmas. Recall the inclusion
$\xi: \underline{\B}_e\hookrightarrow \B_e$ that is induced by the inclusion $\underline{\mathcal{B}}
\hookrightarrow \mathcal{B}$ via $\underline{\mathfrak{b}}\mapsto \underline{\mathfrak{b}}\oplus
\g^{>0}$. We also have another inclusion,
$\xi^-: \underline{\B}_e\hookrightarrow \B_e$ induced from
$\underline{\mathfrak{b}}\mapsto \underline{\mathfrak{b}}\oplus \g^{<0}$.

Recall also, Section \ref{SS_canonical},  that
$\tilde{\D}: D^b(\Coh^{{\underline{Q}}\times \C^\times}(\B_e))
\xrightarrow{\sim} D^b(\Coh^{{\underline{Q}}\times \C^\times}(\B_e))$
is given by $\tilde{T}_{w_0}^{-1}\tilde{\D}_{coh}\langle \ell(w_0)\rangle$.

\begin{Lem}\label{Lem:D_comput_via_xi_minus}
We have $\tilde{\D}\circ \xi_*\cong \tilde{T}_{u_{\underline{G}}}^{-1}\circ \xi^-_*\circ \underline{\tilde{\D}}\langle \ell(u_{\underline{G}})\rangle$.
\end{Lem}
\begin{proof}
Recall that $\tilde{\D}_{coh}$ is the $\varsigma$-twisted   Serre duality functor.
Since $\xi$ is a closed embedding, the functor $\xi_*$ intertwines the usual Serre duality functors.
Note that $\varsigma\circ \xi\cong \xi^-\circ \underline{\varsigma}$, where $\underline{\varsigma}=\varsigma|_{\ug}$. It follows that $\tilde{\D}_{coh}\circ\xi_*\cong
\xi^-_*\circ \underline{\tilde{\D}}_{coh}\langle \ell(u_{\underline{G}})\rangle$. It remains to show
that $\tilde{T}_{w_0}^{-1}\circ \xi^-_*\cong \tilde{T}_{u_{\underline{G}}}^{-1}\circ \xi^-_*\circ \tilde{T}_{\underline{w}_{0}}^{-1}$,
equivalently $\tilde{T}_{\underline{w}_{0}}^{-1}\circ \xi^-_*\cong \xi^-_*\circ \tilde{T}_{\underline{w}_{0}}^{-1}$.
This is standard from the construction of the elements
$\tilde{T}_i$ for Dynkin roots $\alpha_i$ given in  \cite[Theorem 1.3.2]{BR}.
\end{proof}

To prove Proposition \ref{Prop:D_comput}, it remains to establish the following
formula
\begin{equation}\label{eq:main_limit}
[\xi^-_*]=v^{-\ell(u_{\underline{G}})}\lim_{\theta\rightarrow +\infty} X_{-\theta^*}H^{-1}_{u_{\underline{G}}} X_{-\theta}[\xi_*].
\end{equation}

Note that $\mathfrak{X}(\underline{G})$
naturally acts on the right on $\,^G\hat{\mathfrak{C}}_{\underline{P}}$ by $\Hecke^a_G$-linear automorphisms.
Let $t_\theta$ denote the image of $\theta$ under this action.  We have
\begin{equation}\label{eq:transl_commut}
X_{-\theta}[\xi_*]=[\xi_*]t_{-\theta}.
\end{equation}

In order to prove (\ref{eq:main_limit}) we need to find an alternative presentation
of $\hat{K}_0(\Coh^{\underline{Q}}(\B_e))$ and get some information
on $H_u^{-1} [\xi_*]$ in terms of this presentation. Let $u\in W$ be such that
$u^{-1}\in W^{\underline{G},-}$.
Let us write $\mathfrak{m}^{u}$ for the maximal
nilpotent subalgebra of the parabolic subalgebra $\ug+ u^{-1}(\mathfrak{b})$.
Then we have the embedding $\xi^{u}:\underline{\B}_e\hookrightarrow \B_e$
induced from $\underline{\mathfrak{b}}\mapsto \underline{\mathfrak{b}}\oplus \mathfrak{m}^{u}$.
Note that $\xi=\xi^1$ and $\xi^-=\xi^{u_{\underline{G}}}$. The corresponding embedding
$\underline{\tilde{\g}}\hookrightarrow \tilde{\g}$ will also be denoted by $\xi^{u}$.

\begin{Lem}\label{Lem:equiv_localization}
The map $\bigoplus_{u} [\xi^{u}_*]:\hat{K}_0^{{\underline{Q}}\times \C^\times}(\underline{\B}_e)^{\oplus |W|/|W_{\underline{G}}|}
\rightarrow  \hat{K}_0^{{\underline{Q}}\times \C^\times}(\B_e)$ is an isomorphism.
\end{Lem}
\begin{proof}
The localization theorem in the equivariant $K$-theory applied
to the action of $T_0$ (which is a central subgroup of $\underline{Q}$)
on $\tilde{\g}$ says that
$\bigoplus_{u} [\xi^{u}_*]$ is an isomorphism
\begin{equation}\label{eq:K_0_loc_iso} \bigoplus_ u K_0(\Coh^{{\underline{Q}}\times \C^\times}(\underline{\B}_e))_{loc}
\xrightarrow{\sim}   K_0(\Coh^{{\underline{Q}}\times \C^\times}(\B_e))_{loc},
\end{equation}
where the subscript ``loc'' means that we localize the  classes
\begin{equation}\label{eq:classes}[\Lambda^\bullet N_{\tilde{\g}}\xi_u(\underline{\tilde{\g}})],
\end{equation}
where the notation $N_YX$ means a normal bundle to a smooth subvariety $Y$ in a smooth variety $X$.
The classes (\ref{eq:classes}) in $K_0^{{\underline{Q}}\times \C^\times}(\underline{\B}_e)$
have inverses in $\hat{K}_0^{{\underline{Q}}\times \C^\times}(\underline{\B}_e)$.
It follows that  the isomorphism (\ref{eq:K_0_loc_iso}) extends to an isomorphism between $\hat{K}_0$'s.
\end{proof}

For $m\in \hat{K}^{{\underline{Q}}\times \C^\times}_0(\B_e)$ we set
$$\pr_u(m):=[\xi^u_*]\left(\bigoplus_{u} [\xi^{u}_*]\right)^{-1}m,$$
so that $m=\sum_{u}\pr_u(m)$.

\begin{proof}[Proof of Proposition \ref{Prop:D_comput}]
We need to prove (\ref{eq:main_limit}). The proof is in several steps.

{\it Step 1}.
Consider the direct sum decomposition of $\hat{K}_0^{{\underline{Q}}\times \C^\times}(\B_e)$
from Lemma \ref{Lem:equiv_localization}. We claim that for $m\in \hat{K}_0^{{\underline{Q}}\times \C^\times}(\B_e)$ the limit
$$\lim_{\theta\rightarrow +\infty}X_{-\theta^*}m t_{-\theta}$$
exists and equals to $\pr_{u_{\underline{G}}}(m)$. This is equivalent to showing that
for $m\in \operatorname{im}[\xi^u_{*}]$ we have
$$\lim_{\theta\rightarrow +\infty} X_{-\theta^*}mt_{-\theta}=\begin{cases} m,\text{ if }u=u_{\underline{G}},\\
 0, \text{ else }\end{cases}$$

Recall that $X_{\theta^*}$ is the class of the line bundle $\mathcal{O}(\theta^*)$
on the diagonal in the Steinberg variety $\St_\h$.
We have $(\xi^u)^* \mathcal{O}(-\theta^*)=\underline{\mathcal{O}}(-u\theta^*)$. Note that
for $u=u_{\underline{G}}$, we have $-u\theta^*=\theta$ and hence $\underline{\mathcal{O}}(-u\theta^*)$ is
a trivial bundle with $\underline{Q}$-equivariant structure via $\theta$. It follows
that for $m\in \operatorname{im}[\xi^{u_{\underline{G}}}_{*}]$ we have
$X_{-\theta^*}m t_{-\theta}=m$.

Let $u\neq u_{\underline{G}}$. Set $\underline{m}:=(\bigoplus_{u} [\xi^{u}_*])^{-1}m$.
We need to show that
\begin{equation}\label{eq:limit_aux} X_{-u(\theta^*)}\underline{m} t_{-\theta}\rightarrow 0
\end{equation}
in $\hat{K}_0^{{\underline{Q}}\times \C^\times}(\underline{\B}_e)$. Consider the image of
$X_{-u(\theta^*)}\underline{m} t_{-\theta}$ in $\hat{K}_0(\underline{\U}^\chi_{(2\underline{\rho}-2\rho,\F)}\operatorname{-mod}^{{\underline{Q}},gr})$.
We can assume that the image of $\underline{m}$ there is the class of a simple object, say $\tilde{\mathcal{L}}$.
The $\nu$-highest weight of simples that appear in the cohomology of $(X_{-u(\theta^*)}\tilde{\mathcal{L}})\otimes \F_{-\theta}$ will be of the form $k+\langle -u\theta^*-\theta,\nu\rangle$,
where $k$ is the $\nu$-highest weight of $\tilde{\mathcal{L}}$. We reduce (\ref{eq:limit_aux})
to checking that $\lim_{\theta\rightarrow +\infty}\langle -u\theta^*-\theta,\nu\rangle=-\infty$.
Set $u'=u\underline{w}_{0}^{-1}w_0$ so that $-u\chi^*=u'\chi$ and $u'\in W^{\underline{G},-}$.
Then $u'\theta-\theta$ is a  $\Z_{\leqslant 0}$-linear combination of positive roots.
Note that there must be a root of $\g^{>0}$  appearing with a nonzero coefficient.
As $\theta\rightarrow +\infty$ the minimum of the coefficients in this linear
combination also goes to $-\infty$. It follows that $\lim_{\theta\rightarrow +\infty}\langle -u\theta^*-\theta,\nu\rangle=-\infty$. This finishes the proof of this step.

{\it Step 2}. It remains to prove that
\begin{equation}\label{eq:projections}
v^{-\ell(u_{\underline{G}})}\pr_{u_{\underline{G}}}\left(H_{u_{\underline{G}}}^{-1}[\xi_*]m\right)=[\xi_*^-]m.
\end{equation}
For this we will prove similar statements
for related $K_0$-groups, where it is easier to do computations.

Consider the partial Steinberg variety $\tilde{\g}\times_{\g}\underline{\tilde{\g}}$
(in fact, it is a derived scheme) and its $\underline{G}\times \C^\times$-equivariant
$K$-theory. The convolution map
$$K_0^{\underline{G}\times \C^\times}(\tilde{\g}\times_{\g}\underline{\tilde{\g}})\times
K_0^{{\underline{Q}}\times \C^\times}(\underline{\B}_e)\rightarrow K_0^{{\underline{Q}}\times \C^\times}(\B_e)$$
is continuous so gives rise to
$$\hat{K}_0^{\underline{G}\times \C^\times}(\tilde{\g}\times_{\g}\underline{\tilde{\g}})\times \hat{K}_0^{{\underline{Q}}\times \C^\times}(\underline{\B}_e)\rightarrow \hat{K}_0^{{\underline{Q}}\times \C^\times}(\B_e).$$
This map is $\Hecke^a_G$-linear in the first argument due to associativity of convolution.
Similarly to Lemma \ref{Lem:equiv_localization}, we have the decomposition
$$\bigoplus_{u}[\tilde{\xi}^u_*]:\hat{K}_0^{\underline{G}\times \C^\times}(\underline{\St}_\h)^{\oplus |W|/|W_{\underline{G}}|}\xrightarrow{\sim}\hat{K}_0^{\underline{G}\times \C^\times}(\tilde{\g}\times_{\g}\underline{\tilde{\g}}),$$
where $\underline{\St}_\h$ is the analog of $\St_\h$ for $\ug$ and the meaning of
$\tilde{\xi}^u$ is similar to that of $\xi^u$.

Let $\delta$ denote the class
of diagonal in $\hat{K}_0^{\underline{G}\times \C^\times}(\underline{\St}_\h)$, the unit in this algebra.
Then we have $[\xi^u_*](\bullet)=[\tilde{\xi}^u_*](\delta)*\bullet$. So
it is enough to prove that
\begin{equation}\label{eq:projection1}v^{-\ell(u_{\underline{G}})}\pr_{u_{\underline{G}}}\left(H^{-1}_{u_{\underline{G}}}[\tilde{\xi}_*]
(\delta)\right)=
[\tilde{\xi}^{u_{\underline{G}}}_*](\delta).\end{equation}

{\it Step 3}. Now consider the convolution map
$$\hat{K}_0^{\underline{G}\times \C^\times}(\tilde{\g}\times_{\g}\tilde{\underline{\g}})
\times \hat{K}_0^{T\times \C^\times}(T^*\underline{\B})
\rightarrow \hat{K}_0^{T\times \C^\times}(T^*\B),$$
where the second and the third completions are again with respect
to $\nu$. Let us write $\C_{1\underline{B}}$
for the skyscraper sheaf at the point $1\underline{B}\in \underline{\B}$.
We claim that the map
$$a\mapsto a*[\C_{1\underline{B}}]: \hat{K}_0^{\underline{G}\times \C^\times}(\tilde{\g}\times_{\g}\underline{\tilde{\g}}))
\rightarrow \hat{K}_0^{T\times \C^\times}(T^*\B).$$
is injective. This claim reduces to the case when $G=\underline{G}$ is semisimple
using the localization theorem.
In this case $K_0^{T\times \C^\times}(T^*\B)$ is the periodic
$\Hecke^a_G$-module up to localization, see \cite[Section 10]{Lusztig_K1}.
Our claim translates to the claim that the standard basis element labelled
by $1$ is not annihilated by any nonzero element of $\Hecke^a_G$. It is now immediate.

{\it Step 4}. We write $\C_{wB},w\in W,$ for the skyscraper sheaf at $wB\in T^*\B$
with the trivial $T\times \C^\times$-action. The classes $[\C_{wB}]$  form a basis in the $K_0(\operatorname{Rep}(T\times \C^\times))_{loc}$-module $K_0^{T\times \C^\times}(T^*\B)_{loc}$, where
``loc'' stands for localization functor, where we invert all nonzero
elements in $K_0^{T\times \C^\times}(\operatorname{pt})$. Hence what remains to prove is that
$H_{u_{\underline{G}}}^{-1}[\C_{1B}]=v^{\ell(u_{\underline{G}})}[\C_{(u_{\underline{G}})^{-1}B}]+?$, where $?$ stands for the sum of
the basis elements $1_{wB}$ with $w\prec u_{\underline{G}}$ with some coefficients
in  $K_0^{T\times \C^\times}(pt)_{loc}$. Note that this statement
immediately reduces to the following claim: $H_s^{-1} [\C_{wB}]=v[\C_{swB}]+?[\C_{wB}]$
if $sw>w$. And it is sufficient to verify that statement in the case when $G=\operatorname{SL}_2$
and $w=1$. This is what we assume from now on.

{\it Step 5}.
The character of $\C^\times$ on the cotangent fibers is $v^{-2}$.
The action of $G\times \C^\times$ on $T^*\mathbb{P}^1$ factors through that of
$\operatorname{GL}_2$ on the total space of $\mathcal{O}_{\mathbb{P}^1}(-2)$.
Let $\tilde{T}$ stand for the maximal torus $\{\operatorname{diag}(t_1,t_2)\}
\subset \operatorname{GL}_2$. We also write $t_1,t_2$ for the corresponding equivariant
parameters. Our convention is that $(t_1,t_2)$ acts on the standard homogeneous coordinate
functions $x,y$ on $\mathbb{P}^1$ by $t_1,t_2$, respectively.
The point $1B\in \mathbb{P}^1$ is $[1:0]$ (i.e., $x=1,y=0$) and
the point $sB$ is $[0:1]$. Then $v^{-2}=t_1t_2$.

Consider the element $c=v^{-1}H_{s}+1$.
According to \cite[Section 7.5]{CG}, the element $c$ acts on $K_0(\Coh^{T\times \C^\times}(T^*\mathbb{P}^1))$
as the convolution with
$$[\mathcal{O}_{\mathbb{P}^1}]\boxtimes
\left([\mathcal{O}_{T^*\mathbb{P}^1}]-t_1t_2[\pi^* \Omega_{\mathbb{P}^1}]\right).$$
The class of $\Omega_{\mathbb{P}^1}|_{[1:0]}$ is $t_1^{-2}$.
It follows that $c [\C_{[1:0]}]=[\mathcal{O}_{\mathbb{P}^1}](1-t_2t_1^{-1})$.
On the other hand, in the coordinate chart $(y\neq 0)$, the class
$[\C_{[0:1]}]$ is that of the complex $\mathcal{O}_{\mathbb{A}^1}\rightarrow \mathcal{O}_{\mathbb{A}^1}$,
where the map is the multiplication by $x/y$, of weight $t_1t_2^{-1}$.
It follows that the class  $[\C_{[0:1]}]$ in the coordinate chart is
also  $[\mathcal{O}_{\mathbb{P}^1}](1-t_2t_1^{-1})$.
Therefore $(v^{-1}H_s+1)[\C_{[1:0]}]=[\C_{[0:1]}]+? [\C_{[1:0]}]$.
Hence $H_s^{-1}[\C_{[1:0]}]=v[\C_{[0:1]}]+? [\C_{[1:0]}]$.
This finishes the proof.
\end{proof}

\subsection{Semi-periodic Kazhdan-Lusztig polynomials}\label{SS_KL}
The goal of this section is to define the completed semi-periodic module for $\Hecke_G^a$
with its standard basis, a bar-involution on this module, and the canonical
basis. Then we relate this canonical basis to affine Kazhdan-Lusztig polynomials.
We note that much of this is done in \cite{Stroppel}, but our construction also
allows to relate these semi-periodic polynomials to the multiplicities in
$\U^\chi_{(0),\F}\operatorname{-mod}^{\underline{Q}}$.

Fix a standard Levi subgroup $\underline{G}\subset G$. Let $\mathfrak{X}^+(\underline{G})$ denote the intersection of the positive Weyl chamber with $\mathfrak{X}(\underline{G})$.

We consider the completion
$\hat{\Hecke}^a_{\underline{G}}$ defined as follows: it consists of all infinite sums
$\sum_{x\in W^a_{\underline{G}}}a_xH_x$, where for each $N\in \Z$, the set
$\{x\in W^a_{\underline{G}}| \langle\nu,x\rangle>N, a_x\neq 0\}$ is finite.
Here we abuse the notation and write
$\langle \nu,x\rangle$ for the pairing of $\nu$ and the projection of $x$ to $\mathfrak{X}(\underline{G})\otimes_{\Z}\mathbb{Q}$ (see the discussion after
Proposition \ref{Prop:dualities_compat}).

Then set $\hat{\Hecke}^a_G:=\Hecke^a_G\otimes_{\Hecke_{\underline{G}}^a}\hat{\Hecke}_{\underline{G}}^a$.
This space carries a natural structure of a $\Hecke^a_G$-$\Hecke^a_{\underline{G}}$-bimodule.
It is also a complete topological $\mathbb{Z}[v^{\pm 1}]$-module and the topology is compatible
with the bimodule structure.


Now we proceed to defining a bar-involution on $\hat{\Hecke}^a_G$.
For $\theta\in \mathfrak{X}(\underline{G})$ and $a\in \Hecke^a_G$ we set
\begin{equation}\label{eq:shifted_involution}
\overline{a}^\theta:=\overline{aX_{\theta}}X_{-\theta},
\end{equation}
where in the right hand side $\overline{\bullet}$ denotes the usual
bar-involution on $\Hecke^a_G$. The following lemma describes elementary
properties of $\overline{\bullet}^\theta$.

\begin{Lem}\label{Lem:shifted_involution}
The following claims are true:
\begin{enumerate}
\item $\overline{\bullet}^\theta$ is an involution.
\item We have $\overline{abc}^\theta=\overline{a}\cdot \overline{b}^\theta\overline{c}$ for all
$a,b\in \Hecke^a_G$ and $c\in \Hecke_{W_{\underline{G}}}$.
\item We have $\overline{1}^\theta=H_{u_{\underline{G}}}X_{-\theta^*}H_{u_{\underline{G}}}^{-1}X_{-\theta}$.
\end{enumerate}
\end{Lem}
\begin{proof}
(1) and (2) are straightforward from the definition. Let us prove (3).
We have $\overline{1}^\theta=\overline{X_\theta}X_{-\theta}$. As we recalled in
Section \ref{SS_affine_Hecke_reminder},
$\overline{X_\theta}=H_{w_0}X_{w_0(\theta)}H_{w_0}^{-1}$. But $w_0(\theta)=-\theta^*$,
hence $H_{w_0}X_{-\theta^*}H_{w_0}^{-1}=H_{u_{\underline{G}}}X_{-\theta^*}H_{u_{\underline{G}}}^{-1}$.
This proves (3).
\end{proof}

Here is a stabilization property for the involution $\overline{\bullet}^\theta$.

\begin{Prop}\label{Prop:stabilization_bar}
For all $a\in \Hecke^a_G$, the
limit $\lim_{\theta\rightarrow \infty} \overline{a}^\theta$ exists in
$\hat{\Hecke}^a_G$. 
\end{Prop}
\begin{proof}
Thanks to (2) of Lemma \ref{Lem:shifted_involution} and the continuity of the
product, the only thing we need to check is the existence
of $\lim_{\theta\rightarrow +\infty}\overline{1}^\theta$. For this, we use
the proof of Proposition \ref{Prop:D_comput}. Namely, consider
the completed $K_0$-group $\hat{K}_0^{\underline{G}\times \C^\times}(\tilde{\g}\times_\g \tilde{\ug})$.
We have an inclusion $\Hecke^a_G\hookrightarrow \hat{K}_0^{\underline{G}\times \C^\times}(\tilde{\g}\times_{\g}\tilde{\ug}), a\mapsto a*[\tilde{\xi}_*]\delta$.
This map extends to an isomorphism of topological $\Hecke^a_G$-modules
$\hat{\Hecke}^a_G\xrightarrow{\sim} \hat{K}_0^{\underline{G}\times \C^\times}(\tilde{\g}\times_{\g}\tilde{\ug})$.
The argument of Step 2 of the proof of Proposition \ref{Prop:D_comput} shows
that this homomorphism is a topological isomorphism.
The proof of the subsequent steps of
Proposition \ref{Prop:D_comput} shows that $\lim_{\theta\rightarrow +\infty}
H_{u_{\underline{G}}}X_{-\theta^*}H_{u_{\underline{G}}}^{-1}X_{-\theta}[\tilde{\xi}_*] \delta$ exists in
$\hat{K}_0^{\underline{G}\times \C^\times}(\tilde{\g}\times_{\g}\tilde{\ug})$.
%
\end{proof}

We will write $\overline{a}^\infty$ for the limit.
Now we discuss the canonical basis for $\overline{\bullet}^\infty$.

\begin{Prop}\label{Prop:canon_basis}
The following claims are true:
\begin{itemize}
\item[(0)] For $x\in W^a$, the element $H_x^\infty:=H_{xt_\theta}X_\theta^{-1}$ is independent
of $\theta\in \mathfrak{X}^+(\underline{G})$ as long as $\theta$ is large enough.
The elements $H^\infty_x$ form a topological basis in $\hat{\mathcal{H}}^a_G$.
\item[(1)] The map $\overline{\bullet}^\infty:\mathcal{H}^a_G\rightarrow \hat{\mathcal{H}}^a_G$
extends to $\hat{\mathcal{H}}^a_G$ by continuity. The extension is an involution.
\item[(2)] There is a unique collection of elements
$C^\infty_x\in H^\infty_x+v^{-1}\operatorname{Span}^{conv}_{\Z[v^{-1}]}(H^\infty_y| y\in W^a)$
(where the superscript ``conv'' means that we take all converging sums) such that
$\overline{C^\infty_x}^\infty=C^\infty_x$.
\item[(3)] The coefficient of $H^\infty_y$ in $C^\infty_x$ coincides with
$c_{xt_\theta,yt_\theta}$, where $\theta\in \mathfrak{X}^+(\underline{G})$ is large enough (depending on $x,y$). In particular, $C^\infty_x=\lim_{\theta\rightarrow \infty} C_{xt_\theta}X_{-\theta}$.
\end{itemize}
\end{Prop}
\begin{proof}
The proof is in several steps. We write $\Delta,\Delta_{\underline{G}}$ for the root systems of
$\g,\underline{\g}$. The superscript ``+''means the system of positive roots.

{\it Step 1}. Let $x=wt_\zeta$ for $w\in W,\zeta\in \mathfrak{X}(T)$. Suppose that $\langle \zeta,\alpha^\vee\rangle<0$ and $\alpha>0\Rightarrow \alpha\in \Delta_{\underline{G}}$. Suppose that $w=w'w''$, where $w''\in W_{\underline{G}}$
and $w'\in W^{\underline{G},-}$.
We have
\begin{align*}
&\ell(x)=\sum_{\alpha, w(\alpha)>0}|\langle\zeta,\alpha^\vee\rangle|+
\sum_{\alpha>0, w(\alpha)<0}|1+\langle\zeta,\alpha^\vee\rangle|=
\sum_{\alpha\in \Delta^+\setminus \Delta^+_{\underline{G}}, w(\alpha)>0}\langle\zeta,\alpha^\vee\rangle
+\\
& \sum_{\alpha\in \Delta^+\setminus \Delta^+_{\underline{G}}, w(\alpha)<0}(1+\langle\zeta,\alpha^\vee\rangle)+
\sum_{\alpha\in \Delta_{\underline{G}}^+, w(\alpha)>0}|\langle\zeta,\alpha^\vee\rangle|+
\sum_{\alpha\in \Delta_{\underline{G}}^+, w(\alpha)<0}|1+\langle\zeta,\alpha^\vee\rangle|\\
&=\ell(w')+\sum_{\alpha\in \Delta^+\setminus \Delta^+_{\underline{G}}}\langle\zeta,\alpha^\vee\rangle+
\sum_{\alpha\in \Delta_{\underline{G}}^+, w''(\alpha)>0}|\langle\zeta,\alpha^\vee\rangle|+
\sum_{\alpha\in \Delta_{\underline{G}}^+, w''(\alpha)<0}|1+\langle\zeta,\alpha^\vee\rangle|.
\end{align*}

{\it Step 2}.
Recall that $\preceq$ denotes the Bruhat order on $W^a$. Then we have the following well-defined order $\preceq^\infty$ on $W^a$: $x\preceq^\infty y$ if and only if $xt_\theta\preceq yt_\theta$ for all large enough $\theta\in \mathfrak{X}^+(\underline{G})$. Note that $\hat{\Hecke}^a_G$ is the completion with respect to $\preceq^\infty$.

Fix $x\in W^a$ as in Step 1 and take  $\theta\in \mathfrak{X}^+(\underline{G})$.
Then $H_{t_\theta}=X_\theta$ and $\ell(x)+\ell(t_\theta)=\ell(xt_\theta)$, hence
\begin{equation}\label{eq:standard_shift}
H_x X_{\theta}=H_{xt_\theta}.
\end{equation}

 From here it follows that $H^\infty_x$ is indeed well-defined for all $x\in W^a$.
And these elements form a topological basis in $\hat{\mathcal{H}}^a_G$.
This implies (0).

Also, for $x\in W^a$ as in Step 1 and $\theta\in \mathfrak{X}^+(\underline{G})$,
(\ref{eq:standard_shift}) implies that  $\overline{H}_x^\theta=\overline{H_{xt_\theta}}X_{-\theta}$.
From here it follows that
\begin{equation}\label{eq:involution_basis}
\overline{H^\infty_x}^\infty\in
H^\infty_x+\operatorname{Span}_{\Z[v^{\pm 1}]}(H^\infty_y| y\prec^\infty x).
\end{equation}
Since the elements $H^\infty_x$ form a topological basis in $\hat{\Hecke}^a_G$,
(\ref{eq:involution_basis}) implies that $\overline{\bullet}^{\infty}$ extends
to a continuous map $\hat{\Hecke}^a_G\rightarrow \hat{\Hecke}^a_G$.
(\ref{eq:involution_basis}) also shows the collection of elements $C^\infty_x$ is unique if it exists.

{\it Step 3}.
The existence of the limit of $\overline{\bullet}^\theta$ (Proposition \ref{Prop:stabilization_bar})
can be interpreted as follows. Take $x=wt_\zeta\in W^a$ to be as in Step 1 and $\theta\in \mathfrak{X}^+(\underline{G})$.
Thanks to (\ref{eq:standard_shift}),
Proposition \ref{Prop:stabilization_bar} means that for all $x,y\in W^a$,
and all large enough $\theta\in \mathfrak{X}^+(\underline{G})$ the coefficient of
$H_{yt_\theta}$ in $\overline{H_{xt_\theta}}$ is independent of $\theta$.

{\it Step 4}.
Let us show $c_{xt_\theta, yt_\theta}$ is independent of
$\theta$ as long as $\theta$ is large enough. For an element
$x\in W^a$, let $\Hecke^a_{G,\preceq x}$ denote the span
of all $H_z$ with $z\preceq x$. Consider the quotient
$\Hecke^a_{G,\preceq xt_\theta}/\Hecke^a_{G,\preceq yt_\theta}$.
This quotient has a well-defined involution induced by the usual
bar-involution. Moreover,
such quotients for different $\theta,\theta'\gg 0$ are identified:
via $x t_{\theta}\mapsto x t_{\theta'}$. The previous step shows
that these identifications intertwine the involutions. It follows that
the identifications intertwine the canonical bases for these involutions.
Hence $c_{xt_\theta,yt_\theta}$ is independent of $\theta$
as long as $\theta$ is large enough. We will write $c_{xy}^\infty$
for that coefficient.

Note that, by the construction, we have
$c^\infty_{xt_\theta, yt_\theta}=c^\infty_{x,y}$
for all $\theta\in \mathfrak{X}(\underline{G})$ and $c^\infty_{x,y}\neq 0\Rightarrow
y\preceq^\infty x$.

{\it Step 5}. Set
$$C^\infty_x=\sum_{y\preceq^\infty x} c^\infty_{x,y}H^\infty_y.$$
This sum makes sense in $\hat{\Hecke}^a_G$. Note that $c^\infty_{xx}=1$
and $c^\infty_{xy}\in v^{-1}\Z[v^{-1}], c_{xy}^\infty \neq 0\Rightarrow
y\preceq^\infty x$ by the construction. From here we deduce
that $$C^\infty_x\in H^\infty_x+v^{-1}\operatorname{Span}^{conv}_{\Z[v^{-1}]}(H^\infty_y|
y\prec^\infty x).$$
In particular, the elements $C^\infty_x$ form a topological basis in $\hat{\Hecke}^a_G$.
Step 4 shows that $\overline{C^\infty_x}^\infty=C^\infty_x$.
From here we finally deduce that $\overline{\bullet}^\infty$ is an involution.
This finishes the proof.
\end{proof}

\subsection{The basis of simples}\label{SS_final}
Recall that the $\Hecke^a_G$-module $\Hecke^{a,P}_G$
is embedded into $\Hecke^a_G$ as explained in
Section \ref{SS_affine_Hecke_reminder}. This gives rise to an
embedding $\hat{\Hecke}^{a,P}_G\hookrightarrow \hat{\Hecke}^a_G$.
For $x\in W^{a,P}$, set
$H^{\infty,P}_x=\sum_{u\in W_P}(-v)^{-\ell(u)}H^\infty_{xu}$.
These elements form a topological basis in
$\hat{\Hecke}^{a,P}_{G}$.

\begin{Prop}\label{Prop:canon_basis_Sigma}
The following statements hold:
\begin{enumerate}
\item The elements $C^\infty_x$ with $x\in W^{a,P}$ form a
topological basis in $\hat{\Hecke}^{a,P}_{G}$.
\item The kernel of $\hat{\Hecke}^{a,P}_{G}\twoheadrightarrow \,^G\hat{\mathfrak{C}}_{P}$
is topologically spanned by the elements $C^\infty_x$ with
$x$ of the form $u\underline{x}$ with $u\in W^{\underline{G},-}$
and $\underline{x}\in W^{a,P}_{\underline{G}}\setminus {\mathfrak{c}}_{\underline{P}}$.
\item Let $x=u\underline{x}$ be such that $u\in W^{\underline{G},-}$
and $\underline{x}\in {\mathfrak{c}}_{\underline{P}}$. Let $\mathcal{L}$ be the irreducible
module in $\U^\chi_{(0),\F}\operatorname{-mod}^{\underline{Q}}$ labelled
by $x$. The image of $H_u C_{\underline{x}}$ in $\,^G\hat{\mathfrak{C}}_{P}$
coincides with $[\tilde{\Delta}_{\mathcal{L}}]$, while the image of $C^\infty_{x}$
coincides with $[\tilde{\mathcal{L}}]$.
\end{enumerate}
\end{Prop}
\begin{proof}
We first prove (1). Note that for $x\in W^{a,P}$ and sufficiently large
$\theta\in \mathfrak{X}(\underline{G})$, we have $xt_\theta\in W^{a,P}$.
Using this and (3) of Proposition \ref{Prop:canon_basis}
we deduce that $C^\infty_x\in \hat{\Hecke}^{a,P}_G$ for all
$x\in W^{a,P}$. By Step 5 of the proof of Proposition
\ref{Prop:canon_basis}, we have $C^\infty_x\in H^\infty_x+
\operatorname{Span}(H^\infty_y| y\prec^\infty x)$. It follows
that the change matrix from $H^{\infty,P}_x$ to $C^\infty_x$
(for $x\in W^{a,P}$) is strictly uni-triangular. This implies (1).


Let us prove (2) and (3). Let $\pi$ denote natural projection
$\hat{\Hecke}^{a,P}_G\twoheadrightarrow \,^G\hat{\mathfrak{C}}_{\underline{P}}$.
Recall the isomorphism $\,^G\hat{\mathfrak{C}}_{\underline{P}}\cong
\hat{K}_0^{{\underline{Q}}\times \C^\times}(\B_e)$ from Corollary \ref{Cor:Hecke_K_0_action}.
Thanks to  Proposition \ref{Prop:D_comput}, we see that
$$[\tilde{\D}]\pi(C_{w_{0,P}})=\lim_{\theta\rightarrow +\infty}H_{u_{\underline{G}}} X_{-\theta^*}H_{u_{\underline{G}}}^{-1}X_{-\theta}[\underline{\tilde{\D}}]\pi(C_{w_{0,P}}).$$
In the case when $G=\underline{G}$, the element
$\pi(C_{w_{0,P}})$ corresponds to the class $[\tilde{W}^\chi_\F(2\rho_L-2\rho)]$
hence is fixed by $[\tilde{\D}]$. In general, we get that $\pi(C_{w_{0,P}})$
is fixed by $[\underline{\tilde{\D}}]$.

Applying (3) of Lemma \ref{Lem:shifted_involution} and Proposition \ref{Prop:D_comput},
we see that $$\pi(\overline{C_{w_{0,P}}}^\infty)=[\tilde{\D}]\pi(C_{w_{0,P}}).$$
Now we can combine Proposition \ref{Prop:dualities_compat} with (2) of
Lemma \ref{Lem:shifted_involution} to see that
$\pi$ intertwines $\overline{\bullet}^\infty$ and $[\tilde{\D}]$.

For  $x=u\underline{x}\in W^{a,P}$, the image of $H_uC_{\underline{x}}$ in $\,^G\hat{\mathfrak{C}}_{\underline{P}}$ is zero if
$\underline{x}\not\in {\mathfrak{c}}_{\underline{P}}$ and coincides with $[\tilde{\Delta}_{\mathcal{L}}]$
if $x\in {\mathfrak{c}}_{\underline{P}}$. This is a consequence of Corollary
\ref{Cor:Hecke_K_0_action}.

It remains that show that $\pi(C^\infty_x)$ is the simple labelled by $x$ if $\underline{x}\in \mathfrak{c}_{\underline{P}}$ and $0$ else.
Note that the topological $\Z[v^{-1}]$-spans of $H_u C^\infty_{\underline{x}}$ and of $C^\infty_x$
in $\hat{\Hecke}^{a,P}_G$ coincide.  The image of this $\Z[v^{-1}]$-lattice in $\,^G \hat{\mathfrak{C}}_{\underline{P}}$
coincides with the $\Z[v^{-1}]$-lattice topologically spanned by $[\tilde{\Delta}_{\mathcal{L}}]$'s.
The elements $\pi(C^\infty_{u\underline{x}})$, where $\underline{x}\in {\mathfrak{c}}_{\underline{P}}$,
in $\,^G\hat{\mathfrak{C}}_{\underline{P}}$  satisfy the canonical basis conditions analogous to those of
(1) of Proposition \ref{Prop:canon_basis}. So do the classes $[\tilde{\mathcal{L}}]$
thanks to the self-duality property and Lemma
\ref{Lem:stand_graded_lift}.  For the same reason as in Step 2 of
the proof of Proposition \ref{Prop:canon_basis}, $[\tilde{\mathcal{L}}]=\pi(C^\infty_{u\underline{x}})$.
This shows (3). And if $\underline{x}\not\in {\mathfrak{c}}_{\underline{P}}$, then
$$C^\infty_x\in v^{-1}\operatorname{Span}^{conv}_{\Z[v^{-1}]}([\tilde{\mathcal{L}}]),$$
where the span is taken over all simples $\mathcal{L}$.
 Such an element can only
be self-dual if it is equal to zero. This completes the proof.
\end{proof}

Our next result implies Theorem \ref{Thm_dim_general}.

\begin{Thm}\label{Thm:K0_class_general}
Let $x=u\underline{x}\in W^{a,P}$. Suppose that $\underline{x}\in {\mathfrak{c}}_{\underline{P}}$ and $u\in W^{\underline{G},-}$. Let $L^\chi_{x,\F}$ be the corresponding simple
object in $\U^\chi_{(0),\F}\operatorname{-mod}^{\underline{Q}}$.
We have the following identity in $\hat{K}_0$:
\begin{equation}\label{eq:k0_general1}[L^\chi_{x,\F}]=\sum_{y\in W^{a,P}}c^\infty_{x,y}(1) [W^\chi_{\F}(\mu_y)].
\end{equation}
And if $\underline{x}\not\in {\mathfrak{c}}_{\underline{P}}$, we get
\begin{equation}\label{eq:k0_general2}0=\sum_{y\in W^{a,P}}c^\infty_{x,y}(1) [W^\chi_{\F}(\mu_y)].
\end{equation}
\end{Thm}
\begin{proof}
Theorem \ref{Thm:distinguished_p_dim} expresses the classes of
$[\underline{\Delta}_{\mathcal{L}}]$ via the classes  $[W^\chi_{\F}(\mu_y)]$.
Now (\ref{eq:k0_general1}), (\ref{eq:k0_general2}) follow
from (3) and (2) of Proposition \ref{Prop:canon_basis_Sigma}.
\end{proof}

\subsection{From equivariantly irreducible to usual irreducible}\label{SS_irred}
The goal of this section is to explain how to compute the dimensions of
the irreducible $\U^\chi_{(0),\F}$-modules. Let $V$ be an equivariantly
irreducible module. By Lemma \ref{Lem:irred_reln}, it is completely
reducible and all of its irreducible summands have the same dimension.
Every irreducible $\U^{\chi}_{(0),\F}$-module $U$ occurs in some $V$
so to compute the dimension of $U$ one needs to divide the dimension of
$V$ by the number of irreducible summands.
The computation of this number easily reduces to the case when $\chi$
is distinguished: by taking highest weight spaces for $\nu$.

Before considering the general case, let us discuss a relatively easy case
when $e$ is principal in  $\ug$. Here we consider the category
$\U^\chi_{(0),\F}\operatorname{-mod}^{Z(\underline{G})}$. A more traditional
category to consider would be  $\U^\chi_{(0),\F}\operatorname{-mod}^{Z(\underline{G})^\circ}$
of weight modules over $\U^\chi_{(0),\F}$. Let us compare the simple objects in these
two categories.

\begin{Lem}\label{Lem:irred_principal}
Every irreducible object in $\U^\chi_{(0),\F}\operatorname{-mod}^{Z(\underline{G})}$
remains irreducible in $\U^\chi_{(0),\F}\operatorname{-mod}^{Z(\underline{G})^\circ}$.
\end{Lem}
\begin{proof}
Note that we can replace $G$ with $\underline{G}$ and also consider $\U^\chi_{0,\F}$
instead of $\U^\chi_{(0),\F}$. The element $\chi$ is now principal.
The algebra $\U^\chi_{0,\F}$ is Morita equivalent
to $\F$ and the Morita equivalence is given by any $\chi$-Weyl module,
they are all irreducible. So we reduce to showing that the restriction
of an irreducible $Z(\underline{G}_\F)$-module to $Z(\underline{G}_\F)^\circ$
is irreducible. Since   $Z(\underline{G}_\F)$ is commutative, every irreducible
module is one-dimensional, and our result follows.
\end{proof}

In particular, one can easily recover the multiplicities in $\U^\chi_{(0),\F}\operatorname{-mod}^{Z(\underline{G})^\circ}$
if one knows the multiplicities in
$\U^\chi_{(0),\F}\operatorname{-mod}^{Z(\underline{G})}$
but not vice versa.

It turns out that the former multiplicities are
the coefficients of the canonical basis elements associated
to periodic W-graphs.
Let us elaborate on this. Lusztig, \cite[Section 8]{Lusztig_K2},
identified $K_0^{T_0\times \C^\times}(\B_e)$ with the periodic $\Hecke^a_G$-module
he introduced in \cite{Lusztig_periodic}.
The latter (up to completion) comes with two bases: the standard basis and the canonical
basis. Lusztig checked in \cite[Section 10]{Lusztig_K2}, see, in particular, \cite[Proposition 10.7]{Lusztig_K2}, that the resulting canonical
basis is the canonical basis whose existence in the general case he
conjectured in \cite{Lusztig_K1} (in the $Z(\underline{G})^\circ$-equivariant K-theory). He conjectured in
\cite[9.20(a)]{Lusztig_K2} (see also \cite[Conjecture 13.16]{Lusztig_periodic})
that the coefficients of the transition
matrix from the standard basis to the canonical one are in
$\Z_{\geqslant 0}[v]$ (with our sign conventions).

\begin{Prop}\label{Prop:positive_Lusztig}
The positivity conjecture in the previous paragraph is true.
\end{Prop}
\begin{proof}
First, we give a representation theoretic interpretation of the standard basis
in Lusztig's periodic module: as classes of graded lifts of thick baby Verma modules.

We write $\mathfrak{A}$ for the labelling set of baby Verma modules in $\U^\chi_{(0),\F}\operatorname{-mod}^{T_0}$ and $\tilde{\underline{\Delta}}(\alpha)$
for the canonical graded lift of the baby Verma module labelled by $\alpha$
for $\alpha\in \mathfrak{A}$.
By passing from $Z(\underline{G})$-equivariant to $T_0$-equivariant objects, from Corollary \ref{Cor:Hecke_K_0_action} we deduce
that under the action of $\Hecke^a_G$ the classes of graded lifts of baby Verma modules
in $\U^\chi_{(0),\F}\operatorname{-mod}^{T_0}$ transform as the elements of the standard
basis in the periodic module. The classes of graded lifts  of dual baby Verma
modules, to be denoted by $\tilde{\underline{\nabla}}(\alpha)$, are obtained from those for baby Verma modules by duality. So they transform in
the dual way. We can also  consider the {\it thick} baby Verma modules $\underline{\bf{\Delta}}(\alpha)$ in $\U^\chi_{(0),\F}\operatorname{-mod}^{T_0}$. While usual baby Verma modules are parabolically
induced from the one-dimensional representations of the algebras $\underline{\U}^\chi_{(w\cdot 0),\F}$,
the thick baby Verma modules are similarly induced from  $\underline{\U}^\chi_{(w\cdot 0),\F}$
themselves (note that all these algebras are finite dimensional, commutative, and local). As with the baby Verma modules (Lemma \ref{Lem:stand_graded_lift}),
each thick baby Verma module admits a unique graded lift,
to be denoted by $\tilde{\underline{\bf{\Delta}}}(\alpha)$. Note that in the graded category  we have
$$\dim\operatorname{Ext}^i(\tilde{\underline{\bf{\Delta}}}(\alpha),
\tilde{\underline{\nabla}}(\beta))=\delta_{i,0}\delta_{\alpha,\beta}.$$
It follows that under the action of $\Hecke^a_G$, the classes
$[\tilde{\underline{\bf{\Delta}}}(\alpha)]$ transform as the standard basis
of the periodic module. We identify the $K_0$ group of the exact category of
$\tilde{\underline{\bf{\Delta}}}$-filtered module with the periodic module
by sending the classes $[\tilde{\underline{\bf{\Delta}}}(\alpha)]$ to the standard basis elements preserving the labels.

The projective objects in the graded lift of $\U^\chi_{(0),\F}\operatorname{-mod}^{T_0}$
are $\tilde{\underline{\bf{\Delta}}}$-filtered.  It follows from  \cite[Theorem 5.3.5]{BM}
that the basis of graded lifts of projectives is the canonical basis.
So the coefficients of the transition matrix are in  $\Z_{\geqslant 0}[v]$.
\end{proof}

\begin{Rem} The action of the Kazhdan-Lusztig basis elements indexed by simple reflections on
the periodic module is categorified by graded lifts of reflection functors.
When we apply these graded lifts
to the graded projective objects, we get the finite direct sum of graded projective objects.
So the proof of Proposition \ref{Prop:positive_Lusztig} also implies the finiteness
conjecture from \cite[Introduction]{Lusztig_periodic}.
\end{Rem}


Now we return to the  case  when $\chi$ is a general distinguished element.
Let $V$ be an irreducible object in
$\U^\chi_{(0),\F}\operatorname{-mod}^{\underline{Q}}$ and we view it
as an object in $\U^\chi_{(0),\F}\operatorname{-mod}^{\underline{Q}^\circ}$.
Below we will produce a general recipe to compute the number of irreducible summands of $V$ based on the representation theory of affine Hecke
algebras. Unfortunately, our recipe is not so good for computations.
This method works best when  $A$ is commutative.
The latter is always the case for classical Lie algebras when $G$
is the corresponding linear group. From now on we assume that
$\chi$ is distinguished and $A$ is abelian.

First, let us recall a definition of a centrally extended set with a group action.
Let $Y$ be a finite set together with an action of a finite group $\Gamma$ and $\mathbb{K}$
be an algebraically closed field. By a {\it centrally
extended $\Gamma$-set} structure on $Y$ we mean a $\Gamma$-invariant assignment
$y\mapsto s_y\in H^2(\Gamma_y,\mathbb{K}^\times)$. To a centrally extended
set $Y$ one can assign a category $\mathsf{Sh}^\Gamma(Y)$ of $\Gamma$-equivariant
sheaves of finite dimensional vector spaces on $Y$: for such a sheaf its fiber at
$y$ is a projective representation of $\Gamma$ with Schur multiplier $s_y$.

An example we need is as follows. Take $\mathbb{K}:=\F,Y:=\operatorname{Irr}(\U^\chi_{(0),\F}),
\Gamma:=A$. The structure of a centrally extended $A$-set on $Y$ comes from
the action of $A$ on $\U^\chi_{(0),\F}$ by algebra automorphisms.

The category of semisimple objects in $\U^\chi_{(0),\F}\operatorname{-mod}^A$
is  equivalent to $\mathsf{Sh}^A(Y)$ via the functor $\Hom_{\U}(S,\bullet)$,
where $S$ is the direct sum of all irreducible $\U^\chi_{(0),\F}$-modules,
each with multiplicity $1$. Every irreducible object in $\mathsf{Sh}^A(Y)$ is supported on
a single orbit. Let $V$ be an irreducible object in $\mathsf{Sh}^A(Y)$
and $Ay$ be its support. Its fiber at $y$ is an irreducible projective representation $V_y$ of
$A_y$ with Schur multiplier $s_y$. Then the number of irreducible constituents
of $V$ coincides with $|Ay|\dim V_y$. We remark that this number is a power of 2
when $\g$ is of types B,C,D, or $\g$ is adjoint and $A$ is abelian.
To compute it we need to know $|Ay|$ and also $\dim K_0(\mathsf{Sh}^A(Ay))$, which
coincides with the  number of the irreducible projective $A_y$-modules with  Schur
multiplier $s_y$. We remark that nontrivial Schur multipliers
indeed appear, this can be deduced from \cite[Section 4]{LP}.

Let $J^a_G$ denote Lusztig's asymptotic Hecke algebra
for $W^a_G$. Lusztig has constructed a $\Z[v^{\pm 1}]$-algebra homomorphism  $\Hecke^a_G
\rightarrow J^a_G[v^{\pm 1}]$, \cite{Lusztig_affine}. Let $J^a_{G,\Orb}$ denote the direct summand
of $J^a_G$ corresponding to $\Orb$.

Recall that $\mathfrak{c}_P$ denotes the left
cell containing $w_{0,P}$.

The following claim is known to follow from Theorem \ref{Thm:perv_equiv}.
We provide a proof for readers convenience.

\begin{Prop}\label{Prop:cells}
The following claims hold:
\begin{enumerate}
\item
The $A$-orbits in $Y$ are identified with the left (or right) cells inside the two-sided
cell corresponding to $\Orb$.
\item The based algebras $J^a_{G,\Orb}$ and $K_0(\mathsf{Sh}^A(Y\times Y))$
are identified so that the image of $x\in W^a$ is in $K_0(\mathsf{Sh}^A(Y_\ell\times Y_r))$,
where $Y_{\ell},Y_r$ are $A$-orbits corresponding to the left cells
of $x,x^{-1}$.
\item The identification in (2) restricts to a bijection
between $\mathfrak{c}_P$ and the irreducible objects in
$\mathsf{Sh}^A(Y)$.
\end{enumerate}
\end{Prop}
\begin{proof}
Recall, Theorem \ref{Thm:perv_equiv}, that we have an equivalence of the subquotient categories
$\Perv_{I}(\Fl)_{\Orb}\xrightarrow{\sim} \A\otimes_{\C[\g^*]}\A^{opp}\operatorname{-mod}_{\Orb}$.
This equivalence is monoidal. We can consider the categories of semisimple objects
in these two categories, they are closed with respect to tensor product functors
(the truncated convolution for $\Perv_{I}(\Fl)_{\Orb}$ and the usual tensor product for
$\A\otimes_{\C[\g^*]}\A^{opp}\operatorname{-mod}_{\Orb}$).
This category for  $\A\otimes_{\C[\g^*]}\A^{opp}\operatorname{-mod}_{\Orb}$
is nothing else but $\Sh^A(Y\times Y)$. For $\Perv_{I}(\Fl)_{\Orb}$
we get Lusztig's asymptotic Hecke category $\mathfrak{J}^a_{G,\Orb}$
categorifying $J^a_{G,\Orb}$. So we have a monoidal equivalence
$\mathfrak{J}^a_{G,\Orb}\xrightarrow{\sim}
\Sh^A(Y\times Y)$.

We have naturally defined left and right equivalence relations on the
set of simple objects in any semisimple monoidal category. For $\mathfrak{J}^a_{G,\Orb}$
we recover the usual left and right equivalence relations on elements in a
given two-sided cell. For $\Sh^A(Y\times Y)$, two simple objects are left
(resp., right) equivalent if and only if their left (resp., right)
supports coincide. These observations imply (1) and (2).

To prove (3), we note that the orbit corresponding to $\mathfrak{c}_P$
is a single point: a simple $\U^\chi_{0,\F}$-module corresponding to
$w_{0,P}$ is $W^\chi_\F(0)$, it is $A$-stable.
\end{proof}

For $y\in Y$, we write $\mathfrak{c}_y$ for the right cell
corresponding to $Ay$. By (3) of the previous proposition,
the irreducible objects in
$\mathsf{Sh}^A(Ay)$ are in bijection with $\mathfrak{c}_y\cap \mathfrak{c}_P$.

Now we interpret $|Ay|$.
For a generic number $\alpha\in \C^\times$, we have $\Hecke^a_G|_{v=\alpha}
\twoheadrightarrow J^a_{G,\Orb}$. This homomorphism is surjective because
the target is finite dimensional and semisimple and, as Lusztig proved
in \cite{Lusztig_affine}, the pullbacks of the irreducible representations
of $J^a_{G,\Orb}$ to $\Hecke^a_G|_{v=\alpha}$
are irreducible and pairwise non-isomorphic. The representation of
$\Hecke^a_G|_{v=\alpha}$ at $K^{\C^\times}(\B_e)|_{v=\alpha}$ is pulled back
from the $J^a_{G,\Orb}$-module $K_0(\mathsf{Sh}(Y))$   (here we consider
complexified  $K_0$-groups). Let $e_y$ denote the idempotent in $J^a_{G,\Orb}$
corresponding to $Ay$. Let $\tilde{e}_y$ be its preimage in $\Hecke^a_G$. Then $|Ay|$
coincides with the trace of $\tilde{e}_y$ in $K^{\C^\times}(\B_e)|_{v=\alpha}$.

\subsection{Towards categorification of Theorem \ref{Thm:K0_class_general}}\label{SS_categorif}
Theorem \ref{Thm:distinguished_p_dim} is essentially a $K_0$-manifestation
of results of Section \ref{SS_parab_statements} but Theorem \ref{Thm_dim_general} is
a purely $K$-theoretic statement. One could try to categorify it by producing a
constructible realization of $D^b(\Coh^{\underline{G}}(\tilde{\g}\times^L_\g \underline{\tilde{\Nilp}}))$
or the heart of one of its t-structures, e.g. $\operatorname{Perv}(\A_\h\otimes_{\C[\g^*]}\underline{\A}^{opp}\operatorname{-mod}^{\underline{G}})$,
where $\underline{\A}$ is an analog of $\A$ for $\underline{G}$.

We expect that $\operatorname{Perv}(\A_\h\otimes_{\C[\g^*]}\underline{\A}^{opp}\operatorname{-mod}^{\underline{G}})$
should be equivalent to the category of $I^\circ$-equivariant perverse sheaves on a
``quarter-infinite'' affine flag variety. Morally, this should be a category
of perverse sheaves on $G^\vee((t))/I'$, where $I'$ is a mixed Iwahori subgroup
constructed as follows. Let $\underline{I}$ denote the Iwahori subgroup
of $\underline{G}^\vee((t))$.
We set $I':=\underline{I}\ltimes G^{\vee,>0}((t))$. An issue, however, that
the space $G^\vee((t))/I'$ behaves pretty badly. In the case when $\underline{G}=T$
the issue was circumvented in \cite{ABBGM}, where the right version of
of $\Perv_{I^\circ}(G^\vee((t))/I')$ was constructed. Moreover, it was proved
that the multiplicities in this category are given by the periodic affine Kazhdan-Lusztig
polynomials. And finally, \cite{ABBGM} establishes an equivalence between
$\Perv_{I^\circ}(G^\vee((t))/I')$ and $\A_{\h,0}\operatorname{-mod}^T$, which is what
the category of perverse bimodules becomes in this particular case. It is an interesting question of whether these constructions and results generalize to the case of an arbitrary Levi
subgroup $\underline{G}$.

\end{document}